\documentclass[10pt]{article}

\usepackage{latexsym}
\usepackage{amsfonts}
\usepackage{amsmath}
\usepackage{amsthm}
\usepackage{amssymb}
\usepackage{mathrsfs}
\usepackage{dsfont}    
\usepackage{bbold}     
\usepackage[english]{babel}
\usepackage{caption}
\usepackage{epsfig}
\usepackage{float}
\usepackage{psfrag}
\usepackage{graphicx}
\usepackage{epsfig}
\usepackage{hyperref}

\newtheorem{hypo}{Hypothesis}[section]

\newtheorem{theorem}{Theorem}

\newtheorem{defi}[theorem]{Definition}
\newtheorem{lemma}[theorem]{Lemma}
\newtheorem{proposition}[theorem]{Proposition}
\newtheorem{rmk}[theorem]{Remark}

\newcommand{\zerarcounters}{\setcounter{equation}{0}}

\newcommand{\ZZZ}{\mathds{Z}}
\newcommand{\CCC}{\mathds{C}}
\newcommand{\NNN}{\mathds{N}}

\newcommand{\RRR}{\mathds{R}}
\newcommand{\TTT}{\mathds{T}}
\newcommand{\uno}{\mathds{1}}

\newcommand{\HH}{{\mathcal H}}
\newcommand{\calA}{{\mathcal A}}
\newcommand{\BB}{{\mathcal B}}
\newcommand{\CCCC}{{\mathcal C}}
\newcommand{\DD}{{\mathcal D}}

\newcommand{\calF}{{\mathcal F}}
\newcommand{\calG}{{\mathcal G}}
\newcommand{\calH}{{\mathcal H}}

\newcommand{\calL}{{\mathcal L}}
\newcommand{\MM}{{\mathcal M}}

\newcommand{\calO}{{\mathcal O}}
\newcommand{\calP}{{\mathcal P}}
\newcommand{\calQ}{{\mathcal Q}}
\newcommand{\RR}{{\mathcal R}}
\newcommand{\SSSS}{{\mathcal S}}
\newcommand{\TT}{{\mathcal T}}
\newcommand{\calU}{{\mathcal U}}
\newcommand{\VV}{{\mathcal V}}

\newcommand{\gote}{{\mathfrak e}}

\newcommand{\gots}{{\mathfrak s}}

\newcommand{\gotD}{{\mathfrak D}}

\newcommand{\und}{\underline}

\newcommand{\ol}{\overline}

\newcommand{\Fullbox}{{\rule{2.0mm}{2.0mm}}}

\newcommand{\EP}{\hfill\Fullbox\vspace{0.2cm}}
\newcommand{\prova}{\noindent{\it Proof. }}

\newcommand{\e}{\varepsilon}
\newcommand{\al}{\alpha}
\newcommand{\de}{\delta}
\newcommand{\be}{\beta}
\newcommand{\z}{\zeta}
\newcommand{\n}{\nu}
\newcommand{\m}{\mu}
\newcommand{\x}{\xi}

\newcommand{\ka}{\kappa}
\newcommand{\g}{\gamma}

\newcommand{\h}{\eta}
\newcommand{\la}{\lambda}
\newcommand{\La}{\Lambda}
\newcommand{\f}{\varphi}
\newcommand{\s}{\sigma}

\newcommand{\del}{\partial}

\newcommand{\oo}{\omega}

\newcommand{\ff}{\mathtt{f}}

\newcommand{\CC}{\boldsymbol{\Sigma}}

\newcommand{{\resonance}}{relevant self-energy cluster }

\newcommand{\ii}{{\rm i}}

\def\ins#1#2#3{\vbox to0pt{\kern-#2 \hbox{\kern#1 #3}\vss}\nointerlineskip}

\makeindex
\begin{document}



\title{\bf KAM for quasi-linear forced hamiltonian NLS}

\author{\bf 
R. Feola 
\\
\small
 SISSA, Trieste, rfeola@sissa.it; 
\footnote{
This research was partially supported by the European Research Council under
FP7 ``Hamiltonian PDEs and small divisor problems: a dynamical systems approach'' grant n. 306414-HamPDEs;
partially supported by PRIN 2012 ``Variational and perturbative aspects of nonlinear differential problems''.
}}

%


\date{}

\maketitle

\begin{abstract}
In this paper we prove the existence of quasi-periodic, small-amplitude, solutions for quasi-linear Hamiltonian perturbations 
of the non-linear Schr\"odinger equation on the torus in presence of a quasi-periodic forcing. In particular we prove that such solutions are linearly stable. 
The proof is based on a Nash-Moser implicit function theorem and on
 a reducibility result on the 
 linearized operator
in a neighborhood of zero.
The proof of the reducibility relies on  changes of coordinates such as 
diffeomorphisms of the torus,  pseudo-differential operators and a KAM-reducibility arguments.
Due to the multiplicity of the eigenvalues we obtain a block-diagonalization.
%
\end{abstract}

\tableofcontents

\zerarcounters
\setcounter{equation}{0}
\section{Introduction and Main result}
\label{sec:1}

In the theory of Hamiltonian partial differential equation an important matter is about the existence of quasi-periodic solutions. 
This topic has been widely studied in literature using different approach. The classical results on semi-linear PDE's (where the non-linearity does not contains derivatives), have been obtained using KAM theory, see for instance \cite{W,K1,KP}, ora {\em via} Nash-Moser theory \cite{CW}.
In this paper we  study the existence of \emph{reducible quasi-periodic solutions} for the hamiltonian NLS equation with unbounded perturbations: 
\begin{equation}\label{mega}
iu_{t}=u_{xx}+mu+   \e \ff(\oo t,x,u,u_{x},u_{xx}),
\quad x\in\TTT:=\RRR/2\pi\ZZZ, 
\end{equation}
where $\e>0$ is a small parameter, $m>0$ and the nonlinearity is quasi-periodic in time with
diophantine frequency vector $\oo\in\RRR^{d}$ and $\ff(\f,x,z)$, with $\f\in\TTT^{d}$, $z=(z_{0},z_{1},z_{2})\in\CCC^{3}$
is in $C^{q}(\TTT^{d+1}\times\CCC^{3};\CCC)$ in the real sense (i.e. as function of ${\rm Re}(z)$ and ${\rm Im}(z)$).
Note that our case is \emph{quasi-linear}, i.e. our non-linearity contains space derivatives of order $\de=n$, where $n$ is
the order of the highest derivative appearing in the linear constant coefficients term.
The Hamiltonian non linear Schr\"odinger equation (NLS) is one of the most studied model in the literature.
KAM theory for PDE's was in fact first developed for the semi-linear NLS with Dirichelet boundary conditions. 
Under this assumption it is known that the linearized operator at some approximate solution 
has simple eigenvalues. 
It is also known that 
extending the classical theory to  the circle is not completely trivial
since the linearized operator has 
 multiple eigenvalues.
Moreover in our case we have to deal also with the difficulties arising from
unbounded non-linearities. It turns out that dealing with 
 these two difficulties at the same time requires subtle analysis, already in the case of equation \eqref{mega} with only one derivative.
In order to clarify this point we first discuss
the main ideas needed in order to deal with unbounded non-linearities.

The first result in the case of unbounded perturbation is due to Kuksin in \cite{Ku2} for a class of KdV-type equations, where $\de<n-1$ (\emph{non-critical} unbounded perturbations) and one has simple eigenvalues.
Concerning the NLS equation we mention the results in \cite{ZGY} (reversible case) and in \cite{LY} (hamiltonian case).
These two works are about the NLS in presence of one derivative in the non-linearity, i.e. $\de=n-1$ and with Dirichelet boundary conditions.
In order to deal with this problem (\emph{critical} unbounded perturbations) the authors uses 
 an appropriate generalization of the ideas developed in \cite{Ku2}. The main point is that one has to deal with time-depending
 scalar homological equation, whose solvability is the content of the so called \emph{Kuksin's Lemma}.
 In the case of the circle (double eigenvalues) one would get a time-dependent matricial homological equation. 
   We also mention \cite{BBiP1}-\cite{BBiP2} where a KAM theory is developed to study the case
 of a ''weaker'' dispersion law in the derivative Klein-Gordon equation. 

The ideas used to deal with the case $\de\leq n-1$ do not apply if $\de=n$. 
 For fully non-linear cases the first results are on the existence of periodic solutions,
 see \cite{IPT} on water waves, and \cite{Ba1} and \cite{Ba2} for  Kirkhoff and Benjamin-Ono equations. 
 These results have been obtained by using a Nash-Moser iterative scheme combined 
 with tecniques of pseudo-differential calculus.
The main point is that the linearized operator has the form $\del_{t}+\DD$ where $\DD$ is a differential operator of order $n$ with non-constant coefficients. The breakthrough idea in \cite{IPT} is  conjugate $\DD$ to an operator of the form $\gotD+\RR$ where
 $\gotD$ has constant coefficients and   $\RR$ is a regularizing pseudo-differential operator of order $k$ sufficiently large (i.e. $\del_{x}^{k}\circ\RR$ is bounded).
 For periodic solutions this is enough to invert the linearized operator by Neumann series since $\gotD^{-1} \RR$ is bounded. In the case of quasi-periodic solutions this is not true and substantial new ideas are required.
A very efficient strategy has been developed in a series of papers by Baldi, Berti and Montalto (see \cite{BBM},\cite{BBM11}) mostly on the KdV equation
 which were recently extended to the NLS in \cite{FP}. 
  The aim of  the present paper is to 
  extend the result of \cite{FP} to the case of
  Hamiltonian non-linearities. 
  

We  now briefly describe the general strategy, which is essentially the same adopted in \cite{FP}. 
Here we focus on   the differences 
we have to deal with in order to obtain the result.

\smallskip
\indent  {\bf Nash-Moser scheme.}
The first ingredient is a generalized implicit function theorem with parameters (in our case the frequency $\oo$).
 This  is a well-established iterative scheme which allows to find zeros of a functional provided that one can prove invertibility of its linearization in a neighborhood of the origin. 
 This is fairly standard material and is based on a formal definition of {\em good parameters} where the algorithm runs through. We restate it in 
Proposition \ref{teo4} in order to adapt to our notation.

 \smallskip
 
 \indent{\bf Inversion of the linearized operator.} An efficient way to prove bounds on the inverse of a linear operator is to diagonalize it: the so called reducibility. 
 In our case, the linearized operator at some approximate solution has double eigenvalues, and in addition to this, it is 
 a second order pseudo-differential operator with  
 non constant coefficients. In Section \ref{sec:3ham} and \ref{sec:4ham} we show that it is possible to
obtain a $2\times 2$ block-diagonal reduction of such linearized operator.
This is actually the content of Proposition \ref{teo2ham}.   The proof is divided in two steps:
 \begin{enumerate}
\item Since we are dealing with unbounded non-linearities, before performing diagonalization, we need to apply some changes of variables in order 
to reduce the  operator to a constant coefficients unbounded operator plus a 
smoothing reminder. How to do this is shown in Section \ref{sec:3ham}. 
 The results are detailed in Lemmata \ref{lem:3.88} and \ref{lem:3.9ham}. 
We remark that
this is a common feature of the above-mentioned  literature. 
Indeed a similar result can be founded in Section $3$ in \cite{FP}.
In the present paper there are some differences with respect to the analysis in \cite{FP}.
First of all we need to adapt the changes of coordinates used in Section $3$ of \cite{FP} 
in order to preserves the Hamiltonian structure of the linearized operator. Secondly 
we need to give a better asymptotic expansion 
of the eigenvalues. As one can see in Lemma \ref{lem:3.88} we need to conjugate our linearized operator
to an operator which is diagonal plus a remainder which ``gain'' one derivatives. 
The analysis in \cite{FP} provides only a remainder which is bounded.
This better approximation of the eigenvalues is necessary here in order to impose the non degeneracy condition 
required in Section \ref{sec:4ham}. 
\item The previous step gives  a good understanding of the  eigenvalues of the matrix which we are diagonalizing. Then by imposing the Second Melnikov conditions 
(quantitative bounds on the difference of eigenvalues)
one diagonalizes by a linear KAM-like scheme.  
 Roughly speaking we need to prove the invertibility of an operator of the form
 $L=D+ \e M$ where 
 $D$ is diagonal with respect to the exponential basis, $M$ a bounded operator on the Sobolev space $H^{s}(\TTT)$ and where $\e>0$ is a small parameter. The analysis of Section \ref{sec:3ham} guarantes that 
 $D$ has the form
  \begin{equation}\label{equaquattro}
 D_{j}^{j}(l)=\ii \oo\cdot l+\ii m_{2} j^{2}+\ii m_{1}j, \quad l\in \ZZZ^{d}, \quad j\in \ZZZ, \quad \oo\in \RRR^{d}\;\; d\geq1,
 \end{equation}
 for some positive  constants $m_{2}=1+O(\e)$ and $m_{1}=O(\e)$.
The idea of reducibility is to find a change of coordinates such that $D+\e M$
is conjugated to and operator of the form
 $D_++\e^2 M_+$ where $D_+$ is again diagonal in the exponentials basis.  
 The equation which defines the change of variables is called the {\emph {homological equation} } while the operators $D,D_+$ are called the \emph{normal form}. 
If one defines $A$ as the generator of the \emph{quasi-}identically transformation, the homological equation has the form
\begin{equation}\label{equacinque}
{\rm ad}(D)[A]:=[D, A]=\e M-\e [M],
\end{equation}
where $[M]$ is a suitable linear operator.
It is clear that the eigenvalues of the adjoint operator ${\rm ad}(D)$ involves the differences of the 
eigenvalues of $D$: here one imposes on $\oo$ the so called
the Second Mel'nikov conditions,  
 which are  lower bounds of the form
\begin{equation}\label{equasei}
|\psi_{j,k}(\ell)|:=|\oo\cdot l+m_{2}(j^{2}-k^{2})+m_{1}(j-k)|\geq\frac{\g|j^{2}-k^{2}|}{1+|l|^{\tau}}.
\end{equation}
It turns out that, if $\ell\neq0$ and $j\neq \pm k$, for ``many'' $\oo$ the bounds \eqref{equasei} holds true. 
For $\ell=0$ and $j=k=0$ one has that $\psi_{j,k}(\ell)=0$. 
For $\ell=0$ and $j=-k$ one has that $\psi_{j,k}(\ell)\sim O(\e)$.
In order to prove 
that \eqref{equacinque} has a solution  we need to 
 impose that $[M]_{j}^{k}(\ell)=M_{j}^{k}(\ell)$
with $\ell=0$ and $j=\pm k$. This is why we can get only  a block diagonal reduction to a $2\times2$ block diagonal time independent matrix. This is an important difference w.r.t. \cite{FP}. 
Actually it is know that reducibility arguments are difficult in the case of operator of the form $L=D+\e M$
where  $D$
 has ``multiple'' eigenvalues. For instance we mention \cite{elikuk}, \cite{grepatu}, where
 the authors deal with a problem of multiple eigenvalues 
 in the more difficult case of unbounded multiplicity, but for semi linear equations. 
 Here we just have that the multiplicity 
 is \emph{two}. 
 
 \noindent
 The Second Melnikov conditions which we require are explicitly stated in Proposition \ref{teo2ham}.
 There is another  important difference with respect to \cite{FP}. It is the presence of 
 a correction of order one to the eigenvalues of the linearized operator, i.e. the constant $m_{1}\neq0$
 in \eqref{equaquattro}.
 In \cite{FP} the constant $m_1$ is zero.
 The presence of such correction implies that for $\ell\neq0$ and $j=\pm k$ we need to 
 require  a weaker condition with respect to \eqref{equasei} (see the definition of $\calO_{\infty}^{2\g}$ in \eqref{martina10ham}). This is needed in order to perform the measure estimates.
\end{enumerate}
Once we have diagonalized  the bounds on the inverse follow from bounds on the eigenvalues, see Proposition \ref{inverseofl}. 

\noindent In the Nash-Moser scheme we need to invert the operator linearized at each approximate solution, namely we perform the diagonalization procedure infinitely many times.

\smallskip
{ \bf Measure estimates.} Now we collect all the Melnikov conditions that we have imposed in the previous steps. In order to conclude the proof we need to show that these conditions are fulfilled for a positive measure set of parameters.  The first basic requirement is to prove that we may impose each {\em single} non-resonance condition by only removing a small set of parameters. In our case this is a non trivial problem which we overcome by imposing a non-degeneracy condition (see Hypothesis \ref{hyp3ham}) and by considering vectors $\oo$ as in \eqref{dio}. 
Then we need to show that the union of the resonant sets is still small, this requires proving a  ''summability'' condition. This is the most delicate part of the paper where substantial new ideas are needed,  see  Section 6 for a more 
detailed comparison  with the case of single eigenvalues \cite{FP}.

\noindent
We consider the equation \eqref{mega}
with
diophantine frequency vector
\begin{equation}\label{dio}
\begin{aligned}
&\oo
\in\Lambda:=\left[\frac{1}{2},\frac{3}{2}\right]^{d}\subset\RRR^{d}, 
\;\; |{\oo}\cdot\ell|\geq \frac{\g_{0}}{|\ell|^{\tau_0 }}, \; \forall\; \ell\in\ZZZ^{d}\backslash\{0\}.
\end{aligned}
\end{equation}
For instance one can fix $\tau_0=d+1$.
We are interested in the existence of quasi-periodic solution of (\ref{mega})  in $H^{s}$, for some $s$,
for a positive measure sets of $\oo$ that is a function $\mathtt{u}(t,x)=u(\oo t,x)$ where
$$
u(\f,x) : \TTT^{d}\times\TTT\to\CCC.
$$
 In other words we look for non-trivial 
$(2\pi)^{d+1}-$periodic solutions $u(\f,x)$ of
\begin{equation}\label{mega2ham}
i\oo\cdot\del_{\f}u=u_{xx}+ \mathtt{m} u+\e  \ff(\f,x,u,u_{x},u_{xx})
\end{equation}
in the Sobolev space $H^{s}:=H^{s}(\TTT^{d}\times\TTT; \CCC):=$
\begin{equation}\label{1.1}
\Big\{u(\f,x)=\!\!\!\sum_{(\ell,k)\in\ZZZ^{d}\times\ZZZ}\!\!\!u_{\ell,k}e^{i(\ell\cdot\f+k\cdot x)} :
\|u\|^{2}_{s}:=\sum_{i\in\ZZZ^{d+1}}|u_{i}|^{2}\langle i\rangle^{2s}<+\infty
\Big\}.
\end{equation}
where $s>\gots_{0}:=(d+2)/2>(d+1)/2$, 
$i=(\ell,k)$ and $\langle i\rangle:=\max(|\ell|, |k|,1)$, $|\ell|:=\max\{|\ell_{1}|,\ldots,
|\ell_{d}|\}$. For $s\geq\gots_{0}$ $H^{s}$ is a Banach Algebra and 
$H^{s}(\TTT^{d+1})\hookrightarrow C(\TTT^{d+1})$ continuously. We are moreover interested in 
studying the linear stability of the possible solutions.

We assume that $\ff(\f,x,z)$, with $\f\in\TTT^{d}$, $z=(z_{0},z_{1},z_{2})\in\CCC^{3}$
is such that
\begin{equation*}
 \ff(\f,x,u,u_{x},u_{xx})=f_{1}(\f,x,\x,\h,\x_{x},\h_x,\x_{xx},\h_{xx})+
if_{2}(\f,x,\x,\h,\x_{x},\h_x,\x_{xx},\h_{xx}),
\end{equation*}
 where we set
$u=\x+i\h$, with $\x(\f,x),\h(\f,x)\in H^{s}(\TTT^{d+1};\RRR)$ for some $s\geq0$,
and where
\begin{equation}\label{CONTROLLA}
f_{i}(\f ,x, \x_{0},\h_{0},\x_{1},\h_{1},\x_{2},\h_{2}) : \TTT^{d+1}\times\RRR^{6}\to\RRR, \quad i=1,2.
\end{equation}
for some $q\in\NNN$ large enough.
In this paper we assume moreover the following:

\begin{hypo}\label{hyp2ham}
Assume that $\ff$ is 
such that
\begin{equation}\label{hamham}
\ff(\oo t,x, u, u_{x},u_{xx})=\del_{\bar{z}_{0}}G(\oo t,x, u,u_{x})-\frac{d}{d x}[\del_{\bar{z}_{1}}G(\oo t,x,u,u_{x})]
\end{equation}
with $\del_{\bar{z}_{i}}=\del_{\x_{i}}+i\del_{\h_{i}}$, $i=0,1$, and
\begin{equation}
G(\oo t, x,u, u_{x}):=F(\oo t, x, \x,\h,\x_{x},\h_{x}) : \TTT^{d+1}\times\RRR^{4} \to \RRR, 
\end{equation}
of class $C^{q+1}$. 
\end{hypo}

\begin{hypo}\label{hyp3ham}
Assume that $\ff$ is 
such that
\begin{equation}\label{nondeg}
\frac{1}{(2\pi)^{d+1}}\int_{\TTT^{d+1}}(\del_{\bar{z}_{1}}\ff)(\f,x,0,0,0)d x d\f=\gote\neq0.
\end{equation}
\end{hypo}

\noindent
Hypothesis \ref{hyp3ham} si quite technical and we will see in the following where we need it. On the contrary 
Hypothesis \ref{hyp2ham} is quite natural and it implies that
the equation (\ref{mega}) can be rewritten as an Hamiltonian PDE 
\begin{equation}\label{mega3}
u_{t}=i \del_{\bar{u}}\calH(u), \quad \calH(u)=\int_{\TTT}|u_{x}|^{2}+m|u|^{2}+
\e G(\oo t,x,u,u_{x})
\end{equation}
with respect to the non-degenerate symplectic form
\begin{equation}\label{simplectic}
\Omega(u,v):={\rm Re}\int_{\TTT} i u\bar{v}d x, \quad u,v\in H^{s}(\TTT^{d+1}; \CCC),
\end{equation}
where $\del_{\bar{u}}$ is the $L^{2}-$gradient with respect the complex scalar product. The main result of the paper is the following.

\begin{theorem}\label{teo1}
There exists $s:=s(d,\tau_{0})>0$, $q=q(d)\in\NNN$ such that
for every nonlinearity $\ff\in C^{q}(\TTT^{d+1}\times\RRR^{6};\CCC)$ that satisfies
Hypotheses \ref{hyp2ham} and \ref{hyp3ham} if $\e\leq \e_{0}(s,d)$ small enough, then there exists a Lipschitz map
$$
u(\e,\oo) : [0,\e_{0}]\times \Lambda\to H^{s}(\TTT^{d+1};\CCC)
$$
($\Lambda$ defined in \eqref{dio})
such that, if $\oo\in \CCCC_{\e}\subset \Lambda$, $u(\e,\oo)$ is a solution of (\ref{mega2ham}). 
Moreover,
the set $\CCCC_{\e}\subset\Lambda$ is a Cantor set of asymptotically full Lebesgue measure, i.e.
\begin{equation}\label{asy}
|\CCCC_{\e}|\to 1 \quad  as \quad \e\to0,
\end{equation}
and
 $||u(\e,\oo)||_{s}\to0$ as $\e\to0$. In addiction, $u(\e,\oo)$ is {\rm linearly stable}.
\end{theorem}

\zerarcounters
\section{Functional Setting and scheme of the proof}\label{sec2}

\subsection{Scale of Sobolev spaces}
%
%
\noindent
For a function $f : \Lambda\to E$ where $\Lambda\subset \RRR^{n}$ and $(E,\|\cdot\|_{E})$ is a Banach space
we define
\begin{eqnarray}\label{lnorm2}
{\it sup \phantom{g}norm:} \; \|f\|_{E}^{sup}&\!\!\!\!\!\!:=\!\!\!\!\!\!&\|f\|^{sup}_{E,\Lambda}:=\sup_{\la\in\Lambda}\|f(\oo)\|_{E},\\
 {\it Lipschitz \phantom{g} semi\!-\!norm:} \;
 \|f\|_{E}^{lip}&\!\!\!\!\!\!:=\!\!\!\!\!\!&\|f\|_{E,\Lambda}^{lip}:=
\sup_{\substack{\oo_{1},\oo_{2}\in\Lambda \\ \oo_{1}\neq\oo_{2}}}
\frac{\|f(\oo_{1})-f(\oo_{2})\|_{E}}{|\la_{1}-\la_{2}|}\nonumber
\end{eqnarray}
%
and for $\g>0$ the weighted  Lipschitz norm
\begin{equation}\label{lnorm}
\|f\|_{E,\g}:=\|f\|_{E,\Lambda,\g}:=\|f\|^{sup}_{E}+\g\|f\|_{E}^{lip}.
\end{equation}
In the paper we will work with parameter families of functions in the spaces $H^{s}$ defined in \eqref{1.1}.
Note that the for $s\geq0$ $H^{s}$ is a scale of Banach spaces, i.e.

$$
\forall  s \leq s', \  \ H^{s'} \subseteq H^s \ \ {\rm and} \  \ \|u\|_s \leq \|u \|_{s'} \, , \ \forall u \in H^{s'} \,.
$$
We define
$ \HH := \cap_{s\geq 0} H^s $.
For a function
$u=u(\oo)\in {\rm Lip}(\Lambda,H^s)$
where $\Lambda\subset \RRR^{d}$ we  write $\|f\|_{H^{s},\g}:=\|f\|_{s,\g}$.

\paragraph{Smoothing operators.}
We define the subspaces of trigonometric polynomials
\begin{equation}\label{trig}
H_{{n}}=H_{N_n}:=\big\{u\in L^{2}(\TTT^{d+1}) : u(\f,x):=\sum_{|(\ell,j)|\leq N_{n}}u_{j}(\ell)e^{i(\ell\cdot\f+jx)}\big\}
\end{equation}
where $N_{n}:=N_{0}^{(\frac{3}{2})^{n}}$, and the orthogonal projection
$$
\Pi_{n}:=\Pi_{N_{n}}: L^{2}(\TTT^{d+1})\to H_{n}, \quad \Pi^{\perp}_{n}:=\uno-\Pi_{n}.
$$
We have the following classical result. 
\begin{lemma}\label{approdo}
Fo any $s\geq0$ and $\n\geq0$ there exists a constant $C:=C(s,\n)$ such that
\begin{equation}\label{smoothing}
\begin{aligned}
&\|\Pi_{n}u\|_{s+\n,\g}\leq C N_{n}^{\nu}\|u\|_{s,\g}, \;\; \forall u\in H^{s}, \\
& \|\Pi^{\perp}_{n}u\|_{s}\leq C N_{n}^{-\nu}\|u\|_{s+\nu}, \;\; \forall u\in H^{s+\nu}.
\end{aligned}
\end{equation}
\end{lemma}

\paragraph{Interpolation and Tame estimates.}
The following are results on the properties of algebra, tame product of the norms on the spaces $H^S$ 
introduced above. 

\begin{lemma}\label{A} Let $s_{0}>d/2$. Then
\begin{enumerate}
\item[(i)] {\bf Embedding.} $||u||_{L^{\infty}}\leq C(s_{0})||u||_{s_{0}}$, $\forall \; u\in H^{s_{0}}$.
\item[(ii)] {\bf Algebra.} $||uv||_{s_{0}}\leq C(s_{0})||u||_{s_{0}}||v||_{s_{0}}$, $\forall\; u,v\in H^{s_{0}}$.
\item[(iii)] {\bf Interpolation.} For $0\leq s_{1}\leq s\leq s_{2}$, $s=\la s_{1}+(1-\la)s_{2}$,
\begin{equation}\label{A1}
||u||_{s}\leq ||u||^{\la}_{s_{1}}||u||_{s_{2}}^{1-\la}, \quad \forall\; u\in H^{s_{2}}.
\end{equation}
Let $a,b\geq0$ and $p,q>0$. For all $u\in H^{a+p+q}$ and $v\in H^{b+p+q}$ one has
\begin{equation}\label{A2}
||u||_{a+p}||v||_{b+q}\leq ||u||_{a+p+q}||v||_{b}+||u||_{a}||v||_{b+p+q}.
\end{equation}
Similarly, for the $|u|^{\infty}_{s }:=\sum_{|\al|\leq s}||D^{\al}u||_{L^{\infty}}$ norm, one has
\begin{equation}\label{A3}
|u|^{\infty}_{s }\leq C(s_{1},s_{2})(|u|^{\infty}_{s_{1}})^{\la}(|u|_{s_{2}}^{\infty})^{1-\la}, 
\quad \forall\; u\in W^{s_{2},\infty},
\end{equation}
%
\begin{equation}\label{A4}
|u|^{\infty}_{a+p }|v|^{\infty}_{b+q }\leq C(a,b,p,q)(
|u|^{\infty}_{a+p+q }|v|^{\infty}_{b }+|u|^{\infty}_{a }|v|^{\infty}_{b+p+q }),
\quad \forall\; u\in W^{a+p+q,\infty}, \; v\in W^{b+p+q,\infty}
\end{equation}
\item[(iv)] {\bf Asymmetric tame product.} For $s\geq s_{0}$ one has
\begin{equation}\label{A5}
||uv||_{s}\leq C(s_{0})||u||_{s}||v||_{s_{0}}+C(s)||u||_{s_{0}}||v||_{s}, \quad \forall\; u,v\in H^{s}.
\end{equation}
\end{enumerate}
If $u:=u(\la)$ and $v:=v(\la)$ depend in a lipschitz way on $\la\in\Lambda\subset\RRR^{d}$,
all the previous statements hold also for the norms
$|\cdot|^{\infty}_{s }$, $||\cdot||_{s,\g}$ and  $|\cdot|^{\infty}_{s,\g }$.
\end{lemma}

\noindent
We omit the proof of the Lemmata \ref{approdo} and \ref{A}.
We refer the reader to the Appendix of \cite{FP}.

Along the paper we shall write also
\begin{equation*}
a\leq_{s} b \;\;\; \Leftrightarrow \;\;\; a\leq C(s) b \;\;\; {\rm for \; some \; constant}\;\; C(s)>0.
\end{equation*}
Moreover to indicate unbounded or regularizing spatial differential operator we shall write $O(\del_{x}^{p})$ for some 
$p\in \ZZZ$. More precisely we say that an operator $A$ is $O(\del_{x}^{p})$ if
\begin{equation}\label{pseudo}
A : H_{x}^{s}\to H_{x}^{s-p}, \quad \forall s\geq0.
\end{equation}
Clearly if $p<0$ the operator is regularizing.

\subsection{Hamiltonian structure}\label{odioham}

We introduce the following product spaces:
\begin{equation}\label{spaces}
\begin{aligned}
{\rm H}^{s}&:={\rm H}^{s}(\TTT^{d+1};\RRR)=H^{s}(\TTT^{d+1};\RRR)\times H^{s}(\TTT^{d+1};\RRR),\\
{\bf H}^{s}&:={\bf H}^{s}(\TTT^{d+1};\CCC)=H^{s}(\TTT^{d+1};\CCC)\times H^{s}(\TTT^{d+1};\CCC)\cap \calU,
\end{aligned}
\end{equation}
where
$$
\calU=\{(h^{+},h^{-})\in H^{s}(\TTT^{d+1};\CCC)\times H^{s}(\TTT^{d+1};\CCC)  \; : \; h^{+}=\ol{h^{-}}\}.
$$
There is a one-to-one correspondence between these two spaces given by  
${\rm H}^{s}\ni v=(v^{(1)},v^{(2)})\mapsto w=(u,\bar{u})\in{\bf H}^{s}$ with $u=v^{(1)}+i v^{(2)}$.
To simplify the notation, in the paper we use the same symbol  $v$ to indicate a function $v\in {\rm H}^{s}$
or $v\in{\bf H}^{s}$. We will use different symbols in some cases only to avoid confusion.

We also write ${\rm H}^{s}_{x}$ and ${\bf H}^{s}_{x}$ to denote the phase space of functions
in ${\rm H}^{s}(\TTT ;\RRR)=H^{s}(\TTT^{1};\RRR)\times H^{s}(\TTT^{1};\RRR)$ and 
${\bf H}^{s}(\TTT ;\CCC)=H^{s}(\TTT^{1};\CCC)\times H^{s}(\TTT^{1};\CCC)\cap \calU,
$
On the product spaces  ${\rm H}^{s}$ and ${\bf H}^{s}$ we define, with abuse of notation, the norms
\begin{equation}\label{spaces1}
\begin{aligned}
\|z\|_{{\rm H}^{s}}&:=\max\{\|z^{(i)}\|_{s}\}_{i=1,2}, \quad z=(z^{(1)},z^{(2)})\in{\rm H}^{s},\\
||w||_{{\bf H}^{s}}&:=\|z\|_{H^{s}(\TTT^{d+1};\CCC)}=\|z\|_{s}, \quad w=(z,\bar{z})\in{\bf H}^{s}, \quad
z=z^{(1)}+iz^{(2)}.
\end{aligned}
\end{equation}

For a function 
$u\in {\rm H}^{s}$
if we write
$u=\x+i\h$ 
one has that the equation
(\ref{mega2ham}) reads
\begin{equation}\label{realsyst}
\left\{
\begin{aligned}
\oo\cdot\del_{\f}\x &=\h_{xx}+\mathtt{m}\h+\e f_{2}(\f,x,\x,\h,\x_{x},\h_{x},\x_{xx},\h_{xx}),\\
-\oo\cdot\del_{\f}\h &=\x_{xx}+\mathtt{m}\x+\e f_{1}(\f,x,\x,\h,\x_{x},\h_{x},\x_{xx},\h_{xx}),
\end{aligned}\right.
\end{equation}
where $f_{i}$ for $i=1,2$ are defined in (\ref{CONTROLLA}). 
Equation \eqref{realsyst} is equivalent to 
equation \eqref{mega2ham}. 
Now we analyze its Hamiltonian structure.
Thanks to Hypotesis \ref{hyp2ham} we can write
\begin{equation}\label{realsyst2}
\dot{w}=\chi_{H}(w):=J\nabla H(w), \quad w=(\x,\h) \in {\rm H}^{s},\quad J=\left(\begin{matrix}0&1\\ -1& 0 \end{matrix}\right),
\end{equation}
If we consider the space ${\rm H}^{s}$ endowed with
the symplectic form
\begin{equation}\label{symform}
\tilde{\Omega}(w,v):=\int_{\TTT} w\cdot J v d x=(w,Jv)_{L^{2}\times L^{2}}, \quad \forall\; w,v\in 
{\rm H}^{s}
\end{equation}
where $\cdot$ is the usual $\RRR^{2}$ scalar product, then 
 $\chi_{H}$  is the Hamiltonian vector field generator by the hamiltonian function
\begin{equation}\label{realHam}
\begin{aligned}
&H : {\rm H}^{s}\to \RRR ,\qquad
H(w)=\frac{1}{2}\int_{\TTT}|w_{x}|^{2}+\mathtt{m}|w|^{2}+\e F(\oo t,x, w,w_{x}).
\end{aligned}
\end{equation}
Indeed, for any $w,v\in{\rm H}^{s}$
one has
$$
d H(w)[h]=(\nabla H(w), h )_{L^{2}(\TTT)\times L^{2}(\TTT)}=\tilde{\Omega}(\chi_{H}(u),h), 
$$

With this notation one has
\begin{equation}\label{eq12ham}
\begin{aligned}
f_{1}
:=
-\del_{\x}F+\del_{\x \x_{x}}F \x_{x}+\del_{\h \x_{x}}F\h_{x}+\del_{\x_{x}\x_{x}}F\x_{xx}+
\del_{\x_{x}\h_{x}}F\h_{xx},\\
f_{2}
:=
-\del_{\h}F+\del_{\x \h_{x}}F \x_{x}+\del_{\h \h_{x}}F\h_{x}+\del_{\x_{x}\h_{x}}F\x_{xx}+
\del_{\h_{x}\h_{x}}F\h_{xx},
\end{aligned}
\end{equation}
where all the functions are evaluated in $(\f,x,\x,\h,\x_{x},\h_{x},\x_{xx},\h_{xx})$.
One can check that the \eqref{realsyst} is equivalent to \eqref{mega2ham}. It is sufficient
to multiply by the constant $i$ the first equation and to add or subtract the second one,
one obtains
\begin{equation}\label{14}
\begin{aligned}
i\oo\cdot\del_{\f}u&=i\oo\cdot\del_{\f}\x-\oo\cdot\del_{\f}\h=
u_{xx}+\mathtt{m}u+\e\ff,\\
i\oo\cdot\del_{\f}\bar{u}&=i\oo\cdot\del_{\f}\x+\oo\cdot\del_{\f}\h
=-\bar{u}_{xx}-\mathtt{m}\bar{u}-\e\ol{\ff}
\end{aligned}
\end{equation}
The classical approach is to consider the ``double'' the  NLS in the product space 
$H^{s}(\TTT^{d+1};\CCC)\times H^{s}(\TTT^{d+1};\CCC)$ in the
complex independent variables $(u^{+},u^{-})$. One recovers the equation (\ref{mega2ham})
 by 
studying the system in the subspace $\calU=\{u^{+}=\ol{u^{-}}\}$ (see the (\ref{14})).

On the contrary we prefer to use the real coordinates, because we are working in a differentiable 
structure. To define a differentiable structure on complex variables 
is more 
less natural. 
Anyway, one can see in \cite{FP}
how to deal with this problem. There, the authors find an extension of the vector fields
on the complex plane that is merely differentiable.  The advantage of that approach, is to deal with
a diagonal linear operator. How we will see in the following of this paper, it is not 
necessary to apply the abstract Nash-Moser Theorem proved in  
\cite{FP}.

The phase space for the NLS is ${\rm H}^{1}:=H^{1}(\TTT;\RRR)\times H^{1}(\TTT;\RRR)$.
In general we have the following definitions:
\begin{defi}\label{hamilt}
We say that a time dependent linear vector field $\chi(t) : {\rm H}^{s}\to {\rm H}^{s}$ is \emph{Hamiltonian} if $\chi(t)=J\calA(t)$, where $J$ is defined in (\ref{realsyst2}) and $\calA(t)$
is a real linear operator that is self-adjoint with respect the real scalar product on $L^{2}\times L^{2}$.
The corresponding Hamiltonian has the form
$$
H(u):=\frac{1}{2}(\calA(t)u,u)_{L^{2}\times L^{2}}=\int_{\TTT}\calA(t)u\cdot u d x
$$
Moreover, if $\calA(t)=\calA(\oo t)$ is quasi-periodic in time, then the associated operator
$\oo\cdot\del_{\f}\uno-J\calA(\f)$ is called Hamiltonian.
\end{defi}

\begin{defi}\label{hamilt2}
We say that a map $A : {\rm H}^{1}\to{\rm H}^{1}$ is \emph{symplectic} if the symplectic form
$\tilde{\Omega}$ in (\ref{symform}) is preserved, i.e. 
\begin{equation}\label{hamilt3}
\tilde{\Omega}(Au,Av)=\tilde{\Omega}(u,v), \quad \forall\;\; u,v\in {\rm H}^{1}.
\end{equation}
If one has a family of symplectic maps $A(\f)$, $\forall\;\f\in\TTT^{d}$ then we say that 
the corresponding operator acting on quasi-periodic functions $u(\f,x)$
$$
(Au)(\f,x):=A(\f)u(\f,x),
$$
is symplectic.
\end{defi}

\begin{rmk}\label{hamilt4}
Note that in complex coordinates the phase space is ${\bf H}^{1}:=H^{1}(\TTT;\CCC)\times H^{1}(\TTT;\CCC)$. 
The definitions above are the same by using the symplectic form defined in (\ref{simplectic})
and the complex scalar product on $L^{2}$. 
\end{rmk}
\noindent

Let $w:=(\x,\h)\in {\rm H}^{s}$. We define the functional
\begin{equation}\label{totale}
\calF(\oo t,x, w):=D_{\oo}w+\e g(\oo t,x, w),\;\;\;\; D_{\oo}=\left(\begin{matrix} \oo\cdot\del_{\f} & -\del_{xx}-m \\ 
\del_{xx}+m & \oo\cdot\del_{\f}\end{matrix}\right), 
\end{equation}
where
\begin{equation}\label{13}
g(\oo t,x , w):=\left(\begin{matrix} -f_{2}(\f,x,\x,\h,\x_{x},\h_{x},\x_{xx},\h_{xx}) \\
f_{1}(\f,x,\x,\h,\x_{x},\h_{x},\x_{xx},\h_{xx})
\end{matrix}\right).
\end{equation}

Now it is more convenient to pass to the complex coordinates. In other words we identify an element
$V:=(v^{(1)},v^{(2)})\in {\rm H}^{s}$ with a function
$v:=v^{(1)}+iv^{(2)}\in H^{s}(\TTT^{d+1};\CCC)$.
Consider the linearized operator ${d}_{z}\calF(\oo t ,x,z)$ at some function $z$, and consider the system
\begin{equation}\label{linsyst}
\begin{aligned}
D_{\oo}V+\e { d}_{z}g(\oo t,x , z)V=0, \quad V\in {\rm H}^{s}.
\end{aligned}
\end{equation}
We introduce an invertible linear change of coordinate of the form
\begin{equation}\label{changelinear}
\begin{aligned}
T &: {\rm H}^{s} \to {\rm H}^{s},\\
TV:=&\left(\begin{matrix}\frac{i}{\sqrt{2}} & -\frac{1}{\sqrt{2}} \\ \frac{1}{\sqrt{2}} & -\frac{i}{\sqrt{2}}\end{matrix}\right)\left(\begin{matrix} v^{(1)} \\ v^{(2)} \end{matrix} \right)=\left(\begin{matrix}\frac{i}{\sqrt{2}} v \\
\frac{1}{\sqrt{2}} \bar{v} \end{matrix} \right), \qquad T^{-1}:=\left(\begin{matrix}-\frac{i}{\sqrt{2}} & \frac{1}{\sqrt{2}}\\ -\frac{1}{\sqrt{2}} & \frac{i}{\sqrt{2}}\end{matrix}\right).
\end{aligned}
\end{equation}
%
%

\noindent
We postponed the proof of the following Lemma in the Appendix \ref{techtech}:
\begin{lemma}\label{lemmaccio}
The transformation of coordinates $T$ defined in (\ref{changelinear}) is symplectic.
Moreover, a function $V:=(v^{(1)},v^{(2)})\in {\rm H}^{s}$ is a solution of the system
\begin{equation}\label{lemmaccio1}
d_{z}\calF(\oo t,x,z)V=0,
\end{equation}
if and only if the function
\begin{equation}\label{lemmaccio2}
\left(\begin{matrix}v \\ \bar{v}\end{matrix}\right):=T_{1}^{-1}T V, \quad v\in H^{s}(\TTT^{d+1};\CCC), \quad T_{1}^{-1}:=\left(\begin{matrix}-i\sqrt{2}&0 \\ 0& \sqrt{2} \end{matrix}\right)
\end{equation}
solves the system
\begin{equation}\label{lemmaccio3}
\calL(z) \left(\begin{matrix}v \\ \bar{v}\end{matrix}\right):= T_{1}^{-1}T d_{z}\calF(\oo t,x,z)T^{-1}T_{1} \left(\begin{matrix}v \\ \bar{v}\end{matrix}\right)=0
\end{equation}
In particular the operator $\calL(z) : H^{s}(\TTT^{d+1};\CCC)\times H^{s}(\TTT^{d+1};\CCC)\to
H^{s}(\TTT^{d+1};\CCC)\times H^{s}(\TTT^{d+1};\CCC)$ has the form
\begin{equation}\label{lemmaccio4}
\begin{aligned}
 {\calL}(z)
&= \omega\cdot \partial_\f \uno +i (E+A_2) \del_{xx} + i A_1\del_x +i(mE+ A_0)\,,
\end{aligned}
\end{equation}
where
\begin{equation}\label{lemmaccio44}
E=\begin{pmatrix}1&0\\0&-1\end{pmatrix}, \quad A_{i}=A_{i}(\f,x, z):=
\left(\begin{matrix}a_{i} & b_{i} \\ -\bar{b}_{i} &-\bar{a}_{i}
\end{matrix}\right)
\end{equation}
with for $i=0,1,2$, and $\forall z
\in H^{s}(\TTT^{d+1};\CCC)$,
\begin{equation}\label{5005}
\begin{aligned}
2a_{i}(\f,x)&:=\e(\del_{z_{i}}\ff)(\f,x,z(\f,x),z_{x}(\f,x),z_{xx}(\f,x)),\\
2b_{i}(\f,x)&:=\e(\del_{\bar{z}_{i}}\ff)(\f,x,z(\f,x),z_{x}(\f,x),z_{xx}(\f,x)),
\end{aligned}
\end{equation}
where we denoted $\del_{z_{i}}:=\del_{z^{(1)}_{i}}-i\del_{z^{(2)}_{i}}$ and $\del_{\bar{z}_{i}}:=\del_{z^{(1)}_{i}}+i\del_{z^{(2)}_{i}}$ for $i=0,1,2$.
\end{lemma}

The operator $\calL$ has further property. It is clearly Hamiltonian with respect to the symplectic form in (\ref{simplectic})
 and the corresponding quadratic Hamiltonian 
has the form
\begin{equation}\label{linham}
\begin{aligned}
H(u,\bar{u})&=\int_{\TTT}(1+a_{2})|u_{x}|^{2}+\frac{1}{2}\left[b_{2}\bar{u}_{x}^{2}+\bar{b}_{2}u_{x}^{2}
\right]-\frac{i}{2} {\rm Im}(a_{1})(u_{x}\bar{u}-u\bar{u}_{x}) {\rm d}x\\
&+\int_{\TTT}-m|u|^{2}- {\rm Re}(a_{0})|u|^{2}-\frac{1}{2}(
b_{0}\bar{u}^{2}+\bar{b}_{0}u^{2}){\rm d}x.
\end{aligned}
\end{equation}
Note that the symplectic form $\Omega$ in (\ref{simplectic}) is equivalent to the $2-$form 
$\tilde{\Omega}$ in (\ref{symform}), i.e. given $u=u^{(1)}+iu^{(2)},v=v^{(1)}+iv^{(2)}\in H^{s}(\TTT^{d+1};\CCC)$,
one has
\begin{equation}\label{equival}
\Omega(u,w)={\rm Re}\int_{\TTT}iu\bar{v}d x=\int_{\TTT}(u^{(1)}v^{(2)}-v^{(1)}u^{(2)})d x=
\tilde{\Omega}(U,V),
\end{equation}
where we set $U=(u^{(1)},u^{(2)}), V=(v^{(1)},v^{(2)})\in H^{s}(\TTT^{d+1};\RRR)\times H^{s}(\TTT^{d+1};\RRR)$.
The (\ref{linham}) is the general form of a linear Hamiltonian operator as $\calL$, and,
the coefficients $a_{i}$ in \eqref{lemmaccio44} have the form
%
\begin{equation}\label{eq:1111}
\begin{aligned}
&a_{2}(\f,x)\in\RRR, \qquad\qquad\qquad 
a_{1}(\f,x)=\frac{\rm d}{{\rm d}x}a_{2}(\f,x)+i {\rm Im}(a_{1})(\f,x),\\
&b_{1}(\f,x)=\frac{\rm d}{{\rm d}x}b_{2}(\f,x),\qquad
a_{0}(\f,x)={\rm Re}(a_{0})(\f,x)+\frac{i}{2}\frac{\rm d}{{\rm d}x}{\rm Im}(a_1)(\f,x)
\end{aligned}
\end{equation}

\subsection{Scheme of the proof}
For better understanding, we divide the prof of Theorem \ref{teo1} in several propositions.
The strategy is essentially the same followed in \cite{BBM} e \cite{FP}. It is based on a Nash-Moser iteration. We consider the operator $\calF$ in (\ref{lemmaccio4}), our aim is to show that there exists
a  sequence of functions that converges, in some Sobolev space, to a \emph{solution} of (\ref{mega2ham}).

\begin{defi}[{\bf Good Parameters}]\label{invertibility}
Let $\nu=2$,  $\mu>0$, $N>1$ and set $\ka_{2}:=11\mu+25\nu$.
For any Lipschitz family ${ u}(\oo)\in H_{N}\times H_{N}$  with 
$||{ u}||_{\gots_{0}+\mu,\g}\leq1$,
we define the set of good parameters 
 $\oo\in\Lambda$  as:
\begin{subequations}\label{eq104}
\begin{align}
\calG_N({u }):= &\left\{\oo\in \Lambda \; : \;\; ||\calL^{-1}({u }){h }||_{\gots_{0},\g}\leq C(\gots_{0})\g^{-1}||{h }||_{\gots_{0}+\mu,\g},\right. \label{eq104b}\\
&  \; ||\calL^{-1}({u }){h }||_{s,\g}\leq C(s)\g^{-1}\left(||{h }||_{s+\mu,\g}+
||{u }||_{s+\mu,\g}||{h }||_{\gots_{0},\g}\right), \label{eq104a}\\
 &\forall \gots_{0}\leq s\leq \gots_{0}+\ka_{2}-\mu,\;\;\left.\text{ for all Lipschitz maps ${h }(\oo)$}\right\}.\nonumber
 \end{align}
\end{subequations} 
where $\calL$ is the linearized operator defined in \eqref{lemmaccio4}.
\end{defi}
\noindent
Clearly, Definition \ref{invertibility} depends on  $\mu$ and $ N$. For a better understanding of Definition
\ref{invertibility} we refer the reader to Proposition 1.6 in \cite{FP}. Roughly speaking a set $\calG_{N}$
is the  set of parameters $\oo$ for which some tame estimates hold for the inverse of the linearized operator. The constant $\mu$ represents the loss of regularity due to the presence of the small divisors. 
In Sections \ref{sec:4ham} and \ref{sec:5ham} we will give an explicit formulation of the set $\calG_{N}$.
It will turns out that the set $\calG_{N}$ are the sets of parameters $\oo$ such that
the eigenvalues of the linearized operator $\calL(u)$ satisfy non degeneracy conditions. 
We refer the reader to \eqref{martina10ham} and \eqref{primedimham}
for the explicit definition of such non degeneracy conditions.
In particular in Proposition \ref{measurebruttebrutte} in Section \ref{sec6ham}
we show that measure of these sets of parameters is large.


\begin{proposition}\label{teo4}
Fix $\g\leq\g_{0}, \mu>\tau> d$. There exist $q\in\NNN$, depending only on $\tau,d,\mu$, such that for any nonlinearity
$\ff\in C^{q}$ satisfying Hypotheses \ref{hyp2ham} and \ref{hyp3ham} the following holds. Let $\calF$ be defined in Definition
\ref{totale}, then there exists a small constant $\epsilon_0>0$ such that
for any $\e$ with
 $0<\e\g^{-1}<\epsilon_0$, there exist constants $C_{\star}, N_0\in \NNN$, a sequence of functions ${u}_{n}$ and a sequence of sets $\calG_{n}(\g,\tau,\mu)\equiv\calG_{n}\subseteq\Lambda$ defined inductively as
 $\calG_{0}:=\Lambda$ and $\calG_{n+1}:=\calG_{n}\cap \calG_{N_{n}}(u_{n})$
 such that ${ u}_{n} : \calG_{n}\to{ H}^{0}$, $||{ u}_{n}||_{\gots_{0}+\mu,\g}\leq1$ and 
\begin{equation}\label{teo41}
||{ u}_{n}-{ u}_{n-1}||_{\gots_{0}+\mu,\g}\leq C_{\star}\e \g^{-1} N_{n}^{-\ka}, \;\; \ka:=18+2\mu,
\end{equation}
with $N_{n}:=N_{0}^{(\frac{3}{2})^{n}}$ defined in \eqref{trig}.
Moreover the sequence converges in norm $||\cdot||_{\gots_{0}+\mu,\g}$ to a function $u_{\infty}$
such that
\begin{equation}\label{teo42}
\calF({ u}_{\infty})=0, \quad \forall\; \oo\in\calG_{\infty}:=\cap_{n\geq0}\calG_{n}.
\end{equation}
\end{proposition}
In the Nash-Moser scheme the main point is to invert, with appropriate bounds, $\calF$ linearized at any ${ u}_n$. 
Following the classical Newton scheme we define 
$$
{u}_{n+1}= { u}_n -\Pi_{N_{n+1}}d_{z}\calF^{-1}({ u}_n)\Pi_{N_{n+1}}\calF({ u}_{n}).
$$
In principle we do not know wether this definition is well posed since $d_{z}\calF({ u})$ 
may not be invertible.
To study the invertibility of the linearized operator is a problem substantially different for each equation.
In \cite{FP} the authors work in a reversible contest. Essentially the reversibility condition 
introduced there, guarantees that the linearized operator has simple eigenvalues. Hence
it is natural to try to diagonalized $d_{u}\calF$ in order to invert it.
Here, the situation is different. We have that the eigenvalues are multiple, then the diagonalization procedure is more difficult. 

For the proof of Proposition \ref{teo4} we refer the reader to the general result
proved in \cite{FP}. In that work is proved an abstract existence result based on a Nash-Moser scheme
on a scale of Banach spaces. Such result applies \emph{tame} functionals. In our case the functional $\calF$ satisfies such properties by Lemma \eqref{lemA2}.

The main step of our approach is to prove the invertibility of the linearized operator 
$d_{z}\calF(\oo t,x,z )$, at any $x\in{\rm H}^{s}$.
To do this, we will first prove the following diagonalization result on the operator $\calL$ defined in
 (\ref{lemmaccio4}):

 \begin{proposition}[{\bf Reducibility}]\label{teo2ham}
 Fix $\g\leq\g_{0}$ and $\tau>d$ and consider any $\ff\in C^{q}$ that satisfies Hypotheses \ref{hyp2ham} and \ref{hyp3ham}.
Then there exist
  $\h,q\in\NNN$, depending only on $d$, such that
 for $0\leq \e\leq \e_{0}$ with $\e_{0}$ small enough the following holds. Consider  any subset  $\Lambda_{o}\subseteq\Lambda\subseteq \RRR^{d}$ and  any  Lipschitz families $u(\oo) :\Lambda_{o} \to {\bf H}^{0}$ with $||u||_{\gots_{0}+\h,\g}\leq1$. Consider  the linear operator 
 $\calL : {\bf H}^{s}\to {\bf H}^{s}$ in (\ref{lemmaccio4}) computed at $u$. 
 then for all $\s=\pm 2,j\in \NNN$ 
 there exist Lipschitz map 
 $\Omega_{\s,\und{j}} :  \Lambda \to {\rm Mat}(2\times 2,\CCC)$
 of the form
  \begin{equation}\label{1.2.2bis}
\Omega_{\s,\und{j}}= -i \s ({m}_{2}j^{2} + {m}_{0}) \begin{pmatrix}1&0\\ 0 &1 \end{pmatrix} -i\s|{m}_{1}|j \begin{pmatrix}1&0\\ 0 &-1 \end{pmatrix} +i\s R_{\s,\und{j}},
\end{equation}
 
%
where $R_{\s,\und{j}}$ is a self-adjoint matrix and
\begin{equation}\label{1.2.2tris}
\begin{aligned}
&|{m}_{2}-1|_{\g}+|{m}_{0}-m|_{\g}\leq \e C,\quad 
|R_{j}^{k}|_{\g}\leq \frac{\e C}{\langle j\rangle}, \quad k=\pm j, \; j\in\ZZZ,\\
&\e c\leq |m_{1}|^{sup}\leq \e C, \quad |m_{1}|^{lip}\leq \e^{2}\g^{-1}C.
\end{aligned}
\end{equation}
for any $\s\in\CC$, $j\in \NNN\cup\{0\}$, here and in the following $\CC:=\left\{+1,-1\right\}$.
\noindent
Set
\begin{equation}\label{1.2.2quatuor}
\Omega_{\s,\underline{j}}:=\left(\begin{matrix}\Omega_{\s,j}^{\phantom{g}j} & \Omega_{\s,j}^{-j} \\
\Omega_{\s,-j}^{j} & \Omega_{\s,-j}^{-j} \end{matrix}\right),
\end{equation}
 Define  $\mu_{\s,j}$ and $\mu_{\s,-j}$ to be the eigenvalues  
  of $\Omega_{\s,\und{j}}$ 
 Define $\Lambda_{\infty}^{2\g}(u):=\SSSS_{\infty}^{2\g}(u)\cap\calO_{\infty}^{2\g}(u)$ with
 \begin{equation}\label{martina10ham}
 \begin{aligned}
\SSSS^{2\g}_{\infty}({ u})&:=\left\{\begin{array}{ll}
\oo\in\Lambda_{o} : &|\oo\cdot\ell\!+\!\mu_{\s,{j}}(\oo)-\!\mu_{\s',{j'}}(\oo)|\geq
\frac{2\g|\s j^{2}-\s' j'^{2}|}{\langle\ell\rangle^{\tau}}, \\
&\;\ell\in\ZZZ^{d}, \s,\s'\in\CC, j,j'\in\ZZZ 
\end{array}
\right\},\\
\calO_{\infty}^{2\g}(u)&:=\left\{\begin{array}{ll}
\oo\in\Lambda_{o} : &|{\oo}\cdot\ell+\mu_{\s,j}-
\mu_{\s,k}|\geq\frac{2\g}{\langle\ell\rangle^{\tau}\langle j\rangle},\\
&\ell\in\ZZZ^{d}\backslash\{0\}, j\in\ZZZ, k=\pm j, \s\in\CC
\end{array}
\right\},
\end{aligned}
\end{equation}
then we have:


\noindent
(i) for any $s\in(\gots_{0},q-\h)$, if $||z||_{\gots_{0}+\h}<+\infty$ there exist linear
bounded operators $W_{1},W_{2} : {\bf H}^{s}(\TTT^{d+1})\to{\bf H}^{s}(\TTT^{d+1})$ with
bounded inverse, such that $\calL({ u})$ satisfies
\begin{equation}\label{1.2.2ham}
\calL({\bf u})=W_{1}\calL_{\infty}W_{2}^{-1}, \; \calL_{\infty}=\oo\cdot\del_{\f}\uno+\DD_{\infty}, \;
\DD_{\infty}=diag_{(\s,j)\in\CC\times\ZZZ }\{\Omega_{\s,\underline{j}} \},
\end{equation}

\noindent
(ii) for any $\f\in\TTT^{d}$ one has
\begin{equation}\label{1.2.4ham}
W_{i}(\f), W_{i}^{-1}(\f) : {\bf H}^{s}_{x}\to {\bf H}^{s}_{x}, \quad i=1,2.
\end{equation}
with ${\bf H}^{s}_{x}:=H^{s}(\TTT;\CCC)\times  H^{s}(\TTT;\CCC)\cap \calU$ and such that
\begin{equation}\label{eq:4.9madonna}
||(W_{i}^{\pm1}(\f)-\uno){ h}||_{{\bf H}^{s}_{x}}\leq
\e\g^{-1} C(s)(||{ h}||_{{\bf H}^{s}_{x}}+||{ u}||_{s+\h+\gots_{0}}
||{ h}||_{{\bf H}^{1}_{x}}).
\end{equation}
\end{proposition}

\begin{rmk}\label{ciccio}
Note that function $ h(t)\in {\bf H}_{x}^{s}$  is a solution of the forced NLS 
\begin{equation}\label{1.2.5ham}
\calL(z) h=0
\end{equation}
if and only if the function $v(t):=(v_{1},v_{-1}):=W_{2}^{-1}(\oo t)[h(t)]\in {\bf H}^{s}_{x}$ 
solves the constant coefficients dynamical system
\begin{equation}\label{1.2.6ham}
\left(\begin{matrix} \del_{t}v_{1} \\ \del_{t} v_{-1}\end{matrix}\right) 
+\DD_{\infty}\left(\begin{matrix} v_{1} \\ v_{2} \end{matrix}\right)=\left(\begin{matrix} 0\\ 0 \end{matrix}\right), \quad
\dot{v}_{\s,\underline{j}}=-\Omega_{\s,\underline{j}}v_{\s,\underline{j}},\quad
(\s,j)\in\CC \times\ZZZ,
\end{equation}
where all the eigenvalues of the matrices $\Omega_{\s,\underline{j}}$ are purely imaginary. 
Moreover, since $\ol{\Omega_{\s,{j}}^{j}}=-\Omega_{\s,{j}}^{j}$ and 
$\ol{\Omega_{\s,-j}^{\phantom{g} j}}=-\Omega_{\s,j}^{\phantom{g} j}$
then one has
$$
\frac{d}{dt}(|v_{1,j}(t)|^{2}+|v_{1,-j}(t)|^{2})=0, \qquad |v_{\s,0}(t)|^{2}=constant
$$
and hence
\begin{equation}\label{1.2.7ham}
\begin{aligned}
||v_{1}(t)||^{2}_{H_{x}^{s}}&=\sum_{j\in\ZZZ}|v_{1,j}(t)|^{2}\langle j\rangle^{2s}\\
&=|v_{1,0}(t)|^{2}+
\sum_{j\in\NNN}(|v_{1,j}(t)|^{2}+|v_{1,-j}(t)|^{2})\langle j\rangle^{2s}\\
&=|v_{1,0}(0)|^{2}+\sum_{j\in\NNN}
(|v_{1,j}(0)|^{2}+|v_{1,-j}(0)|^{2}) \langle j\rangle^{2s}=||v_{1}(0)||^{2}_{H_{x}^{s}}.
\end{aligned}
\end{equation}
Eq. \eqref{1.2.7ham} means that the Sobolev norm in the space of functions depending on $x$, is constant in time.
\end{rmk}


Proposition \ref{teo2ham} is fundamental in order to prove Theorem \ref{teo1}. Of course one can try
to invert the linearized operator without diagonalize it. In addiction to this we are not able to completely diagonalize it due to the multiplicity of the eigenvalues. This is one of the main difference with respect to the
reversible case.
Anyway the result in Proposition \ref{teo2ham}
is enough to prove the stability of the possible solution. 
What we obtain is a block-diagonal operator with constant coefficients while in \cite{LY} the authors obtain
a normal form depending on time. Here most of the problems appear because we want to obtain a constant coefficient linear operator.
Another important difference between the case of single eigenvalues and double eigenvalues
stands in the set $\calO_{\infty}^{2\g}$ in \eqref{martina10ham}.
Indeed, as one can see in \eqref{martina10ham}, due to the multiplicity of the eigenvalues, we must impose a very weak non degeneracy condition on the eigenvalues.
Moreover, as we will see in Section   \ref{sec6ham}, the measure estimates in the Hamiltonian  case are more difficult with respect to the reversible one, and most of the problems
appear due to the presence of the set $\calO_{\infty}^{2\g}$. In order to overcame such problems we will use the additional Hypotheses \ref{hyp3ham}.
In Section \ref{sec:3ham} we will conjugate $\calL$ to a differential linear operator 
with constant coefficients plus a \emph{bounded} remainder,
 then, in Section \ref{sec:4ham} 
we complete block-diagonalize the operator.

%
Using the reducibility results of Proposition \ref{teo2ham} we are able to  prove (see Section \ref{sec:5ham}) the following result:

\begin{lemma}\label{inverseofl}({\bf Right inverse of $\calL$})
Under the hypotheses of Proposition \ref{teo2ham}, set
\begin{equation}\label{eq:4.4.18}
\z:=4\tau+\h+8.
\end{equation}
where $\h$ is fixed in Proposition \ref{teo2ham}. 
Consider a Lipschitz family $ u(\oo)$ with $\oo\in\Lambda_{o}\subseteq\Lambda$ such that
\begin{equation}\label{eq:4.4.19}
||u||_{\gots_{0}+\z,\g} \leq1.
\end{equation}
Define the set
\begin{equation}\label{primedimham}
\calP^{2\g}_{\infty}({ u}):=\left\{\begin{array}{ll}
\oo\in\Lambda_{o} : &|\oo\cdot\ell\!+\!\mu_{\s,{j}}(\oo)|\geq
\frac{2\g\langle j\rangle^{2}}{\langle\ell\rangle^{\tau}}, \\
&\;\ell\in\ZZZ^{d}, \s,\in\CC, j\in\ZZZ 
\end{array}
\right\}.
\end{equation}
There exists $\epsilon_{0}$, depending only on the data of the problem, such that
if $\e\g^{-1}<\epsilon_{0}$
then, for any $\oo\in\Lambda^{2\g}_{\infty}({u})\cap\calP^{2\g}_{\infty}({u})$ 
(see (\ref{martina10ham})),  and for any Lipschitz family
$g(\oo)\in {\bf H}^{s}$, 
the equation $\calL { h}:=\calL(\oo,{ u}(\oo)) h= { g}$,
admits a solution $h\in{\bf H}^{s}$
such that for $\gots_{0}\leq s \leq q-\mu$
\begin{equation}\label{eq:4.4.23}
||{ h}||_{s,\g}\leq C(s)\g^{-1}\left(||{ g}||_{s+2\tau+5,\g}+
||{ u}||_{s+\z,\g}||{ g}||_{\gots_{0},\g}
\right).
\end{equation}
\end{lemma}

Proposition \ref{inverseofl}, combined with Lemma \ref{lemmaccio}, guarantees 
the invertibility (with suitable estimates) of the linearized of $\calF$. Of course $d_{z}\calF$ can be inverted only in a 
suitable set of parameters depending on the function $z$ on which we linearize.
In principle it can be empty and moreover it is not sufficient to prove that it has positive measure.
 Indeed  we need to invert the linearized operator
in any approximate solution $u_{n}$ (see Proposition \ref{teo2ham}) so that 
$\cap_{n\geq0}\left(\Lambda_{\infty}^{2\g}(u_{n})\cap\calP^{2\g}_{\infty}({u}_{n})\right)$ can have zero measure.
 The last part of the paper is
devoted to give some measure estimates of such set. 
In the Nash-Moser proposition \ref{teo4} we defined in an implicit way the sets $\calG_n$ in order to ensure bounds on the inverse of  $\calL({u}_{n})$. The following Proposition is the main result of Section \ref{sec6ham}. 
\begin{proposition}[{\bf Measure estimates}]\label{STIMEmisura}
Set $\g_n:=(1+2^{-n})\g$ and consider the set $\calG_{\infty}$ of Proposition \ref{teo4}
with $\mu=\zeta$ defined in Lemma \ref{inverseofl} and fix $\g:=\e^{a}$ for some $a\in(0,1)$. We have
\begin{subequations}\label{eq136tot}
\begin{align}
&\cap_{n\geq0}\calP^{2\g_n}_{\infty}({u}_n)\cap \Lambda^{2\g_n}_{\infty}({ u}_n)\subseteq \calG_{\infty}, \label{eq136b}\\
&|\Lambda\backslash\calG_{\infty}|\to 0, \;\; {\rm as} \;\; \e\to0. \label{eq136}
\end{align}
\end{subequations}
\end{proposition}
 Formula (\ref{eq136b}) is essentially trivial. One just need to look at Definition \ref{invertibility}
 which fix the sets $\calG_{n}$. It is important because gives us the connection between $\calG_{\infty}$ and the sets we have constructed at each step of the iteration. The (\ref{eq136}) is more delicate. The first point is that we reduce to computing the measure of the left hand side of (\ref{eq136b}).

The strategy described above is similar to that followed in \cite{BBM} and \cite{FP}. It is quite general and can be applied to various case.
The main differences are in the proof
of Proposition \ref{teo2ham}. 
Clearly it depends on the unperturbed eigenvalues and on the symmetries one ask for on the system.

\subsection{Proof of Theorem \ref{teo1}}

\noindent
Theorem \ref{teo1} essentially follows by Propositions \ref{teo4} and \ref{STIMEmisura}. The measure estimates performed in the last section guarantee that the ``good'' sets defined in Prop. \ref{teo4} are not empty,
but on the contrary have ``full'' measure. 
In particular one uses the result of Proposition \ref{teo4} in order to prove Lemma \ref{inverseofl}. Indeed one one has diagonalized the linearized operator it is trivial to get estimate \eqref{eq:4.4.23}. 
From formula \eqref{eq:4.4.23} essentially follows \eqref{eq136b}.
Concerning the proof of Proposition \ref{teo4}, we have omitted since it is the same of Proposition $1.6$ in \cite{FP}. The only differences is that in \cite{FP} the authors deal with
a functional that is diagonal plus a non linear perturbation. In this case the situation is slightly different.
However the next Lemma guarantees  that the subspaces $H_{n}$ in \eqref{trig} are preserved by the linear part of our functional $\calF$ in \eqref{totale},  
%
\begin{lemma}\label{lemma2}
One has that
\begin{equation}
D_{\oo} : H_{n} \to H_{n}.
\end{equation}
\end{lemma}

\begin{proof}
Let us consider $u=(u^{(1)},u^{2})\in H_{n}$, then
\begin{equation*}
\begin{aligned}
D_{\oo}u&=D_{\oo}\sum_{|(\ell,j)|\leq N_{n}}u^{(i)}_{j}(\ell)e^{i\ell\cdot\f+ijx}\\
&=
\left(\begin{matrix}
\sum_{|(\ell,j)|\leq N_{n}}(i\oo\cdot\ell)u^{1}(\ell)_{j}-[(ij)^{2}+m]u^{(2)}_{j}e^{i\ell\cdot\f+ijx}\\
\sum_{|(\ell,j)|\leq N_{n}}(i\oo\cdot\ell)u^{2}(\ell)_{j}+[(ij)^{2}+m]u^{(1)}_{j}e^{i\ell\cdot\f+ijx}
\end{matrix}\right)\in H_{n}.
\end{aligned}
\end{equation*}
\end{proof}
%
%
%
%
We fix $\g:=\e^{a}$, $a\in(0,1)$. Then the smallness condition $\e\g^{-1}=\e^{1-a}<\epsilon_{0}$ of Proposition
\ref{teo4} is satisfied. Then we can apply it with $\mu=\zeta$ in (\ref{eq:4.4.18}) (see
Proposition \ref{STIMEmisura}).
 Hence by (\ref{teo42}) we have that the function
${u}_{\infty}$ in ${\bf H}^{\gots_{0}+\zeta}$ is a solution of the perturbed NLS with frequency
$\oo$. Moreover, one has
\begin{equation}\label{eq:6.1}
|\Lambda\backslash\calG_{\infty}|\stackrel{(\ref{eq136})}{\to}0,
\end{equation}
as $\e$
 tends to zero. To complete the proof of the theorem, it remains to prove 
 the linear stability of the solution. 
Since the eigenvalues $\mu_{\s,j}$ are purely imaginary, we know that
the Sobolev norm of the solution ${ v}(t)$ of (\ref{1.2.6ham}) is constant in time. We just need to show
that the Sobolev norm of ${ h}(t)=W_{2}{ v}(t)$, solution of $\calL h=0$
does not grow on time. 
Again to do this one can follow the same strategy used in \cite{FP}. In particular one uses the results of Lemma
\ref{lem:3.9ham} in Section \ref{sec:3ham} and estimates \eqref{eq:4.9ham} in Proposition \ref{KAMalgorithmham} in order to get the estimates

%
\begin{subequations}\label{eq:6.14}
\begin{align}
||{ h}(t)||_{H_{x}^{s}}&\leq K||{ h}(0)||_{H_{x}^{s}},\label{eq:6.14a}\\
\!\!\!\!||{ h}(0)||_{H_{x}^{s}}-\e^{b}K||{ h}(0)||_{H_{x}^{s+1}}\leq&||{h}(t)||_{H_{x}^{s}}\leq 
||{ h}(0)||_{H_{x}^{s}}+\e^{b}K||{h}(0)||_{H_{x}^{s+1}}, \label{eq:6.14b}
\end{align}
\end{subequations}
for $b\in(0,1)$.
Clearly the (\ref{eq:6.14}) imply the linear stability of the solution, so we concluded the proof of Theorem \ref{teo1}.
The rest of the paper is devoted to the proof of Propositions \ref{teo2ham},\ref{STIMEmisura} and  Lemma \ref{inverseofl}.
\zerarcounters
\section{ Regularization of the linearized operator}\label{sec:3ham}
 

In this section and in Section \ref{sec:4ham} we apply a reducibility scheme in order to 
conjugate the linearized operator to a linear, constant coefficients differential operator.  
Here we consider the linearized operator $\calL$ in (\ref{lemmaccio4})
and we construct two operators $\VV_{1}$ and $\VV_{2}$ in order to semi-conjugate
$\calL$ to an operator $\calL_{c}$ of the second order with constant coefficients 
plus a remainder of order $O(\del_{x}^{-1})$.
We look for such transformations because, in order to apply a KAM-type algorithm to diagonalize $\calL$, we need first a precise control of the asymptotics of the eigenvalues,  and  also
some estimates of the transformations 
$\VV_{i}$ with $i=1,2$ and their inverse.

The principal result we prove is the following.

\begin{lemma}\label{lem:3.88}
Let $\ff\in C^{q}$ satisfy  the Hypotheses of Proposition \ref{teo4} and assume $q>\h_{1}+\gots_{0}$  
where 
\begin{equation}\label{eq:3.2.0ham}
\h_{1}:=d+2\gots_{0}+10.
\end{equation}
There exists $\epsilon_0>0$ such that, if $\e\g_{0}^{-1}\leq \epsilon_0$ (see (\ref{dio} for the definition of $\g_0$) then, for any $\g\le \g_0$ and for all 
 ${u} \in{\bf H}^{0}$ depending in a Lipschitz way on  $\oo\in \Lambda$,
 if
\begin{equation}\label{eq:3.2.1bham}
||{ u}||_{\gots_{0}+\h_{1},\g}\leq\e\g^{-1},
\end{equation}
 then, for $\gots_{0}\leq s\leq q-\h_{1}$, 
the following holds.

\noindent (i) 
 There exist invertible maps $\VV_1, \VV_2: {\bf H}^{0}\to{\bf H}^{0}$ such that
 $\calL_{7}:=\VV_{1}^{-1}\calL\VV_{2}=$
 \begin{equation}\label{eq:3.5.9ham}
\oo\cdot\del_{\f}\uno+
i\left(\begin{matrix}m_{2} \!\!\!& 0 \\ 0 & \!\!\!-m_{2}\end{matrix}\right)\del_{xx}+
i\left(\begin{matrix}m_{1} \!\!\!& 0 \\ 0 & \!\!\!-\bar{m}_{1}\end{matrix}\right)\del_{x}
+i \left(\begin{matrix}m_{0} \!\!\!& q_{0}(\f,x) \\ -\bar{q}_{0}(\f,x) & \!\!\!-m_{0}\end{matrix}\right)+
\RR
 \end{equation}
 with $m_{2},m_{0}\in \RRR$, $m_{1}\in i\RRR$ and $\RR$ is a pseudo-differential operator of order 
$O(\del_{x}^{-1})$ (see \eqref{pseudo}).
The  $\VV_{i}$ are symplectic maps and moreover for all ${ h}\in {\bf H}^{0}$
 \begin{equation}\label{eq:3.2.3ham}
 ||\VV_{i}{ h}||_{s,\g}+||\VV_{i}^{-1}{ h}||_{s,\g}\leq 
 C(s)(|| { h}||_{s+2,\g}+
 ||{ u}||_{s+\h_{1},\g}||{ h}||_{\gots_{0}+2,\g}), \quad i=1,2.
\end{equation}
\noindent
(ii) The coefficient $m_{i}:=m_{i}({ u})$ for $i=0,1,2$ of $\calL_{7}$ satisfies
\begin{equation}\label{eq:3.2.44}
\begin{aligned}
|m_{2}({ u})-1|_\g,|m_{0}({ u})-\mathtt{m}|_\g&\leq \e C,\;\;
|d_{{ u} }m_{i}({ u})[{ h}]|\leq \e C||{ h}||_{\h_{1}}, 
\;\; i=0,2, \\
|m_{1}({ u})|&\leq \e C,
|d_{{ u} }m_{1}({ u})[{ h}]|\leq \e C||{ h}||_{\h_{1}}, 
\end{aligned}
\end{equation}
and moreover the constant $m_{1}:=m_{1}(\oo,u(\oo))$ satisfies 
\begin{subequations}\label{mammamia}
\begin{align}
&\e c \leq |m_{1}({ u})|, \label{mammamiaa}\\
& \sup_{\oo_{1}\neq\oo_{2}}\frac{|m_{1}(\oo_{1},u(\oo))-m_{1}(\oo_{2},u(\oo))|}{|\la_{1}-\la_{2}|} \leq \e^{2} C\g^{-1}
\label{mammamiab}
\end{align}
\end{subequations}
for some $C>0$.

\noindent
(iii) The operator $\RR:=\RR(u)$ is such that
\begin{eqnarray}\label{eq:3.2.7ham}
\|\RR({ u}){h}\|_{s,\g}&\leq& \e C(s) (\|{ h}\|_{s,\g}+\|{ u}\|_{s+\h_{1},\g}\|{ h}\|_{\gots_{0}}), \label{eq:B15aham}\\
\|d_{{ u}}\RR({ u})[{ h}]{ g}\|_{s}&\leq& \e C(s)\big(\|{ g}\|_{s+1}\|{ h}\|_{\gots_{0}+\h_{1}}+
\|{ g}\|_{2}\|{ h}\|_{s+\h_{1}}\nonumber\\
&+&\|{ u}\|_{s+\h_{1}}\|{ g}\|_{2}\|{ h}\|_{\gots_{0}}\big),\label{eq:B15bham}
\end{eqnarray}
and moreover 
\begin{subequations}\label{eq:3.2.7bisham}
\begin{align}
\|q_{0}\|_{s,\g}&\leq \e C(s)(1+\|{ u}\|_{s+\h_{1},\g}),\label{eq:3.2.6aham}\\
\|d_{{ u} }q_{0}({u})[{ h}]\|_{s}&\leq \e C(s)(\|{ h}\|_{s+\h_{1}}+\|{ u}\|_{s+\h_{1}}+
\|{h}\|_{\gots_{0}+\h_{1}}),\label{eq:3.2.6bham}
\end{align}
\end{subequations}
Finally 
$\calL_{7}$ is Hamiltonian.
\end{lemma}

\begin{rmk}\label{trenonapoli}
The estimate in \eqref{mammamia} is different from that in  \eqref{eq:3.2.44}. As we will see, it is very important 
to estimate the Lipschitz norm of the constant $m_{1}$ in order to get the measure estimates in Section 6.  
The constant $m_{1}$ depends 
in $\oo$ in two way: the first is trough the dependence on $\oo$ of the function $u$; secondly it presents also an explicit dependence on the external parameters. Clearly by \eqref{eq:3.2.44} we can get a bound only on the variation
$
|m_{1}(\oo,u(\oo_{1}))-m_{1}(\oo,u(\oo_{2}))|.
$
To estimate the $|\cdot|^{lip}$ seminorm we need also the \eqref{mammamia}.
\end{rmk}

 The proof of Lemma \ref{lem:3.88} is based on the following strategy.
 At each step we construct a transformation $\TT_{i}$ that
conjugates $\calL_{i}$ to $\calL_{i+1}$. We fix $\calL_{0}=\calL$. Moreover
the $\TT_{i}$ are symplectic, hence $\calL_{i}$ is Hamiltonian and has the form
\begin{equation}\label{elleham}
\begin{aligned}
\calL_{i}:=
\oo\cdot\del_{\f}\uno+
i(E+A_{2}^{(i)})\del_{xx}+iA_{1}^{(i)}\del_{x}+i(mE+A_{0}^{(i)})+\RR_{i},
\end{aligned}
\end{equation}
with $E$ defined in \eqref{lemmaccio44}, 
\begin{equation}\label{elle2ham}
A_{j}^{(i)}=A_{j}^{(i)}(\f,x):=
\left(\begin{matrix}a^{(i)}_{j} \!\! &\!\! b^{(i)}_{j} \\ \!\! -\bar{b}^{(i)}_{j} &\!\! -\bar{a}^{(i)}_{j} \end{matrix}\right),\;\;\;\;
j=0,1,2
\end{equation}
and $\RR_{i}$ is a pseudo-differential operator of order $\del_{x}^{-1}$.
Essentially we need to prove bounds like
\begin{eqnarray}\label{eq:B15aaaa}
\|(\TT_{i}^{\pm1}({u})-\uno){ h}\|_{s,\g}
&\!\!\!\leq\!\!\!& \e C(s) (\|{ h}\|_{s,\g}+\|{ u}\|_{s+\ka_{i},\g}\|{ h}\|_{\gots_{0}}), \label{eq:B15cham}\\
\|d_{{u}}(\TT^{\pm1}_{i})({u})[{ h}]{ g}\|_{s}
&\!\!\!\leq\!\!\!& \e C(s)\big(\|{ g}\|_{s+1}\|{ h}\|_{\gots_{0}+\ka_{i}}+
\|{g}\|_{2}\|{ h}\|_{s+\ka_{i}}\nonumber\\
&\!\!\!+\!\!\!&\|{u}\|_{s+\ka_{i}}\|{ g}\|_{2}\|{h}\|_{\gots_{0}}
\big),\label{eq:B15dham}
\end{eqnarray}
for suitable $\ka_{i}$ and on the 
coefficients in (\ref{elleham}) we need 
\begin{subequations}\label{eq:3.10ham}
\begin{align}
\|a_{j}^{(i)}({u})\|_{s,\g}, \|b_{j}^{(i)}({ u})\|_{s,\g}&\leq \e C(s)(1+\|{ u}\|_{s+\ka_{i},\g}),\label{eq:3.10aham}\\
\!\!\!\!\!\|d_{u}a_{j}^{(i)}({ u})[h]\|_{s}, \|d_{u}b_{j}^{(i)}({ u})[h]\|_{s}&\leq \!\e C(s)(\|{ h}\|_{s+\ka_{i}}\!+\!\|{u}\|_{s+\ka_{i}}\!+\!
\|{ h}\|_{\gots_{0}+\ka_{i}}),\label{eq:3.10bham}
\end{align}
\end{subequations}
for $j=0,1,2$ and $i=1,\ldots,7$ and on $\RR_{i}$ bounds like \eqref{eq:3.2.7ham} with $\ka_{i}$ instead of $\h_{1}$.

The bounds are based on repeated use of  classical tame bounds and interpolation estimates of the Sobolev norms. The proof of such properties of the norm can be found in \cite{BBM} in Appendix A.
To conclude one combine the bounds of each transformation to obtain estimates on the compositions. It turn out that
the constant $\h_{1}$ contains all the loos of regularity of each step. We present only the construction
of the transformation that, in the Hamiltonian case, are more involved.
Moreover the difference between Lemma \ref{lem:3.88} and 
the result contained in Section 3 of \cite{FP}
is also in equation \eqref{mammamia}. Indeed, in this case we need to prove that non degeneracy hypothesis \ref{hyp3ham}
persists during the steps in order to obtain the same lower bound (possibly with a worse constant) for the constant $m_{1}$ in \eqref{eq:3.2.44}. This fact will be used in Section 6 in order to perform measure estimates. 

\paragraph{Step 1. Diagonalization of the second order coefficient}

In this section we want to diagonalize the second order term $(E+A_{2})$ in \eqref{lemmaccio4}.
By a direct calculation one can see that the matrix $(E+A_{2})$ has eigenvalues
$\la_{1,2}:=\sqrt{(1+a_{2})^{2}-|b_{2}|^{2}}$. if we set $a_{2}^{(1)}:=\la_{1}-1$ we have
that $a_{2}^{(1)}\in \RRR$ since $a_{2}\in \RRR$ for any $(\f,x)\in\TTT^{d+1}$ and $a_{i}, b_{i}$ are small.
We define the transformation $\TT^{-1}_{1} : {\bf H}^{0}\to{\bf H}^{0}$ as the matrix 
$\TT_{1}^{-1}=\big((\TT_{1}^{-1})_{\s}^{\s'}\big)_{\s,\s'=\pm1}$ with
\begin{equation}\label{11}
\TT^{-1}_{1}:=
 \left(\begin{matrix} (2+a_{2}+a_{2}^{(1)})(i\la_{0})^{-1} & b_{2}(i\la_{0})^{-1} \\
-\bar{b}_{2}(i\la_{0})^{-1} &  -(2+a_{2}+a_{2}^{(1)})(i\la_{0})^{-1}\end{matrix}
\right),
\end{equation}
where $\la_{0}:=i\sqrt{2\la_{1}(1+a_{2}+\la_{1})}$. Note that ${\rm det}\TT^{-1}_{1}=1$.
One has that 
\begin{equation}\label{fiorellinoham}
 \TT^{-1}_{1}(E +A_{2})\TT_{1}=\left(\begin{matrix}1+a^{(1)}_{2}(\f,x) & 0\\ 0 & -1-a_{2}^{(1)}(\f,x)\end{matrix}\right).
\end{equation}
Moreover, we have that the transformation is symplectic. We can think that $\TT_{1}$ act on the function of $H^{s}(\TTT^{d+1};\CCC)$ is the following way.
Set $U=(u,\bar{u}),V=(v,\bar{v})\in {\rm H}^{s}$ and let $(MZ)_{\s}$ for $\s\in\{+1,-1\}$ be the first or
the second (respectively) component.
Given a function $u\in H^{s}(\TTT^{d+1};\CCC) $ we define, with abuse of notation, $\TT^{-1}_{1}u:=(\TT^{-1}_{1}U)_{+1}:=
((\TT^{-1}_{1})_{1}^{1})u+((\TT^{-1}_{1})_{1}^{-1})\bar{u}$.
With this notation one has that
\begin{equation*}
\begin{aligned}
\Omega\left(\TT^{-1}_{1}{ u},\TT^{-1}_{1}{ v} \right):={\rm Re}\int_{\TTT}
&i\left((\TT^{-1}_{1})_{1}^{1}(\TT^{-1}_{1})_{-1}^{1}uv+(\TT^{-1}_{1})_{1}^{-1}(\TT^{-1}_{1})_{-1}^{-1}\bar{u}\bar{v}\right)\\
&+i\left((\TT^{-1}_{1})_{1}^{1}(\TT^{-1}_{1})_{-1}^{-1}u\bar{v}+(\TT^{-1}_{1})_{1}^{-1}(\TT^{-1}_{1})_{-1}^{1}\bar{u}v\right){\rm d}x\\
&\!\!\!\!\!\!\!\!\!\!\!\!\!\!\!\!\!\!\!\!\!\!\! ={\rm Re}\int_{\TTT} i{\rm Re}((\TT^{-1}_{1})_{1}^{1}(\TT^{-1}_{1})_{-1}^{1}uv)\\
&+i\left((\TT^{-1}_{1})_{1}^{-1}(\TT^{-1}_{1})_{-1}^{1}(u\bar{v}+\bar{u}v)\right)\\
&+i\left((\TT^{-1}_{1})_{1}^{1}(\TT^{-1}_{1})_{-1}^{-1}-(\TT^{-1}_{1})_{1}^{-1}(\TT^{-1}_{1})_{-1}^{1}\right)u\bar{v}{\rm d} x\\
&\!\!\!\!\!\!\!\!\!\!\!\!\!\!\!\!\!\!\!\!\!\!\! ={\rm Re}\int_{\TTT} iu\bar{v}{\rm d}x=:\Omega(u,v).
\end{aligned}
\end{equation*}
%
which implies that $\TT_{1}^{-1}$ is symplectic.

Now we can conjugate the operator $\calL$ to an operator $\calL_{1}$ with
a diagonal coefficient of the second order spatial differential operator. Indeed, one has
\begin{equation}\label{eq:3.122}
\begin{aligned}
\calL_{1}&:=\TT_{1}^{-1}\calL \TT_{1}=\oo\cdot\del_{\f}\uno+i \TT_{1}^{-1}(E+A_{2})\TT_{1}\del_{xx}\\
&+
i(2\TT_{1}^{-1}(E+A_{2})\del_{x}\TT_{1}+\TT_{1}^{-1}A_{1}\TT_{1})\del_{x}\\
&+i\left[-i\TT_{1}^{-1}(\oo\cdot\del_{\f}\TT_{1})+\TT_{1}^{-1}(E+A_{2})\del_{xx}\TT_{1}\right.\\
&\left.+
\TT_{1}^{-1}A_{1}\del_{x}\TT_{1}+\TT_{1}^{-1}(mE+A_{0})\TT_{1}
\right];
\end{aligned}
\end{equation}
the (\ref{eq:3.122}) has the form \eqref{elleham}. This identify uniquely the coefficients
$a_{j}^{(1)},b_{j}^{(1)}$ for $j=0,1,2$ and $\RR_{1}$. In particular we have that
$b_{2}^{(1)}\equiv0$ and $\RR_{1}\equiv0$. 
Moreover, since the transformation is symplectic, then the new operator $\calL_{1}$ is 
Hamiltonian, with an Hamiltonian function
\begin{equation}\label{ham1}
\begin{aligned}
\!\!\!\!\!\!H_{1}(u,\bar{u})&=\int_{\TTT}(1+a_{2}^{(1)})|u_{x}|^{2}-\frac{i}{2}{\rm Im}(a_{1}^{(1)})(u_{x}\bar{u}-u\bar{u}_{x})\!-\!
{\rm Re}(a_{0}^{(1)})|u|^{2}{\rm d}x\\
&+\int_{\TTT}\!\!-m|u|^{2}\!-\!\frac{1}{2}(b_{0}^{(1)}\bar{u}^{2}+\bar{b}_{0}^{(1)}u^{2}){\rm d}x
:=\int_{\TTT}f_{1}(\f,x,u,\bar{u},u_{x},\bar{u}_{x}){\rm d}x,
\end{aligned}
\end{equation}
hence, since $f_{1}$ depends only linearly on $\bar{u}_{x}$, thanks to the Hamiltonian structure, 
one has
\begin{equation}
b_{1}^{(1)}(\f,x)=\frac{\rm d}{{\rm d}x}\left(\del_{\bar{z}_{1}\bar{z_{1}}}f_{1}\right)\equiv0.
\end{equation}
This means that we have diagonalized also the matrix of the first order spatial differential operator.

\begin{rmk}\label{rmk23}
It is important to note that $a_{1}^{(1)}(\f,x)$ as the form 
$$
a_{1}^{(1)}(\f,x)=\frac{\rm d}{{\rm d}x}a_{2}^{(1)}(\f,x)+\del_{z_{0}\bar{z}_{1}}f_{1}-\del_{z_{1}\bar{z}_{0}}f_{1}
$$
so that the real part of $a_{1}^{(1)}$ depends only on the spatial derivative of $a_{2}^{(1)}$.
\end{rmk}

\paragraph{Step 2. Change of the space variable}

We consider a $\f-$dependent family of diffeomorphisms of the $1-$dimensional torus $\TTT$
of the form
\begin{equation}\label{20ham}
y=x+\x(\f,x),
\end{equation}
where $\x$ is as small real-valued funtion, $2\pi$ periodic in all its arguments. We define the change of variables
on the space of functions as 
\begin{equation}\label{21ham}
\begin{aligned}
(\TT_{2} h)(\f,x)&:=\sqrt{1+\x_{x}(\f,x)}h(\f,x+\x(\f,x)), \; {\rm with\;\; inverse} \;\;\;\;\\
(\TT_{2}^{-1}v)(\f,y)&:=\sqrt{1+\widehat{\x}_{x}(\f,y)}v(\f,y+\widehat{\x}(\f,y))
\end{aligned}
\end{equation}
where 
\begin{equation}\label{20bisham}
x=y+\widehat{\x}(\f,y),
\end{equation}
%
is the inverse  diffeomorphism of $(\ref{20ham})$.
With a slight abuse of notation we  extend the operator to ${\bf H}^s$:
\begin{equation}\label{eq:3.14ham}
{{\TT_{2}}} : {\bf H}^{s}\to {\bf H}^{s}, \quad {\TT_{2}}\left(\begin{matrix} h \\ \bar{h}\end{matrix}\right)=\left(
\begin{matrix} (\TT_{2} h)(\f,x) \\ (\TT_{2} \bar{h})(\f,x)\end{matrix}\right).
\end{equation}

Now we have to calculate the conjugate ${\TT_{2}}^{-1}\calL_{1}{\TT_{2}}$ 
of the operator $\calL_{1}$ in 
$(\ref{eq:3.122})$.

\noindent
The conjugate $\TT_{2}^{-1} a\TT_{2}$ of any multiplication operator
$a : h(\f,x) \to a(\f,x)h(\f,x)$ is the multiplication operator 
\begin{equation}\label{con}
v(\f,y)\mapsto(\TT_{2}^{-1}a\sqrt{1+\x_{x}})(\f,y)v(\f,y)=a(\f,y+\widehat{\x}(\f,y))v(\f,y).
\end{equation}
In \eqref{con} we have used the relation
\begin{equation}\label{con2}
0\equiv \x_{x}(\f,x)+\widehat{\x}_{y}(\f,y)+\x_{x}(\f,x)\widehat{\x}_{y}(\f,y),
\end{equation}
that follow by \eqref{20ham} and \eqref{20bisham}.
The conjugate of the differential operators will be
\begin{equation}\label{22ham}
\begin{aligned}
\TT_{2}^{-1}\oo\cdot\del_{\f}\TT_{2}&=\oo\cdot\del_{\f}+[\TT_{2}^{-1}(\oo\cdot\del_{\f}\x)]\del_{y}
-\TT_{2}^{-1}\left(\frac{\oo\cdot\del_{\f}\x_{x}}{2(1+\x_{x})}\right)
,\\
\TT_{2}^{-1}\del_{x}\TT_{2}&=[\TT_{2}^{-1}(1+\x_{x})]\del_{y}-\TT_{2}^{-1}\left(
\frac{\x_{xx}}{2(1+\x_{x})}\right),\\
\TT_{2}^{-1}\del_{xx}\TT_{2}&=[\TT_{2}^{-1}(1+\x_{x})^{2}]\del_{yy}-\TT_{2}^{-1}\left(\frac{2\x_{xxx}+\x_{xx}^{2}}{4(1+\x_{x})^{2}} \right),
\end{aligned}
\end{equation}
where all the coefficients
are periodic functions of $(\f,x)$.
%
%
%
\noindent
Thus, by conjugation, we have that $\calL_{2}=\TT_{2}^{-1}\calL_{1}\TT_{2}$ has the form (\ref{elleham}) 
with
\begin{equation}\label{24ham}
\begin{aligned}
&\!\!\!\!\!\!1+a_{2}^{(2)}(\f,y)=\TT_{2}^{-1}[(1+a_{2}^{(1)})(1+\x_{x})^{2}],
 \\ 
&\!\!\!\!\!\! a_{1}^{(2)}(\f,y)=
\TT_{2}^{-1}(a_{1}^{(1)}(1+\x_{x}))-i \TT_{2}^{-1}(\oo\cdot\del_{\f}\x),
\\
&\!\!\!\!\!\! a_{0}^{(2)}(\f,y)=i\TT_{2}^{-1}\!\!\left(\frac{\oo\cdot\del_{\f}\x_{x}}{2(1+\x_{x})}\right)
\!-\!\TT_{2}^{-1}\!\left(\frac{\x_{xx}}{2(1+\x_{x})}\right)
\!-\!\TT_{2}^{-1}\!\left(\frac{2\x_{xxx}+\x_{xx}^{2}}{4(1+\x_{x})^{2}} \right)\!,
\\
&\!\!\!\!\!\! b^{(2)}_{0}(\f,y)=\TT_{2}^{-1}(b^{(1)}_{0}),
\end{aligned}
\end{equation}
and $b_{2}^{(2)}=b_{1}^{(2)}=0$.
We are looking for $\x(\f,x)$ such that the coefficient $a_{2}^{(2)}(\f,y)$ does not depend on
$y$, namely
\begin{equation}\label{255}
1+a_{2}^{(2)}(\f,y)=\TT_{2}^{-1}[(1+a_{2}^{(1)})(1+\x_{x})^{2}]=1+a_{2}^{(2)}(\f),
\end{equation}
for some function $a_{2}^{(2)}(\f)$. Since $\TT_{2}$ operates only on the space variables, the $(\ref{255})$
is equivalent to 
\begin{equation}\label{266}
(1+a_{2}^{(1)}(\f,x))(1+\x_{x}(\f,x))^{2}=1+a_{2}^{(2)}(\f).
\end{equation}
Hence we have to set
\begin{equation}\label{277}
\x_{x}(\f,x)=\rho_{0}, \qquad \rho_{0}(\f,x):=(1+a_{2}^{(2)}(\f))^{\frac{1}{2}}(\f)(1+a_{2}^{(1)}(\f,x))^{-\frac{1}{2}}-1,
\end{equation}
that has solution $\g$ periodic in $x$ if and only if $\int_{\TTT}\rho_{0} dy=0$. This condition
implies
\begin{equation}\label{288}
a_{2}^{(2)}(\f)=\left(\frac{1}{2\pi}\int_{\TTT}(1+a_{2}^{(1)}(\f,x))^{-\frac{1}{2}}\right)^{-2}-1.
\end{equation}
Then we have the solution (with zero average) of $(\ref{277})$
\begin{equation}\label{299}
\x(\f,x):=(\del_{x}^{-1}\rho_{0})(\f,x),
\end{equation}
where $\del_{x}^{-1}$ is defined by linearity as
\begin{equation}\label{300}
\del_{x}^{-1}e^{ikx}:=\frac{e^{ikx}}{ik}, \quad \forall \; k\in\ZZZ\backslash\{0\}, \quad \del_{x}^{-1}=0.
\end{equation}
In other word $\del_{x}^{-1}h$ is the primitive of $h$ with zero average in $x$.
Moreover, the  map  $\TT_{2}$ is canonical with respect to the $NLS-$symplectic form, indeed, for any $u,v\in H^{s}(\TTT^{d+1};\CCC)$,
  \begin{equation*}
  \begin{aligned}
  \Omega(\TT_{2} u, \TT_{2} v)=&
{\rm Re}\int_{\TTT}(i\sqrt{1+\x_{x}}u(\f,x
+\f(\f,x)))\sqrt{1+\x_{x}}\bar{v}(\f,x
+\f(\f,x)){\rm d}x\\
&={\rm Re}\int_{\TTT}(1+\x_{x}(\f,x))(iu(\f,x+\x(\f,x)))\bar{v}(\f,x+\x(\f,x)){\rm d}x\\
&={\rm Re}\int_{\TTT}(iu(\f,y))\bar{v}(\f,y){\rm d}y=:\Omega(u,v).
\end{aligned}
  \end{equation*}
  Thus, conjugating $\calL_{1}$ through the operator $\TT_{2}$ in (\ref{eq:3.14ham})
  we obtain the Hamiltonian operator 
 $\calL_{2}=  \TT_{2}^{-1}\calL_{1} \TT_{2}$
%
with Hamiltonian function given by
\begin{equation}\label{linham2}
\begin{aligned}
\!\!\!\!\!\!\!H_{2}(u,\bar{u})&\!=\!\int_{\TTT}(1+a_{2}^{(2)}(\f))|u_{x}|^{2}\!-\!\frac{i}{2}{\rm Im}(a_{1}^{(2)})(u_{x}\bar{u}\!-\!u\bar{u}_{x})
\!-\!
{\rm Re}(a_{0}^{(2)})|u|^{2}{\rm d}x\\
&\!+\!\int_{\TTT}\!-\!m|u|^{2}\!-\!\frac{1}{2}(b_{0}^{(2)}\bar{u}^{2}+\bar{b}^{(2)}_{0}u^{2}){\rm d}x
\!:=\!\int_{\TTT}f_{2}(\f,x,u,\bar{u},u_{x},\bar{u}_{x}){\rm d}x,
\end{aligned}
\end{equation}

\begin{rmk}\label{rmk24}
As in Remark \ref{rmk23}, the real part of coefficients $a_{1}^{(2)}$ depends on the spatial derivatives
of $a_{2}^{(2)}$, then in this case, again thanks the Hamiltonian structure of the problem,
one has that
$a_{1}^{(2)}(\f,y)=i{\rm Im}(a_{1}^{(2)})(\f,y)$, i.e. it is purely imaginary.
Moreover $b_{2}^{(2)}=b_{1}^{(2)}\equiv 0$ and $\RR_{2}\equiv0$.
\end{rmk}

\paragraph{Step 3: Time reparametrization}

In this section we want to make constant the coefficient of the highest order spatial derivative
operator $\del_{yy}$ of $\calL_{2}$, by a quasi-periodic reparametrization of time.
We consider a diffeomorphism of the torus $\TTT^{d}$ of the form
\begin{equation}\label{32ham}
\theta=\f+\oo\al(\f), \quad \f\in\TTT^{d}, \quad \al(\f)\in\RRR,
\end{equation}
where $\al$ is a small real valued function, $2\pi-$periodic in all its arguments. The induced
linear operator on the space of functions is
\begin{equation}\label{332}
(\TT_{3} h)(\f,y):=h(\f+\oo\al(\f),y),
\end{equation}
whose inverse is
\begin{equation}\label{333bis}
(\TT_{3}^{-1}v)(\theta,y)=v(\theta+\oo\tilde{\al}(\theta),y),
\end{equation}
where $\f=\theta+\oo\tilde{\al}(\theta)$ is the inverse diffeomorphism of
$\theta=\f+\oo\al(\f)$. We extend  the operator to ${\bf H}^{s}$:
\begin{equation}\label{eq:3.15ham}
 \TT_{3} : {\bf H}^{s}\to {\bf H}^{s}, \quad  \TT_{3}\left(\begin{matrix} h \\ \bar{h}\end{matrix}\right)=\left(
\begin{matrix} (\TT h)(\f,x) \\ (\TT_{3} \bar{h})(\f,x)\end{matrix}
\right).
\end{equation}
\noindent
 By conjugation, we have that the differential operator become
\begin{equation}\label{333}
\TT_{3}^{-1}\oo\cdot\del_{\f}\TT_{3}=\rho(\theta)\oo\cdot\del_{\theta},\;\;
\TT_{3}^{-1}\del_{y}\TT_{3}=\del_{y}, \;\;\rho(\theta):=\TT_{3}^{-1}(1+\oo\del_{\f}\al).
\end{equation}
Hence we have $\TT_{3}^{-1}\calL_{2}\TT_{3}=\rho \calL_{3}$ where $\calL_{3}$ has the form \eqref{elleham}
and 
\begin{equation}\label{eq:3.166}
\begin{aligned}
1+a_{i}^{(3)}(\theta)&:=\frac{(\TT_{3}^{-1}(1+a_{i}^{(2)}))(\theta)}{\rho(\theta)},
\quad a_{i}^{(3)}(\theta):=\frac{(\TT_{3}^{-1}a_{i}^{(2)})(\theta)}{\rho(\theta)}, \;\; i=0,1, \\
& b_{0}^{(3)}(\theta,y):=\frac{(\TT_{3}^{-1}b_{0}^{(2)})(\theta,y)}{\rho(\theta)}, 
\end{aligned}
\end{equation}
We look for solution $\al$ such that the coefficient $a_{2}^{(3)}$ is constant in time, namely
\begin{equation}\label{355}
(\TT_{3}^{-1}(1+a_{2}^{(2)}))(\theta)=m_{2}\rho(\theta)=m_{2}\TT_{3}^{-1}(1+\oo\cdot\del_{\f}\al)
\end{equation}
for some constant $m_{2}$, that is equivalent to require that
\begin{equation}\label{366}
1+a_{2}^{(2)}(\f)=m_{2}(1+\oo\cdot\del_{\f}\al(\f)),
\end{equation}
By setting 
\begin{equation}\label{377}
m_{2}=\frac{1}{(2\pi)^{d}}\int_{\TTT^{d}}(1+a_{2}^{2}(\f))d\f,
\end{equation}
we can find the (unique) solution of $(\ref{366})$ with zero average
\begin{equation}\label{388}
\al(\f):=\frac{1}{m_{2}}(\oo\cdot\del_{\f})^{-1}(1+a_{2}^{(2)}-m_{2})(\f),
\end{equation}
where $(\oo\cdot\del_{\f})^{-1}$ is defined by linearity
$$
(\oo\cdot\del_{\f})^{-1}e^{i\ell\cdot\f}:=\frac{e^{i\ell\cdot\f}}{i\oo\cdot\ell}, \; \ell\neq0, \quad
(\oo\cdot\del_{\f})^{-1}1=0.
$$
Moreover, the operator $\TT_{3}$ acts only on the time variables, then it is clearly symplectic, since
$$
\Omega(\TT_{3} u, \TT_{3} v)=\Omega(u,v).
$$
Then the operator $\calL_{3}$ is Hamiltonian with hamiltonian function $H_{3}$ 
\begin{equation}\label{linham3}
\begin{aligned}
H_{3}(u,\bar{u})&=\int_{\TTT}m_{2}|u_{x}|^{2}\!-\!\frac{i}{2}{\rm Im}(a_{1}^{(3)})(u_{x}\bar{u}-u\bar{u}_{x})\!-\!
{\rm Re}(a_{0}^{(3)})|u|^{2}{\rm d}x\\
&+\int_{\TTT}-\frac{1}{2}(b_{0}^{(3)}\bar{u}^{2}+\bar{b}^{(3)}_{0}u^{2}){\rm d}x
:=\int_{\TTT}f_{3}(\f,x,u,\bar{u},u_{x},\bar{u}_{x}){\rm d}x,
\end{aligned}
\end{equation}

 \begin{rmk}\label{rmk25}
 Also in this case, thanks to the hamiltonian structure of the operator, we have that
 the coefficient $a_{1}^{(3)}\in i\RRR$, $b_{2}^{(3)}=b_{1}^{(3)}\equiv 0$ and $\RR_{3}\equiv0$.
\end{rmk}

\paragraph{Step 4. Change of space variable (translation)}

The goal of this section, is to conjugate 
$\calL_{3}$ in (\ref{elleham}) with coefficients in (\ref{eq:3.166}) 
to an operator in which the coefficients of the first order spatial derivative
operator, has zero average in $y$.

Consider the change of the space variable
\begin{equation}\label{eq:3.4.2ham}
z=y+\be(\theta)
\end{equation}
which induces the operators on functions
\begin{equation}\label{eq:3.4.3ham}
\TT_{4} h(\theta,y):=h(\theta,y+\be(\theta)), \quad \TT_{4}^{-1}v(\theta,z-\be(\theta)).
\end{equation}
We extend the operator $ \TT_{4}$ to  ${\bf H}^{s}$ as
\begin{equation}\label{eq:3.4.4ham}
 \TT_{4}\left(\begin{matrix} h \\ \bar h\end{matrix}\right)=\left(\begin{matrix}(\TT_{4} h)(\theta,y) \\ (\TT_{4}\bar{h})(\theta,y) \end{matrix}\right).
\end{equation}
By conjugation,  the differential operators become
\begin{equation}\label{eq:3.4.5ham}
\TT_{4}^{-1}\oo\cdot\del_{\theta}\TT_{4}=\oo\cdot\del_{\theta}+(\oo\cdot\del_{\theta}\be(\theta))\del_{z}, \qquad
\TT_{4}^{-1}\del_{y}\TT_{4}=\del_{z}.
\end{equation}
%
Hence one has that $\calL_{4}:=\TT_{4}^{-1}\calL_{3}\TT_{4}$ has the form \eqref{elleham} 
where 
\begin{equation}\label{eq:3.4.7ham}
\begin{aligned}
a_{1}^{(4)}(\theta,z)&:=-i\oo\cdot\del_{\theta}\be(\theta)+(\TT_{4}^{-1}a_{1}^{(3)})(\theta,z), \\
a_{0}^{(4)}(\theta,z):=(&\TT_{4}^{-1}a_{0}^{(3)})(\theta,z), \quad
b_{0}^{(4)}(\theta,z):=(\TT_{4}^{-1}b_{0}^{(4)})(\theta,z).
\end{aligned}
\end{equation}
The aim is to find a function $\be(\theta) $ such that
\begin{equation}\label{eq:3.4.8ham}
\frac{1}{2\pi}\int_{\TTT}a_{1}^{(4)}(\theta,z){\rm d}z=m_{1}, \quad \forall \; \theta\in\TTT^{d},
\end{equation}
for some constant $m_{1}\in\CCC$, independent on $\theta$. By using the (\ref{eq:3.4.7ham}) we have that
the (\ref{eq:3.4.8ham}) become
\begin{equation}\label{eq:3.4.9ham}
-i\oo\cdot\del_{\theta}\be(\theta)=m_{1}-\int_{\TTT}a_{1}^{(3)}(\theta,y){\rm d}y=:V(\theta).
\end{equation}
This equation has a solution periodic in $\theta$ if and only if $V(\theta)$ 
has zero average in $\theta$. So that
we have to define
\begin{equation}\label{eq:3.4.10ham}
m_{1}:=\frac{1}{(2\pi)^{d+1}}\int_{\TTT^{d+1}}a_{1}^{(3)}(\theta,y){\rm d}\theta{\rm d}y.
\end{equation} 
Note also that $m_{1}\in i\RRR$ (see Remark \ref{rmk25}). Then the function $V$ 
is purely imaginary. Now we can set
\begin{equation}\label{eq:3.4.11ham}
\be(\theta):=i(\oo\cdot\del_{\theta})^{-1}V(\theta),
\end{equation}
to obtain a real diffeomorphism of the torus $y+\be(\theta)$. 
Morover one has, for any $u,v\in H^{s}(\TTT^{d+1};\CCC)$
\begin{equation}\label{eq:3.4.13ham}
\Omega(\TT_{4} u,\TT_{4} v)={\rm Re}\int_{\TTT}iu(\f,x+\be(\f))\bar{v}(\f,x+\be(\f))=\Omega(u,v),
\end{equation}
hence $\TT_{4}$ is symplectic. This implies that $\calL_{4}$ is Hamiltonian with
hamiltonian function of the form
\begin{equation}\label{linham4}
\begin{aligned}
\!\!\!\!\!\!H_{4}(u,\bar{u})&=\int_{\TTT}m_{2}|u_{x}|^{2}\!-\!\frac{i}{2}{\rm Im}(a_{1}^{(4)})(u_{x}\bar{u}-u\bar{u}_{x})\!-\!
{\rm Re}(a_{0}^{(4)})|u|^{2}-m|u|^{2}{\rm d}x\\
&+\int_{\TTT}-\frac{1}{2}(b_{0}^{(4)}\bar{u}^{2}+\bar{b}_{0}^{(4)}u^{2}){\rm d}x
:=\int_{\TTT}f_{4}(\f,x,u,\bar{u},u_{x},\bar{u}_{x}){\rm d}x,
\end{aligned}
\end{equation}
\begin{rmk}\label{rmk35ham}
 Again one has $b_{2}^{(4)}=b_{1}^{(4)}\equiv 0$ and $\RR_{4}\equiv0$.
\end{rmk}

For simplicity we rename the variables $z=x$ and $\theta=\f$.

\paragraph{Step 5. Descent Method: conjugation by multiplication operator}

In this section we want to eliminate the dependance on $\f$ and $x$ on the coefficient
$c_{9}$ of  the operator $\calL_{4}$.
To do this, we consider an operator
of the form
\begin{equation}\label{eq:3.5.1ham}
\TT_{5}:=\left(\begin{matrix}1+z(\f,x) & 0 \\ 0 & 1+\bar{z}(\f,x) \end{matrix}\right),
\end{equation}
where $z : \TTT^{d+1}\to \CCC$.  By a direct calculation we have that
\begin{equation}\label{eq:3.5.2ham}
\begin{aligned}
\calL_{4}\TT_{5}-\TT_{5}&\left[\oo\cdot\del_{\f}\uno+
i\left(\begin{matrix}m_{2} \!\!\! & 0 \\ 0 &\!\!\! -m_{2} \end{matrix}\right)\del_{xx}+i
\left(\begin{matrix}m_{1} \!\!\! & 0 \\ 0 &\!\!\! -\ol{m}_{1} \end{matrix}\right)\del_{x}
\right]=\\
&=i\left(\begin{matrix}r_{1}(\f,x) \!\!\! & 0 \\ 0 &\!\!\! -\ol{r}_{1}(\f,x) \end{matrix}\right)\del_{x}
+i\left(\begin{matrix}m+c(\f,x) \!\!\! & d(\f,x) \\ -\ol{d}(\f,x) &\!\!\!-m -
\ol{c}(\f,x) \end{matrix}\right)
\end{aligned}
\end{equation}
where
\begin{equation}\label{eq:3.5.3ham}
\begin{aligned}
r_{1}(\f,x)&:=2m z_{x}(\f,x)+(a_{1}^{(4)}(\f,x)-m_{1})(1+z(\f,x)),\\
c(\f,x)&:=-i(\oo\cdot\del_{\f}z)(\f,x)+a_{0}^{(4)}(\f,x)(1+z(\f,x)),\\
d(\f,x)&:=b_{0}^{(4)}(\f,x)(1+\bar{z}(\f,x)).
\end{aligned}
\end{equation}
We look for $z(\f,x)$ such that $r_{1}\equiv0$. If we look for solutions of the form
$1+z(\f,x)=\exp(s(\f,x))$ we have that $r_{1}=0$ become
\begin{equation}\label{eq:3.5.7ham}
2m_{2} s_{x}+a_{1}^{(4)}-m_{1}=0,
\end{equation}
that has solution
\begin{equation}\label{eq:3.5.8ham}
s(\f,x):=\frac{1}{2m}\del_{x}^{-1}(a_{1}^{(4)}-m_{1})(\f,x)
\end{equation}
where $\del_{x}^{-1}$ is defined in (\ref{300}). Moreover, since 
$a_{1}^{(4)}\in i\RRR$, one has that $s(\f,x)\in i\RRR$. Clearly the operator $\TT_{5}$ is invertible for
$\e$ small, then we obtain $\calL_{5}:=\TT_{5}^{-1}\calL_{4}\TT_{5}$ with
\begin{equation}\label{eq:3.5.9bisham}
\calL_{5}:=
\oo\cdot\del_{\f}\uno+i\left(\begin{matrix}m_{2}\!\!\! & 0 \\ 0 & \!\!\! -m_{2}\end{matrix}\right)\del_{xx}+
i\left(\begin{matrix}m_{1}\!\!\! & 0 \\ 0 & \!\!\! -\bar{m}_{1}\end{matrix}\right)\del_{x}+imE
+iA_{0}^{(5)}
\end{equation}
that has the form  \eqref{elleham} with
 $m_{2}$ and $m_{1}$ are defined respectively in (\ref{377}) and (\ref{eq:3.4.10ham}), while the coefficients of $A_{5}^{(i)}$ are
\begin{equation}\label{eq:3.5.10ham}
\begin{aligned}
a_{0}^{(5)}(\f,x)&:=(1+z(\f,x))^{-1}c(\f,x),\\
b_{0}^{(5)}(\f,x)&:=(1+z(\f,x))^{-1}d(\f,x).
\end{aligned}
\end{equation}
It remains to check that the transformation $\exp(s(\f,x))=1+z$ is symplectic. 
One has
\begin{equation}\label{eq:3.5.11ham}
\Omega(e^{s}u,e^{s}v)={\rm Re}\int_{\TTT}ie^{s(\f,x)}u(\f,x)e^{-s(\f,x)}\bar{v}(\f,x){\rm d}x=
\Omega(u,v),
\end{equation}
where we used that $\bar{s}=-s$, that follows by $s\in i\RRR$.
Hence the operator $\calL_{5}$ is Hamiltonian, with corresponding hamiltonian function
\begin{equation}\label{linham5}
\begin{aligned}
\!\!\!\!\!\!\!H_{5}(u,\bar{u})&=\int_{\TTT}m_{2}|u_{x}|^{2}-\frac{i}{2}{\rm Im}(m_{1})(u_{x}\bar{u}-u\bar{u}_{x})-
{\rm Re}(a_{0}^{(5)})|u|^{2}{\rm d}x\\
&\!+\!\int_{\TTT}\!-\!m|u|^{2}\!-\!\frac{1}{2}(b_{0}^{(5)}\bar{u}^{2}+\bar{b}_{0}^{(5)}u^{2}){\rm d}x
:=\int_{\TTT}f_{5}(\f,x,u,\bar{u},u_{x},\bar{u}_{x}){\rm d}x.
\end{aligned}
\end{equation}
Again using the Hamiltonian structure, see (\ref{eq:1111}), we can conclude that
\begin{equation}\label{eq:3.5.12ham}
{\rm Im}(a_{0}^{(5)})(\f,x)=\frac{\rm d}{{\rm d}x}{\rm Im}(m_{1})\equiv0,
\end{equation}
that implies $a_{0}^{(5)}\in \RRR$.

\begin{rmk}\label{rmk45}
 We have $b_{2}^{(5)}=b_{1}^{(5)}\equiv 0$ and $\RR_{5}\equiv0$.
\end{rmk}

\paragraph{Step 6. Descent Method: conjugation by pseudo-differential operator}

In this section we want to conjugate $\calL_{5}$ in (\ref{eq:3.5.9bisham}) to an operator of the 
form $\oo\cdot\del_{\f}+iM\del_{xx}+iM_{1}\del_{x}+\RR$ where
\begin{equation}\label{eq:3.6.1ham}
M=\left(\begin{matrix}m_{2} & 0 \\ 0 & -m_{2}\end{matrix}\right), \quad M_{1}=
\left(\begin{matrix}m_{1} & 0 \\ 0 & -\bar{m}_{1}\end{matrix}\right),
\end{equation}
and $\RR$ is a pseudo differential operator of order $0$.

We consider an operator of the form
\begin{equation}\label{eq:3.6.2ham}
\tilde{\SSSS}:=\left(\begin{matrix}1+w\Upsilon & 0 \\ 0 & 1+\bar{w}\Upsilon \end{matrix}\right),
\end{equation}
where $w: \TTT^{d+1}\to \RRR$ and $\Upsilon=(1-\del_{xx})\frac{1}{i}\del_{x}$ 
is defined by linearity as
$$
\Upsilon e^{ijx}=\frac{1}{1+j^{2}}je^{ijx}
$$
%
We have that the difference
\begin{equation*}\label{eq:3.6.3ham}
\begin{aligned}
\calL_{5}\tilde{\SSSS}-\tilde{\SSSS}&\left[\oo\cdot\del_{\f}\uno\!+\!i
\left(\begin{matrix}m_{2} \!\!\!& 0 \\ 0 & \!\!\!-m_{2}\end{matrix}\right)\del_{xx}\!+\!
i\left(\begin{matrix}m_{1} \!\!\!& 0 \\ 0 & \!\!\!-\bar{m}_{1}\end{matrix}\right)\del_{x}
\!+\!i \left(\begin{matrix}m\!+\!\hat{a}_{0}^{(5)} \!\!\!& b_{0}^{(5)} \\ -\bar{b}_{0}^{(5)} & \!\!\!-\!m\!-\!\hat{a}_{0}^{(5)}\end{matrix}\right)\right]=\\
&=i\left(\begin{matrix}r_{0} & 0 \\ 0 & -\bar{r}_{0} \end{matrix}\right)+
\RR 
\end{aligned}
\end{equation*}
where $b_{0}^{(5)}$ is defined in (\ref{eq:3.5.10ham}) and
\begin{equation}\label{eq:3.6.4ham}
\begin{aligned}
r_{0}(\f,x)&:=2m_{2}w_{x}\Lambda\del_{x}+(a_{0}^{(5)}(\f,x)-\hat{a}_{0}^{(5)}(\f)),
\quad \RR 
=i\left(\begin{matrix} \tilde{p}_{0} & \tilde{q}_{0} \\ -\bar{\tilde{q}}_{0} & -\bar{\tilde{p}}_{0} 
\end{matrix}\right)
\\
\tilde{p}_{0}(\f,x)&:=-i(\oo\cdot\del_{\f}w) \Upsilon +m_{2}w_{xx}\Upsilon +m_{1}w_{x}\Upsilon +
(a_{0}^{(5)}-\hat{a}_{0}^{(5)})w\Upsilon,\\
\tilde{q}_{0}(\f,x)&:=b_{0}^{(5)} \bar{w}\Upsilon-w\Upsilon b_{0}^{(5)}.
\end{aligned}
\end{equation}
We are looking for $w$ such that $r_{0}\equiv0$ or at least $r_{0}$ is ``small'' in some sense. The operator $\RR$ is a pseudo-differential operator of order $-1$.
We can also note that
$$
\Upsilon\del_{x}u=i u-i(1-\del_{xx})^{-1}u
$$
Since the second term is of order $-2$, we want to solve the equation
$$
2imw_{x} +(a_{0}^{(5)}-\hat{a}_{0}^{(5)})u\equiv0.
$$
This equation has solution if and only if
we define 
\begin{equation}\label{eq:3.6.5ham}
\hat{a}_{0}^{(5)}(\f):=\frac{1}{2\pi}\int_{\TTT}a_{0}^{(5)}(\f,x){\rm d}x,
\end{equation}
and it is real thanks to (\ref{eq:3.5.12ham}). Now, we define
\begin{equation}\label{eq:3.6.6ham}
w(\f,x):=i\frac{1}{2m}  \del_{x}^{-1}(a_{0}^{(5)}-\hat{a}_{0}^{(5)})(\f,x),
\end{equation}
that is a purely imaginary function. 
In this way we can conjugate the operator $\calL_{5}$ to an operator of the form
 $\oo\cdot\del_{\f}+iM\del_{xx}+iM_{1}\del_{x}+iM_{0}+O(\del_{x}^{-1})$ with the diagonal part of $M_{0}$ constant in the  space variable.
Unfortunately, this transformation in not symplectic. 
We reason as follow. Let $w=i(w+\bar{w}):=i a$ and consider the Hamiltonian function
$$
H(u,\bar{u})=\frac{1}{2}\int_{\TTT}-(a\Upsilon+\Upsilon a)u\cdot \bar{u}d x.
$$
Since the function $a$ is real, and the operator $\Upsilon : L^{2}(\TTT;\CCC)\to L^{2}(\TTT,\CCC)$
is self-adjoint, then the operator $a\Upsilon+\Upsilon a$ is self-adjont. As consequence the 
hamiltonian $H$ is real-valued on $L^{2}$. The corresponding (linear) vector field is
$$
\chi_{H}(u,\bar{u})=-i\left(\begin{matrix}\del_{\bar{u}}H\\ \del_{u}H\end{matrix}\right)
=\left(\begin{matrix}\frac{i}{2}(a\Upsilon+\Upsilon a)u \\ 
-\frac{i}{2}(a\Upsilon+\Upsilon a)\bar{u}\end{matrix}
\right).
$$
Then, the 1-flow of $\chi_{H}$ generates a symplectic transformation of coordinates, given by
\begin{equation}\label{megaop}
\begin{aligned}
\TT_{6}&:=\exp(\chi_{H}(u,\bar{u})):=\left(\begin{matrix}e^{i\chi} & 0 \\ 0& e^{-i\chi} \end{matrix}\right)\left(\begin{matrix}u\\ \bar{u}\end{matrix}
\right),\\
& e^{i\chi}u:=\left(\sum_{m=0}^{\infty}\frac{1}{m!}\Big(\frac{1}{2}(a\Upsilon+\Upsilon a)\Big)^{m}\right)u.
\end{aligned}
\end{equation}
We can easily check that, the operators in (\ref{megaop}) and (\ref{eq:3.6.2ham}) differs only for an operator
of order $O(\del_{x}^{-2})$. Indeed one has
\begin{equation}\label{megaop2}
\begin{aligned}
e^{i\chi}u&=u+\frac{i}{2}(a\Upsilon+\Upsilon a)u+O(\del_{x}^{-2})=u+\frac{i}{2}a\Upsilon u+
\frac{i}{2}\Upsilon(au)+O(\del_{x}^{-2})\\
&=u+\frac{i}{2}a\Upsilon u+\frac{i}{2}a \Upsilon u
+\frac{i}{2}\frac{1}{i}\del_{x}\Big((1-\del_{xx})^{-1}a_{xx}(1-\del_{xx})^{-1}u\\
&+2(1-\del_{xx})^{-1}a_{x}(1-\del_{xx})^{-1}\del_{x}u\Big)+O(\del_{x}^{-2})\\
&=(\uno+ia\Upsilon)u+O(\del_{x}^{-2}).
\end{aligned}
\end{equation}
In (\ref{megaop2}) we essentially studied the commutator of the pseudo-differential operator 
$(1-\del_{xx})^{-1}$ with the operator of multiplication by the function $a$. Since the transformation
$\TT_{6}$ is symplectic we obtain the hamiltonian operator
\begin{eqnarray}\label{pippopippoham}
\calL_{6}&=&\TT_{6}^{-1}\calL_{5}\TT_{6}=\tilde{\calL}_{5}+\tilde\RR,\\
\tilde{\calL}_{5}&:=&\oo\cdot\del_{\f}\uno+im_{2}E\del_{xx}+
i\left(\begin{matrix}m_{1} \!\!\!& 0 \\ 0 & \!\!\!-\bar{m}_{1}\end{matrix}\right)\del_{x}
+imE+i \left(\begin{matrix}\hat{a}_{0}^{(5)}(\f) & b_{0}^{(5)} \\ -\bar{b}_{0}^{(5)} & -\hat{a}_{0}^{(5)}(\f)\end{matrix}\right),\nonumber\\
\tilde\RR&:=&\TT_{6}^{-1}\left[\calL_{5}\TT_{6}-\TT_{6}\tilde{\calL}_{5}
\right],\nonumber
\end{eqnarray}
where $\RR$ is hamiltonian and of order $O(\del_{x}^{-1})$.
\begin{rmk}\label{rmk55}
 Here we have that $\calL_{6}$ has the form \eqref{elleham} where
  $b_{2}^{(6)}=b_{1}^{(6)}\equiv 0$, $a_{0}^{(6)}:=\hat{a}_{0}^{(5)}$, $b_{0}^{(6)}:=b_{0}^{(5)}$ and $\RR_{6}:=\tilde\RR$.
\end{rmk}

\paragraph{Step 7. Descent Method: conjugation by multiplication  operator II}

In this section we want to eliminate the dependance on the time variable of the coefficients $\hat{a}_{0}^{(5)}(\f)$ in (\ref{eq:3.6.5ham}).

Consider the operator
\begin{equation}\label{eq:3.7.1ham}
\TT_{7}:=\left(\begin{matrix} 1+k(\f) & 0 \\ 0 & 1+\bar{k}(\f) \end{matrix}\right),
\end{equation}
with $k : \TTT^{d}\to \CCC$.
By direct calculation we have that
\begin{equation}\label{eq:3.7.2ham}
\begin{aligned}
\!\!\!\!\!\!\!\calL_{6}\TT_{7}-\TT_{7}&\left[\oo\cdot\del_{\f}\uno+i m_{2}E+
i\left(\begin{matrix}m_{1} \!\!\!& 0 \\ 0 & \!\!\!-\bar{m}_{1}\end{matrix}\right)\del_{x}
+i \left(\begin{matrix}m_{0}\!\!\! & 0 \\ 0 &\!\!\! -m_{0} \end{matrix}\right)
\right]\\
&=i\left(\begin{matrix}r_{1} & 0 \\ 0 & -\bar{r}_{1} \end{matrix}\right)+
\left[i\left(\begin{matrix}0 & b_{0}^{(6)} \\ -\bar{b}_{0}^{(6)} & 0 \end{matrix}\right)+
\RR_{6}
\right]\TT_{7},
\end{aligned}
\end{equation}
where
\begin{equation}\label{eq:3.7.3ham}
r_{1}(\f)=\oo\cdot\del_{\f}k(\f)+i({a}_{0}^{(6)}(\f)-m_{0})(1+k(\f)),
\end{equation}
We are looking for $\Gamma$ such that $r_{1}\equiv0$. As done in step 5,
we write $1+k(\f)=\exp(\Gamma(\f))$, then equation $r_{1}\equiv0$ reads
\begin{equation}\label{eq:3.7.4ham}
\oo\cdot\del_{\f}\Gamma(\f)+i({a}_{0}^{(6)}(\f)+m-m_{0})=0,
\end{equation}
that has a unique solution if and only if we define
\begin{equation}\label{eq:3.7.5ham}
m_{0}:=m+\frac{1}{(2\pi)^{d}}\int_{\TTT^{d}}{a}_{0}^{(6)}(\f){\rm d}\f.
\end{equation}
Hence we can set
\begin{equation}\label{eq:3.7.6ham}
\Gamma(\f):=-i(\oo\cdot\del_{\f})^{-1}({a}_{0}^{(6)}+m-m_{0})(\f).
\end{equation}
It turns out that  the trasformation $\TT_{7}$ is 
 invertible, then, by conjugation,  we obtain $\calL_{7}:=\TT_{7}^{-1}\calL_{6}\TT_{7}$
with
\begin{equation}\label{eq:3.7.7ham}
\calL_{7}:=
\oo\cdot\del_{\f}\uno+
i\left(\begin{matrix}m \!\!\!& 0 \\ 0 & \!\!\!-m\end{matrix}\right)\del_{xx}+
i\left(\begin{matrix}m_{1} \!\!\!& 0 \\ 0 & \!\!\!-\bar{m}_{1}\end{matrix}\right)\del_{x}
+i \left(\begin{matrix}m_{0} \!\!\!& b_{0}^{(7)} \\ -\bar{b}_{0}^{(7)} & \!\!\!-m_{0}\end{matrix}\right)+
\RR_{7}
\end{equation}
where 
we have defined
\begin{equation}\label{eq:3.7.8ham}
\begin{aligned}
 \left(\begin{matrix}0 \!\!\!& b_{0}^{(7)} \\ -\bar{b}_{0}^{(7)} & \!\!\! 0\end{matrix}\right)&:=
 \TT_{7}^{-1} \left(\begin{matrix}0 & b_{0}^{(6)} \\ -\bar{b}_{0}^{(6)} & 0 \end{matrix}\right)\TT_{7},\quad
{\RR}_{7}:= \TT_{7}^{-1}\RR_{6}\TT_{7}.
\end{aligned}
\end{equation}
Moreover, since by (\ref{eq:3.7.6ham}) the function $\Gamma$ is purely imaginary, then
the transformation is symplectic. Indeed
\begin{equation}\label{eq:3.7.9}
\Omega(e^{\Gamma}u, e^{\Gamma}v):={\rm Re}\int_{\TTT}ie^{\Gamma}ue^{-\Gamma}\bar{v}{\rm d}x=
\Omega(u,v),
\end{equation}
hence the linearized operator $\calL_{7}$ is Hamiltonian.

\subsection{Non-degeneracy Condition}
Here we give the proof of formula \eqref{mammamia}
Let us study the properties of the average of the coefficients of the first order differential operator.
In particular we are interested in how these quantities depends explicitly on $\oo$, see Remark \ref{trenonapoli}.
Consider $a_{1}(\f,x)=a_{1}(\f,x,u)$ where $u$ satisfies \eqref{eq:3.2.1bham} and $a_{i}$ is defined in \eqref{5005}.
One has 
\begin{equation}\label{topotopo}
\left|\int_{\TTT^{d+1}}a_{1}(\f,x)\right|\geq \e \gote-C\e\|u\|_{\gots_{0}+\h_{1}}\geq \frac{\gote}{2}\e,
\end{equation}
if $\e\g_{0}^{-1}$ is small enough. Essentially, by using \eqref{eq:B15cham}, \eqref{eq:B15dham}, \eqref{eq:3.10aham} and \eqref{eq:3.10bham}, one can repeat the reasoning followed in \eqref{topotopo} for the average of $a_{1}^{(i)}$ for $i=1,2,3,4$ and prove the \eqref{mammamiaa} with a constant $c< \gote/16$. Let us check \eqref{mammamiab}.
At the starting point there is no explicit dependence on the parameters $\oo$ in $a_{1}$, hence
we get also for $\oo_{1}\neq\oo_{2}$
\begin{equation}\label{topotopo1}
0=\left|\int_{\TTT^{d+1}}a_{1}(\f,x,\oo_{1},u(\oo))-a_{1}(\f,x,\oo_{2},u(\oo))\right|\leq \e^{2}C|\oo_{1}-\oo_{2}|. 
\end{equation}
Now, by \eqref{eq:3.122} one has that
$$
a_{1}^{(1)}(\f,x,\oo,u(\oo)):=a_{1}^{(1)}(\f,x,u(\oo)):=i(2\TT_{1}^{-1}(E+A_{2})\del_{x}\TT_{1}+\TT_{1}^{-1}A_{1}\TT_{1})_{1}^{1},
$$
and again we do not have explicit dependence on $\oo$ since the matrix $\TT_{1}$ depends on the external parameters only trough the function $u$. Hence bound \eqref{topotopo1} holds.
\noindent
Now consider the coefficients $a_{1}^{(2)}$ in \eqref{24ham}. There is explicit dependence on $\oo$ only in the term
\begin{equation}\label{topotopo2}
\TT_{2}^{-1}(\oo\cdot\del_{\f}\x)=\sqrt{1+\hat{\x_{y}}(\f,y)}{\oo}\cdot\del_{\f}\x(\f,y+\hat{\x}(\f,y)).
\end{equation}
Recall that the functions $\x$ in \eqref{299} and $\hat{\x}$ depends on $\oo$ only through $u$. Hence one has
\begin{equation}\label{topotopo3}
\begin{aligned}
&\left|\int_{\TTT^{d+1}}\sqrt{1+\hat{\x}_{y}}(\oo_{1}-\oo_{2})\cdot\del_{\f}\x(\f,y+\hat{\x}(\f,y)){\rm d}\f {\rm d}y
\right|=\\
&=\left|\int_{\TTT_{d+1}}\frac{(\oo_{1}-\oo_{2})\cdot\del_{\f}\x(\f,x)}{\sqrt{1+\hat{\x}_{y}(\f,x+\x(\f,x)) }}{\rm d}\f{\rm d}x
\right|\\
&\leq|\oo_{1}-\oo_{2}|\left|\int_{\TTT^{d+1}}\del_{\f}\x(\f,x){\rm d}\f {\rm d}x
\right|\\
&+|\oo_{1}-\oo_{2}|\left|\int_{\TTT_{d+1}}\left(\frac{1}{\sqrt{1+\hat{\x}_{y}}}-1\right)\del_{\f}\x(\f,x)
{\rm d}\f{\rm d}x
\right|
\end{aligned}
\end{equation}

By defining $|u|_{s}^{\infty}:=||u||_{W^{s,\infty}}$ and using the standard estimates of the Sobolev embedding
on the function $\x$ in \eqref{299}
we get
\begin{subequations}
\begin{align}
|\x|^{\infty}_{s}\leq C(s)||\x||_{s+\gots_{0}}&\leq C(s)||\rho_{0}||_{s+\gots_{0}}\leq
\e C(s)(1+||{ u}||_{s+\gots_{0}+2}),\label{eq:3.50aham}.
\end{align}
\end{subequations}
The function $\hat{\x}$ satisfies the same bounds by Lemma \ref{change}. Hence, since the first integral in \eqref{topotopo3} is zero, using the interpolation estimates in Lemma \ref{A}, we get 
\begin{equation}\label{topotopo4}
\left|\int_{\TTT^{d+1}}a_{1}^{(2)}(\f,x){\rm d}\f{\rm d}x\right|^{lip}\leq C \e^{2}.
\end{equation}
Let us study the coefficients $a_{1}^{(3)}$ defined in \eqref{eq:3.166}. 
In particular one need to control the difference $a_{1}^{(3)}(\oo_1)-a_{1}^{(3)}(\oo_2)$. To do this one can uses standard formul\ae \; 
of propagation of errors for Lipschitz functions. 
In order to perform the quantitative estimates one can check that
the function $\al({\f})$ defined in (\ref{388}) satisfies the tame estimates (see also Lemma $3.20$ in \cite{FP} ):
\begin{subequations}\label{eq:B24ham}
\begin{align}
|\al|^{\infty}_{s }&\leq\e \g_{0}^{-1}C(s)(1+||{ u}||_{s+d+\gots_{0}+2}),\label{eq:B24aham}\\
|d_{{ u} }\al({ u})[{ h}]|^{\infty}_{s }&\leq \e \g_{0}^{-1}C(s)
(||{ h}||_{s+d+\gots_{0}+2}+
||{ u}||_{s+d+\gots_{0}+2}||{ h}||_{d+\gots_{0}+2}),\label{eq:B24bham}\\
|\al|^{\infty}_{s,\g} &\leq \e\g_{0}^{-1}C(s)
(1+||{ u}||_{s+d+\gots_{0}+2,\g}),\label{eq:B24cham}
\end{align}
\end{subequations}
\noindent
while by the (\ref{333}) one has $\rho=1+\TT_{3}^{-1}(\oo\cdot\del_{\f}\al)$. By using Lemma \ref{change}
and the bounds (\ref{eq:B24ham}) on $\al$ and (\ref{eq:3.2.1bham}) one can prove
\begin{subequations}\label{eq:B34ham}
\begin{align}
|\rho-1|^{\infty}_{s,\g}&\leq
 \e\g_{0}^{-1}C(s)(1+||{ u}||_{s+d+\gots_{0}+4,\g}) \label{eq:B34bham}\\
|d_{{ u} }\rho({ u})[{ h}]|^{\infty}_{s}&\leq 
\e\g_{0}^{-1}C(s)(||{ h}||_{s+d+\gots_{0}+3}+
||{u}||_{s+d+\gots_{0}+4}||{ h}||_{d+\gots_{0}+3}).\label{eq:B34cham}
\end{align}
\end{subequations}
The bounds above follows by classical tame estimates in Sobolev spaces, anyway the proof can be found 
in Section 3 of \cite{FP}. Now by taking the  integral of \eqref{topotopo4} and by using \eqref{eq:B24aham}-\eqref{eq:B34cham}, the tame estimates in Lemma \ref{change}
and the \eqref{eq:3.10aham}, \eqref{eq:3.10bham} one obtain the result on the . For the last step one can reason in the same way. Indeed the most important fact is to prove \eqref{topotopo4}. At the starting point we have no explicit dependence on $\la$ in the average of $a_{1}$, but, once that dependence appear, then we have the estimates \eqref{topotopo4} that is quadratic in $\e$. 


One has also the following result.
\begin{lemma}\label{lem:3.9ham}
Under the Hypotheses of Lemma \ref{lem:3.88} possibly with smaller $\epsilon_{0}$, 
if 
(\ref{eq:3.2.1bham}) holds, one has that the $\TT_i$, $i\neq3 $ identify operators $\TT_{i}(\f)$, 
 of the phase space ${\bf H}^{s}_{x}:={\bf H}^{s}(\TTT)$. Moreover they are invertible
and the following estimates hold for $\gots_{0}\leq s\leq q-\h_{1}$ and i=1,2,4,5,6,7:
\begin{subequations}
\begin{align}
||(\TT_{i}^{\pm1}(\f)-\uno){ h}||_{{\bf H}^{s}_{x}}&\leq \e 
C(s)(||{h}||_{{\bf H}^{s}_{x}}+||{ u}||_{s+d+2\gots_{0}+4}
||{ h}||_{{\bf H}^{1}_{x}}),\label{eq:3.93f}
\end{align}
\end{subequations}
\end{lemma}

The Lemma is essentially a consequence if the discussion above. We omit the details because the proof follows basically the same arguments used in 
Lemma $3.25$ in \cite{FP}.

\zerarcounters
\section{Reduction to constant coefficients}\label{sec:4ham}

In this Section we conclude the proof of Proposition \ref{teo2ham} through a  
reducibility algorithm. First we need to fix some notations.
Let $b\in \NNN$, we consider the exponential basis $\{ e_{i} : i\in \ZZZ^{b}\}$ 
of $L^{2}(\TTT^{b})$. In this way we have that $L^{2}(\TTT^{2})$ is the space
$\{u=\sum u_{i}e_{i} : \sum|u_{i}|^{2}<\infty\}$. A linear operator $A : L^{2}(\TTT^{b})\to L^{2}(\TTT^{b})$ 
can be written as an infinite dimensional matrix
$$
A=(A_{i}^{j})_{i,j\in\ZZZ^{b}}, \quad A_{i}^{j}=(Ae_{j},e_{i})_{L^{2}(\TTT^{b})}, \quad Au=\sum_{i,j}A_{i}^{j}u_{j}e_{i}.
$$

\noindent
where $(\cdot,\cdot)_{L^{2}(\TTT^{d+1})}$ is the usual scalar product on $L^{2}$.
In the following we also use the decay norm
\begin{equation}\label{decayham}
\begin{aligned}
|A|^{2}_{s}:=\sup_{\s,\s'\in \CC}|A_{\s}^{\s'}|^{2}_{s}&:=
\sup_{\s,\s'\in\CC}\sum_{h\in\ZZZ\times\ZZZ^{d}}\langle h\rangle^{2s}
\sup_{k-k'=h}|A_{\s,k}^{\s',k'}|^{2}.
\end{aligned}
\end{equation}

%
If one has that $A:=A(\oo) $ depends on  parameters  $\oo\in\Lambda\subset\RRR$ in a Lipschitz way,  we define
\begin{equation*}\label{2.1}
\begin{aligned}
|A|_{s}^{sup}&:=\sup_{\la\in\Lambda}|A(\la)|_{s}, \;\;\;
|A|^{lip}_{s}:=\sup_{\la_{1}\neq\la_{2}}\frac{|A(\la_{1})-A(\la_{2})|_{s}}{|\la_{1}-\la_{2}|}, \;\\
|A|_{s,\g}&:=|A|_{s}^{sup}+\g|A|^{lip}_{s}.
\end{aligned}
\end{equation*}
%
%
%
%

The decay norm we have introduced in (\ref{decayham}) is suitable for the problem we are studying.
 Note that
\begin{equation*}
\;\;\;
\forall\; s\leq s' \;\; \Rightarrow \;\; |A_{\s}^{\s'}|_{s}\leq|A_{\s}^{\s'}|_{s'}.
\end{equation*}
Moreover norm (\ref{decayham}) gives information on the polynomial off-diagonal decay of the matrices,
indeed $\forall\; k,k'\in\ZZZ_+\times\ZZZ^{d}$
\begin{equation}\label{decay2}
\begin{aligned}
&|A_{\s,k}^{\s,k'}|\leq \frac{|A_{\s}^{\s'}|_{s}}{\langle k-k' \rangle^{s}},\quad 
|A_{i}^{i}|\leq |A|_{0}, \quad |A_{i}^{i}|^{lip}\leq |A|_{0}^{lip}.
\end{aligned}
\end{equation}

In order to prove Prposition \ref{teo2ham} we first prove the following result. We see that
the operator $\calL_{7}$ cannot be diagonalized, but anyway can be block-diagonalized where 
the blocks on the diagonal have fixed size. This is sufficient for our analysis.

\begin{theorem}\label{KAMalgorithmham}
Let $\ff\in C^{q}$ satisfy the Hypotheses of Proposition \ref{teo2ham} with $q>\h_{1}+\be+\gots_{0}$
where $\h_{1}$ defined in (\ref{eq:3.2.0ham}) 
and $\be=7\tau+5$ for some $\tau>d$.
Let $\g\in(0,\g_{0})$,
$\gots_{0}\leq s\leq q-\h_{1}-\be$ and ${ u}(\la)\in{\bf H}^{0}$ be a family of functions depending 
on a Lipschitz way on a parameter
$\oo\in\Lambda_{o}\subset\Lambda:[1/2,3/2]$. 
Assume that
\begin{equation}\label{eq:4.2ham}
||{u}||_{\gots_{0}+\h_{1}+\be, \Lambda_{o},\g}\leq1.
\end{equation}
Then there exist constants $\epsilon_{0}$, $C$, depending only on the data of the problem,
such that,
if $\e\g^{-1}\leq\epsilon_{0}$,
then there exists a sequence of purely imaginary numbers as in Proposition \ref{teo2ham},
namely
$\Omega_{\s,j}^{\phantom{g} j}, \Omega_{\s,j}^{-j} : \Lambda \to \CCC$
 of the form
 \begin{equation}\label{1.2.2bisham}
\begin{aligned}
\Omega_{\s,j}^{\phantom{g} j}&:=-i \s m_{2}j^{2}-i\s|m_{1}|j+i\s m_{0}+i\s r_{j}^{j},\\
\Omega_{\s,j}^{-j}&:=i\s r_{j}^{-j},
\end{aligned}
\end{equation}
where
\begin{equation}\label{1.2.2trisham}
m_{2},m_{0}\in\RRR, \;\; m_{1}\in i\RRR, \quad \ol{r_{j}^{k}}=r_{k}^{j}, 
\;\; k=\pm j
\end{equation}
for any $\s\in\CC$, $j\in \NNN$, moreover
\begin{equation}\label{eq:4.5ham}
|r_{\s,j}^{\phantom{g} k}|_{\g}\leq \frac{\e C}{\langle j\rangle}, \quad \forall\; \s\in\CC,\; j\in\ZZZ, \;\; k=\pm j,
\end{equation}
and
such that, for any $\oo\in \Lambda_{\infty}^{2\g}({ u})$, defined in (\ref{martina10ham}),
there exists a bounded, invertible linear operator
$\Phi_{\infty}(\oo) : {\bf H}^{s}\to {\bf H}^{s}$, with bounded inverse $\Phi_{\infty}^{-1}(\oo)$,
such that
\begin{equation}\label{eq:4.6ham}
\begin{aligned}
\calL_{\infty}(\oo):=\Phi_{\infty}^{-1}(\oo)\circ\calL_{7}\circ\Phi_{\infty}(\oo)&=
{\oo}\cdot\del_{\f}\uno+i \DD_{\infty},\\
{\rm where}\;\;\;\;\;\;\;\;\;
 \DD_{\infty}:={\rm diag_{h=(\s,j)\in\CC\times\NNN}}\{\Omega_{\s, \und{j}}(\oo)\},& \quad 
\end{aligned}
\end{equation}
with $\calL_7$ defined in \eqref{eq:3.5.9ham} and where
\begin{equation}\label{bambaham}
\Omega_{\s,\underline{j}}:=\left(\begin{matrix}\Omega_{\s,j}^{\phantom{g}j} & \Omega_{\s,j}^{-j} \\
\Omega_{\s,-j}^{\phantom{g}j} & \Omega_{\s,-j}^{-j} \end{matrix}\right)
\end{equation}
Moreover, the transformations $\Phi_{\infty}(\la)$, $\Phi_{\infty}^{-1}$ are symplectic and satisfy
\begin{equation}\label{eq:4.8ham}
|\Phi_{\infty}(\la)-\uno|_{s,\Lambda_{\infty}^{2\g},\g}+
|\Phi_{\infty}^{-1}(\la)-\uno|_{s,\Lambda_{\infty}^{2\g},\g}\leq
\e \g^{-1} C(s)(1+||{ u}||_{s+\h_{1}+\be,\Lambda_{o},\g}).
\end{equation}
In addition to this, for any $\f\in\TTT^{d}$, for any $\gots_{0}\leq s\leq q-\h_{1}-\be$
the operator
$\Phi_{\infty}(\f) : {\bf H}^{s}_{x}\to{\bf H}^{s}_{x}$ is an invertible operator
of the phase space ${\bf H}_{x}^{s}:={\bf H}^{s}(\TTT)$ with inverse
$(\Phi_{\infty}(\f))^{-1}:=\Phi_{\infty}^{-1}(\f)$ and
\begin{equation}\label{eq:4.9ham}
||(\Phi_{\infty}^{\pm1}(\f)-\uno){ h}||_{{\bf H}^{s}_{x}}\leq
\e\g^{-1} C(s)(||{ h}||_{{\bf H}^{s}_{x}}+||{ u}||_{s+\h_{1}+\be+\gots_{0}}
||{ h}||_{{\bf H}^{1}_{x}}).
\end{equation}
\end{theorem}

\begin{rmk}
Note that since the $\Phi_{\infty}$ is symplectic then the operator $\calL_{\infty}$ is hamiltonian.
\end{rmk}
 
The main point of the Theorem \ref{KAMalgorithmham} is that 
the bound on the low norm of $u$ in (\ref{eq:4.2ham}) guarantees the bound 
on \emph{higher} norms (\ref{eq:4.8ham}) for the transformations $\Phi_{\infty}^{\pm1}$. This is 
fundamental in order  to get the estimates on the inverse of $\calL$ in high norms.

Moreover, the definition (\ref{martina10ham}) of the set where the second Melnikov conditions
hold, depends only on the final eigenvalues. Usually in KAM theorems, the non-resonance conditions
have to be checked, inductively, at each step of the algorithm. This 
formulation, on the contrary, allow us to discuss the measure estimates only once. 
Indeed, the functions $\m_{h}(\oo)$ are well-defined even if
$\Lambda_{\infty}=\emptyset$, so that, we will perform the measure estimates as the last step of the proof of Theorem \ref{teo1}.

 \subsection{Functional setting and notations}

\subsubsection{The off-diagonal decay norm }

Here we want to show some important properties of the norm $|\cdot|_{s}$. Clearly the same results hold for the 
norm $|\cdot|_{{\bf H}^{s}}:=|\cdot|_{H^{s}\times H^{s}}$. Moreover we will introduce 
some characterization  of the operators we have to deal with during the diagonalization procedure.
 
First of all we have following classical results.

\begin{lemma}\label{bubbole}{\bf Interpolation.} For all $s\geq s_{0}>(d+1)/2$ 
there are $C(s)\geq C(s_{0})\geq1$ 
such that 
if $A=A(\oo)$ and $B=B(\oo)$ depend on the parameter
 $\la\in\Lambda\subset\RRR$
in a Lipschitz way, then
\begin{subequations}
\begin{align}
|AB|_{s,\g}&\leq C(s)|A|_{s_0,\g}|B|_{s,\g}
+C(s_{0})|A|_{s,\g}|B|_{s_0,\g},\label{eq:2.11a}\\
|AB|_{s,\g}&\leq C(s)|A|_{s,\g}|B|_{s,\g}.\label{eq:2.11b}\\
\|Ah\|_{s,\g}&\leq C(s)(|A|_{s_0,\g}\|h\|_{s,\g}
+|A|_{s,\g}\|h\|_{s_0,\g}),\label{eq:2.13b}\end{align}
\end{subequations}
\end{lemma}

Lemma \ref{bubbole} implies that for any $n\geq0$ and $s\geq \gots_{0}$ one has
\begin{equation}\label{eq:2.12}
|A^{n}|_{s_{0},\g}\leq [C(s_{0})]^{n-1}|A|^{n}_{s_{0},\g},  \quad 
|A^{n}|_{s,\g}\leq n[C(s_{0})|A|_{\gots_{0},\g}]^{n-1}C(s)|A|_{s,\g}, 
\end{equation}
\begin{equation}\label{eq:2.12bis}
|[A,B]^{n}|_{s,\g}\leq nC(\gots_{0})^{n-1}|A|_{\gots_{0},\g}^{n-1}|B|^{n-1}_{\gots_{0},\g}
\left(|A|_{s,\g}|B|_{\gots_{0},\g}+|A|_{\gots_{0},\g}|B|_{s,\g}\right), 
\end{equation}

The following Lemma shows how to invert linear operators which 
are ''\emph{near}'' to the identity in norm $|\cdot|_{s}$. 

\begin{lemma}\label{inverse} Let $C(s_0)$ be as in Lemma \ref{bubbole}. Consider an operator of the form $\Phi=\uno+\Psi$
where $\Psi=\Psi(\la)$ depends in a Lipschitz way on $\la\in\Lambda\subset\RRR$. 
Assume that $C(s_{0})|\Psi|_{s_{0},\g}\leq1/2$. Then $\Phi$ is invertible and,
for all $s\geq s_{0}\geq (d+1)/2$, 
\begin{equation}\label{eq:2.14}
 \quad |\Phi^{-1}|_{s_{0},\g}\leq 2,\quad
|\Phi^{-1}-\uno|_{s,\g}\leq C(s)|\Psi|_{s,\g}
\end{equation}
Moreover, if one  has $\Phi_{i}=\uno+\Psi_{i}$, $i=1,2$ such that $C(s_{0})|\Psi_{i}|_{s_0,\g}\leq1/2$, then
\begin{equation}\label{eq:2.15}
\!\!\!|\Phi^{-1}_{2}-\Phi^{-1}_{1}|_{s,\g}\leq\! C(s)\left(|\Psi_{2}-\Psi_{1}|_{s,\g}\!+\!(|\Psi_{1}|_{s,\g}+|\Psi_{2}|_{s,\g})
|\Psi_{2}-\Psi_{1}|_{s_0,\g}\right).
\end{equation}
\end{lemma}
\begin{proof} See \cite{FP}.
\end{proof}
%

\subsubsection{T\"opliz-in-time matrices}

We introduce now a special class of operators, the so-called 
{\em T\"opliz in time} matrices, i.e.
\begin{equation}\label{eq:2.16}
A_{i}^{i'}=A_{(\s,j,p)}^{(\s',j',p')}:=A_{\s,j}^{\s'j'}(p-p'), 
\quad {\rm for} \quad i,i'\in \CC\times\ZZZ\times\ZZZ^{d}.
\end{equation}
To simplify the notation in this case, we shall write
$A_{i}^{i'}=A_{k}^{k'}(\ell)$,
 $i=(k,p)=(\s,j,p)\in \CC\times\ZZZ\times \ZZZ^{d}$, 
$i'=(k',p')=(\s',j',p')\in \CC\times \ZZZ\times\ZZZ^{d}$,
 with  $k,k'\in  \CC\times\ZZZ$.

They are relevant because
one can identify the matrix $A$ with a one-parameter family of operators, acting on the space
${\bf H}^{s}_{x}$,
which depend on the time, namely
$$
A(\f):=(A_{\s,j}^{\s',j'}(\f))_{\substack{\s,\s'\in \CC \\ j,j'\in\ZZZ}}, \quad 
A_{\s,j}^{\s',j'}(\f):=\sum_{\ell\in\ZZZ^{d}}A_{\s,j}^{\s',j'}(\ell)e^{i\ell\cdot\f}.
$$

To obtain the stability result on the solutions we will strongly use this property.

\begin{lemma}\label{1.4} If $A$ is a T\"opliz in time matrix as in (\ref{eq:2.16}), and 
$\gots_{0}:=(d+2)/2$, then one has
\begin{equation}\label{aaaaa}
|A(\f)|_{s}\leq C(\gots_{0})|A|_{s+\gots_{0}}, \quad \forall \; \f\in\TTT^{d}.
\end{equation}

\end{lemma}

\begin{proof} See \cite{FP} or \cite{BBM}.
\end{proof}

\begin{defi}\label{def:smooth} {\bf (Smoothing operator)}
Given $N\in\NNN$, we the define the \emph{smoothing operator} $\Pi_{N}$ as
\begin{equation}\label{smoothop}
(\Pi_{N}A)_{\s,j,\ell }^{\s',j',\ell'}=\left\{
\begin{aligned}
& A_{\s,j,\ell}^{\s',j',\ell} \,,\quad |\ell-\ell'|\leq N,\\
& 0 \quad  {\rm otherwise}
\end{aligned}\right.
\end{equation} 
\end{defi}

\begin{lemma}

Let $\Pi_{N}^{\perp}:=\uno-\Pi_{N}$,

 if $A=A(\la)$ is a Lipschitz family $\la\in\Lambda$, then
\begin{equation}\label{eq:2.22}
|\Pi_{N}^{\perp}A|_{s,\g}\leq N^{-\be}|A|_{s+\be,\g}, \quad \be\geq0.
\end{equation}
\end{lemma}
\begin{proof}
 See \cite{FP} or \cite{BBM}.
\end{proof}

\begin{lemma}\label{multiop}
Consider $a=\sum_{i}a_{i}e_{i}\in H^{s}(\TTT^{b})$. Then the multiplication operator by the function $a $, i.e.
$h\mapsto a h$ is represented by the matrix $A$ defined as $A_{i}^{i'}=a_{i-i'}$. One has
\begin{equation}\label{multiop1} 
|A|_{s}=||a||_{s}.
\end{equation}
Moreover, if $a=a(\la)$ is a Lipschitz family of functions, then
\begin{equation}\label{multiop2} 
|A|_{s,\g}=||a||_{s,\g}.
\end{equation}
\end{lemma}

We need some technical lemmata on finite dimensional matrices.
\begin{lemma}\label{finitema}
Given a matrix $M\in\MM_{n}(\CCC)$, where $\MM_{n}(\CCC)$ is the space of the $n\times n$ matrix with coefficients in $\CCC$,
we define the norm $\|M\|_{\infty}:=\max_{i,j=1,\ldots,n} \{A_{i}^{j}\}$. One has
\begin{equation}\label{finitem1}
\|M\|_{\infty}\leq \| M\|_{2}\leq n\|M\|_{\infty},
\end{equation}
where $\|\cdot\|_{2}$ is the $L^{2}-$operatorial norm.
\end{lemma}
\begin{proof}
It follow straightforward by the definitions.
\end{proof}

\begin{lemma}\label{finitema3}
Take two self adjoint matrices $A,B\in\MM_{n}(\CCC)$. Let us define the operator $M : \MM_{n}(\CCC)\to \MM_{n}(\CCC)$
\begin{equation}\label{finitema4}
M : C \mapsto MC:= AC-CB.
\end{equation}
Let $\la_{j}$ and $\be_{j}$ for $j=1,\ldots,n$ be the eigenvalues respectively of $A$ and $B$. Then,
for any $R\in \MM_{n}(\CCC)$
one has that the equation $MC=R$ has a solution with
\begin{equation}\label{finitema5}
\|C\|_{\infty}\leq K \left(\min_{i,j=1,\ldots,n}\{\la_{j}-\be_{i}\}\right)^{-1}\|R\|_{\infty},
\end{equation}
where the constant $K$ depends only on $n$.
\end{lemma}

\begin{proof}
Define the operator $\TT : \MM_{n}(\CCC) \to \CCC^{n^{2}}$ that associate to a matrix the  vector of its components.
Then the equation $MC=R$ can be rewritten as 
$$
(A\otimes \uno-\uno\otimes B^{T})\TT(C)=\TT(R), 
$$
where $\uno$ is the $n\times n$ identity. Then, by using Lemma \ref{finitema}, one has
\begin{equation}\label{finitema6}
\begin{aligned}
\|C\|_{\infty}&=\max_{i=1,\ldots,n^{2}}\|[\TT(C)]_{i}\|_{\infty}\leq n \|(A\otimes \uno-\uno\otimes B^{T})^{-1}\|_{\infty}\max_{i=1,\ldots,n^{2}}|[\TT(R)]_{i}|\\
&\leq n^{2}\|(A\otimes \uno-\uno\otimes B^{T})^{-1}\|_{2}\max_{i=1,\ldots,n^{2}}|[\TT(R)]_{i}|\leq n^{2} c  \left(\min_{i,j=1,\ldots,n}\{\la_{j}-\be_{i}\}\right)^{-1}\|R\|_{\infty},
\end{aligned}
\end{equation}
that is the \eqref{finitema5}.
\end{proof}

\subsubsection{Hamiltonian operators}

Here we give a characterization, in terms of the Fourier coefficients, of hamiltonian linear operators.
This is important since we want to show that our algorithm is closed 
for such class of operators.

\begin{lemma}\label{lemFouham} Consider a linear operator $B:=(i\s R_{\s}^{\s'}) :{\bf H}^{s}\to {\bf H}^{s}$. Then,
$B$ is hamiltonian with respect to the symplectic form (\ref{simplectic}) if and only if
\begin{equation}\label{eq:2.25ham}
R_{\s,h}^{\s',h'}=R_{-\s',h'}^{-\s,h}, \qquad \ol{R_{\s,h}^{\s',h'}}=R_{\s',-h}^{\s,-h'}
\end{equation}


\end{lemma}

\begin{proof}
In coordinates, an Hamiltonian function for such operator, is a quadratic form real and symmetric,
$$ 
H=\sum_{\substack{\s,\s'\in\CC \\ h,h'\in\ZZZ^{d+1}}}Q_{\s,h}^{\s',h'}z^{\s}_{h}z^{\s'}_{h'}, 
$$ 
where we denote $\ol{z^{\s}_{h}}=z^{-\s}_{-h}$ and $h=(j,p), h'=(j',p')$. This means that, $Q$ satisfies
\begin{equation}\label{antonio}
\ol{Q_{\s,h}^{\s',h'}}=Q_{-\s,-h}^{-\s',-h'}, \quad {Q_{\s,h}^{\s',h'}}=Q_{\s',h'}^{\s,h}
\end{equation}
%
Now, since the hamiltonian vector field associated to the Hamiltonian $H$ is given by
$B=iJ Q$, then 
writing
$$
B=i\left(\begin{matrix}1&0 \\  0&-1\end{matrix}\right)\left(\begin{matrix}1&0 \\  0&-1\end{matrix}\right)JQ
$$
we set
$R_{\s}^{\s'}=Q_{-\s}^{\s'}$ follow the (\ref{eq:2.25ham}).
\end{proof}

Since the operator $\calL_{\infty}$ in Theorem \ref{KAMalgorithmham}  is hamiltonian,
thanks to the characterization in Lemma \ref{lemFouham}
 we can note that the blocks $\Omega_{\s,\und{j}}$ defined in 
\eqref{bambaham} as purely imaginary eigenvalues. 

\subsection{Reduction algorithm}

We prove Theorem \ref{KAMalgorithmham} by means of the following 
Iterative Lemma on the class of linear operators
\begin{defi}\label{classeham}
\begin{equation}\label{eq:4.11ham}
\oo\cdot\del_{\f}\uno+\DD+\RR : \; {\bf H}^{0}\to {\bf H}^{0},
\end{equation}
where $\oo$ is as in \eqref{dio}, and
\begin{equation}\label{eq:4.12ham}
\begin{aligned}
\DD=diag_{(\s,j)\in\CC\times\ZZZ }\{\Omega_{\s,\underline{j}} \}:&=diag_{(\s,j)\in\CC\times \ZZZ}
\left\{\left(\begin{matrix}\Omega_{\s,j}^{\phantom{g}j} & \Omega_{\s,j}^{-j} \\
\Omega_{\s,-j}^{j} & \Omega_{\s,-j}^{-j} \end{matrix}\right)\right\},
 \end{aligned}
\end{equation}
where
\begin{equation}\label{eq:4.12bisham}
\begin{aligned}
\Omega_{\s,j}^{\phantom{g} j}&:=-i \s m_{2}j^{2}-i\s|m_{1}|j+i\s m_{0}+i\s r_{j}^{j},\\
\Omega_{\s,j}^{-j}&:=i\s r_{j}^{-j},
\end{aligned}
\end{equation}
and
\begin{equation}\label{eq:4.12trisham}
\begin{aligned}
m_{2},m_{0}\in\RRR, \;\; m_{1}\in i\RRR, \quad \ol{r_{j}^{k}}=r_{k}^{j}, \quad r_{j}^{k}= O(\frac{\e}{\langle j\rangle})
\;\; k=- j, \;\;  r_{j}^{k}= O(\frac{\e}{\langle j\rangle}),\; k=j
\end{aligned}
\end{equation}
for any $(\s,j)\in\CC\times\NNN$,
with $\RR$ is a T\"opliz in time Hamiltonian operator such that
$\RR_{\s}^{\s}=O(\e\del_{x}^{-1})$ and $\RR_{\s}^{-\s}=O(\e)$ for $\s=\pm1$.
Moreover we set $\mu_{\s,j}$ for $\s\in\CC$ the eigenvalues of $\Omega_{\s,\und{j}}$.
\end{defi}
%

 Note that the operator
$\calL_{7}$ has the form (\ref{eq:4.11ham}) and satisfies the (\ref{eq:4.12ham}) and (\ref{eq:4.12bisham}) as well as the estimates
(\ref{eq:B15aham}) and (\ref{eq:B15bham}). 
Note moreover that for $\calL_{7}$ the matrix $\DD$ is completely diagonal.
This fact is not necessary for our analysis, and it cannot be preserved during the algorithm.

Define
\begin{equation}\label{eq:4.13ham}
N_{-1}:=1, \quad N_{\nu}:=N_{\nu-1}^{\chi}=N_{0}^{\chi^{\n}}, \;\;\forall\;\n\geq0, \;\;\chi=\frac{3}{2}.
\end{equation}
and
\begin{equation}\label{eq:4.14ham}
\al=7\tau+3, \qquad \h_{3}:=\h_{1}+\be,
\end{equation}
where $\h_{1}$ is defined in (\ref{eq:3.2.0ham}) and $\be=7\tau+5$.
 Consider
 $ \calL_7=\calL_0$.
Note that $\calL_{7}$ belongs to the class of Definition \ref{classeham}. Indeed in this case we have
that 
\begin{equation*}
\RR_{0}:=\left(\begin{matrix}0 & q_{0}(\f,x) \\ -\bar{q}_{0}(\f,x) & 0 \end{matrix}\right) + \RR,
\end{equation*}
(see \eqref{eq:3.5.9ham}) and $\RR$ is a pseudo differential operator of order $O(\del_{x}^{-1})$.
We have the following lemma:

\begin{lemma}\label{stimeRham}
The operator $\RR$ defined in Lemma \ref{lem:3.88} satisfies the bounds
\begin{subequations}\label{stimeR1ham}
\begin{align}
|\RR({ u})|_{s,\g}&\leq \e C(s) (1+||{ u}||_{s+\h_{1},\g}), \label{eq:B15aaham}\\
|d_{{ u}}\RR({ u})[{ h}]|_{s}&\leq \e C(s)\left(||{ h}||_{\gots_{0}+\h_{1}}+
||{ h}||_{s+\h_{1}}+||{ u}||_{s+\h_{1}}||{ h}||_{\gots_{0}}
\right),\label{eq:B15bbham}
\end{align}
\end{subequations}
where $\h_{1}$ is defined in Lemma \ref{lem:3.88}.
\end{lemma}

\begin{proof} By the proof of Lemma \ref{lem:3.88} we have that in the operator $\calL_{5}$
in \eqref{eq:3.5.9bisham} the remainder is just a multiplication operator by the functions $a_{0}^{(5)}, b_{0}^{(5)} $. Hence by Remark \ref{multiop} one has that 
the decay norm of the operator is finite. We need to check that the transformation $\TT_{6}$ has a finite decay norm.
First of all we have that the function $w$ in \eqref{eq:3.6.6ham} satisfies the following estimates:
\begin{equation}\label{stimeR2ham}
\begin{aligned}
||w||_{s,\g}&\leq_{s}\e(1+||u||_{s+\tau_{1},\g}),\\
||\del_{u}w(u)[h]||_{s}&\leq_{s} \e(||h||_{s+\tau_{1}}+||u||_{s+\tau_{1}}||h||_{\tau_{1}}),
\end{aligned}
\end{equation}
with $\tau_{1}$ a constant depending only on the data of the problem and much small than $\h_{1}$.\footnote{to prove Lemma \ref{lem:3.88} one prove bounds like \eqref{eq:3.2.7ham} and \eqref{eq:3.2.7bisham} on the coefficients of each $\calL_{i}$ with loss of regularity $\tau_{i}$ at each step. The constant $\h_{1}$ of the Lemma is obtained by collecting together the loss of regularity of each step.}

The operator $\tilde{S}=\uno+w\Upsilon $ defined in \eqref{eq:3.6.2ham}  satisfies the following estimates in norm $|\cdot|_{s}$ defined in \eqref{decayham}:
\begin{equation}\label{stimeR3ham}
\begin{aligned}
|\tilde{S}-\uno|_{s,\g}&\leq_{s}\e(1+\|u\|_{s+\tau_{1},\g}),\\
|\del_{u}\tilde{S}(u)[h]|_{s}&\leq_{s} \e(\|h\|_{s+\tau_{1}}+\|u\|_{s+\tau_{1}}\|h\|_{\tau_{1}}),
\end{aligned}
\end{equation}
The \eqref{stimeR3ham} follow by the \eqref{stimeR2ham} and the fact that $|\Upsilon|_{s}\leq1$ using Lemma \ref{multiop}.
Clearly also the transformation $\TT_{6}$ defined in \eqref{megaop} satisfies the same estimates as in 
\eqref{stimeR3ham}. Hence using  Lemma \ref{bubbole} one has that the remainder $\tilde{R}$ of the operator $\calL_{6}$
in \eqref{pippopippoham} satisfies bounds like \eqref{stimeR1ham} with a different constant $\tau_{2}$ (possibly greater than $\tau_{1}$) instead of $\h_{1}$. Now the last transformation $\TT_{7}$ is a multiplication operator, then, by using
again Lemmata \ref{bubbole} and  \ref{multiop} one obtain the \eqref{stimeR1ham} on the remainder of the operator 
$\calL_{7}$ in \eqref{eq:3.7.7ham}.
\end{proof}

\begin{lemma}\label{teo:KAMham}
Let $q>\h_{1}+\gots_{0}+\be$. There exist constant $C_{0}>0$, $N_{0}\in\NNN$ large, such that
if
\begin{equation}\label{eq:4.15ham}
N_{0}^{C_{0}}\g^{-1}|\RR_{0}|_{\gots_{0}+\be} \leq1,
\end{equation}
then, for any $\nu\geq0$:

\noindent
$({\bf S1})_{\nu}$ There exists operators
\begin{equation}\label{eq:4.16ham}
\calL_{\nu}:=\oo\cdot\del_{\f} \uno +\DD_{\nu}+\RR_{\nu},\;\;
\DD_{\nu}={\rm diag}_{h\in\CC\times\ZZZ}\{\Omega_{\s,\underline{j}}^{\nu}\},
\end{equation}
\noindent
where
\begin{equation}\label{eq:4.16pippoham}
\Omega_{\s,\underline{j}}^{\nu}(\oo)=
\left(\begin{matrix}\Omega_{\s,j}^{\nu,j} & \Omega_{\s,j}^{\nu,-j} \\
\Omega_{\s,-j}^{\nu,j} & \Omega_{\s,-j}^{,\nu-j} \end{matrix}\right), 
\end{equation}
\noindent
and
\begin{equation*}
\begin{aligned}
\Omega_{\s,j}^{\nu, j}&:=-i \s m_{2}j^{2}-i\s|m_{1}|j+i\s m_{0}+i\s r_{j}^{\nu,j}=:\Omega_{\s,j}^{0,j}+i\s r_{j}^{\nu,j},\\
\Omega_{\s,j}^{\nu,-j}&:=i\s r_{j}^{\nu,-j}=:\Omega_{\s,j}^{0,-j}+i\s r_{j}^{\nu,-j},
\end{aligned}
\end{equation*}
with $(\s,j)\in \CC\times\ZZZ$, and defined for   $\la\in\Lambda_{\nu}^{\g}:=\Lambda_{\nu}^{\g}$,
with
$\Lambda^{\g}_{0}:=\Lambda_{o}$ and for $\nu\geq1$, %
\begin{equation}\label{eq:419ham}
\begin{aligned}
\Lambda_{\nu}^{\g}&:=\calP_{\nu}^{\g}(u)\cap\calO_{\nu}^{\g},\\
\SSSS_{\nu}^{\g}(u)&:=
\left\{\oo\in\Lambda_{\nu-1}^{\g} : \begin{array}{ll}
&|i\oo\cdot\ell\!+\! \mu_{h}^{\nu-1}(\oo)\!-\! \mu_{h'}^{\nu-1}(\oo)|\geq
\frac{\g|\s j^{2}-\s'j'^{2}|}{\langle\ell\rangle^{\tau}},\\
&\forall |\ell|\leq N_{\nu-1}, \! h,h'\in\CC\times\ZZZ
\end{array}
\right\},\\
\calO_{\nu}^{\g}(u)&:=\left\{\oo\in\Lambda_{\nu-1}^{\g} : \begin{array}{ll}
&|i\oo\cdot\ell+\mu_{\s,j}^{\nu-1}-\mu_{\s,k}^{\nu-1}|\geq\frac{\g}{\langle\ell\rangle^{\tau}\langle j\rangle}, \\
&\; \ell\in\ZZZ^{d}\backslash\{0\}, j\in\ZZZ
, k=\pm j, \s\in\CC \end{array}
\right\},
\end{aligned}
\end{equation}
where 
  \begin{equation}\label{automondoschifo2}
  \begin{aligned}
  \mu^{\nu}_{\s,j}&:=i\s\left(-{m}_{2}j^{2}+{m_{0}}+{r}_{j}^{\nu,j}+{r}_{-j}^{\nu,-j}+\frac{1}{2}a_{j}{b}_{j}\right),\;  \\
  {b}^{\nu}_{j}&:=\sqrt{\left(-2|{m}_{1}|+\frac{{r}_{j}^{\nu,j}-{r}_{-j}^{\nu,-j}}{a_j}\right)^{2}+4\frac{|{r}_{j}^{\nu,-j}|^{2}}{(a_j)^{2}}}, \\
 &a_{j}=j, \; {\rm if} \; j\neq0, \;\; a_{j}=1, \; {\rm if} \; j=0,
  \end{aligned}
  \end{equation}
are the eigenvalues of the matrix $\Omega_{\s,\und{j}}^{\nu}$.
For $\nu\geq0$ one has $\ol{r_{j}^{\nu,k}}=r_{k}^{\nu,j}$, for $k=\pm j$ and
\begin{equation}\label{eq:4.20ham}
|r_{j}^{\nu,k}|_{\g}:=|r_{j}^{\nu,k}|_{\Lambda_{\nu}^{\g},\g}\leq\frac{\e C}{\langle j\rangle},\qquad 
|r_{j}^{\nu,-j}|_{\g}\leq \frac{\e C}{\langle j\rangle}, \quad
 |{b}^{\nu}_{j}|_{\g}\leq {\e C}. 
\end{equation}
The remainder $\RR_{\nu}$ satisfies $\forall\;s\in[\gots_{0},q-\h_{1}-\be]$ ($\al$ is defined in \eqref{eq:4.14ham})
\begin{equation}\label{eq:4.21ham}
\begin{aligned}
|\RR_{\nu}|_{s}&\leq |\RR_{0}|_{s+\be}N_{\nu-1}^{-\al},\\
|\RR_{\nu}|_{s+\be}&\leq |\RR_{0}|_{s+\be}N_{\nu-1},
\end{aligned}
\end{equation}
%
\begin{equation}\label{eq:4.21bisham}
|(\RR_{\nu})_{\s}^{-\s}|_{s}+|D(\RR_{\nu})_{\s}^{\s}|_{s}\lessdot|\RR_{\nu}|_{s},\; \s\in\CC, \quad {\rm where} \quad
D:={\rm diag}_{j\in\ZZZ}\{j\}.
\end{equation}

%
Moreover
there exists a map $\Phi_{\n-1}$ of the form $\Phi_{\nu-1}:=\exp{(\Psi_{\nu-1})}: {\bf H}^{s}\to{\bf H}^{s}$, where $\Psi_{\nu-1}$ is T\"oplitz in time, $\Psi_{\nu-1}:=\Psi_{\nu-1}(\f)$ (see (\ref{eq:2.16})), such that
\begin{equation}\label{eq:4.16hamham}
\calL_{\nu}:=\Phi_{\nu-1}^{-1}\calL_{\nu-1}\Phi_{\nu-1}
\end{equation}
%
and for $\nu\geq1$ one has:
\begin{equation}\label{eq:4.22ham}
|\Psi_{\nu-1}|_{s,\g}\leq 
|\RR_{0}|_{s+\be}^{0} N_{\nu-1}^{2\tau+1}N_{\nu-2}^{-\al}.
\end{equation}

One has that the operators $\Phi_{\nu-1}^{\pm1}$ are symplectic and
the operator $\RR_{\nu}$ is hamiltonian.
Finally  the eigenvalues $\mu_{\s,j}^{\nu}$ are purely imaginary.

\noindent
$({\bf S2})_{\nu}$ For all $j\in\ZZZ$ 
there exists  Lipschitz extensions $\tilde{\Omega}_{\s,j}^{\nu,k}: \Lambda \to i \RRR$ of 
$\Omega_{\s,j}^{\nu,k}: \Lambda_{\nu}^{\g} \to i \RRR$, for $k=\pm j$, and 
$\tilde{\mu}_{h}^{\nu}(\cdot) : \Lambda \to i\RRR$ of $\mu_{h}^{\nu}(\cdot):\Lambda_{\nu}^{\g}\to i\RRR$,
such that for $\nu\geq1$,
\begin{equation}\label{eq:4.23ham}
\begin{aligned}
&|\tilde{\Omega}_{\s,j}^{\nu,k}-\tilde{\Omega}_{\s,j}^{\nu-1,k}|_{\g}\leq |(\RR_{\nu-1})_{\s}^{\s}|_{\gots_{0}}, \quad  \s\in\CC, j\in\ZZZ,  k= \pm j,\\
&|\tilde{\mu}_{\s,j}^{\nu}-\tilde{\mu}_{\s,j}^{\nu-1}|^{{\rm sup }}\leq|(\RR_{\nu-1})_{\s}^{\s}|_{\gots_{0}}, \quad \s\in\CC, j\in\ZZZ .
\end{aligned}
\end{equation}

\noindent
$({\bf S3})_{\nu}$ Let ${ u}_{1}(\la)$, ${ u}_{2}(\la)$ be Lipschitz families of Sobolev functions, defined for $\la\in\Lambda_{o}$ such that (\ref{eq:4.2ham}), (\ref{eq:4.15ham}) hold with
$\RR_{0}=\RR_{0}({u}_{i})$ with $i=1,2$.
Then for $\nu\geq0$, for any $\la\in\Lambda_{\nu}^{\g_{1}}\cap\Lambda_{\nu}^{\g_{2}}$,
with $\g_{1},\g_{2}\in[\g/2,2\g]$, one has
\begin{subequations}\label{eq:4.24ham}
\begin{align}
|\RR_{\nu}({ u}_{1})-\RR_{\nu}({ u}_{2})|_{\gots_{0}}&\leq
\e N_{\nu-1}^{-\al}||{ u}_{1}-{ u}_{2}||_{\gots_{0}+\h_{3}},
\label{eq:4.24aham}\\
|\RR_{\nu}({ u}_{1})-\RR_{\nu}({ u}_{2})|_{\gots_{0}+\be}&\leq
\e N_{\nu-1}||{ u}_{1}-{ u}_{2}||_{\gots_{0}+\h_{3}},\label{eq:4.24bham}
\end{align}
\end{subequations}
and moreover, for $\nu\geq1$, for any $s\in[\gots_{0},\gots_{0}+\be]$, for any $(\s,j)\in\CC\times\ZZZ$ and $k=\pm j$,
\begin{eqnarray}\label{eq:4.24bisham}
|(r_{\s,j}^{\nu,k}({ u}_{2})-r_{\s,j}^{\nu,k}({ u}_{1}))-(r_{s,j}^{\nu-1,k}({ u}_{2})-r_{\s,j}^{\nu-1,k}({ u}_{1}))|
&\!\!\!\leq\!\!\!&|\RR_{\nu-1}({ u}_{1})-\RR_{\nu-1}({ u}_{2})|_{\gots_{0}},\nonumber\\
|(r_{\s,j}^{\nu,k}({ u}_{2})-r_{\s,j}^{\nu,k}({ u}_{1}))|& \!\!\!\leq\!\!\!& \e C||{ u}_{1}-{ u}_{2}||_{\gots_{0}+\h_{3}}, \label{aaaham}\\
|b^{\nu}_{j}(u_{1})-b^{\nu}_{j}(u_{2})|&\!\!\!\leq\!\!\!&\e C||{ u}_{1}-{ u}_{2}||_{\gots_{0}+\h_{3}}.\label{aaa2ham}
\end{eqnarray}

\noindent
$({\bf S4})_{\nu}$ Let $u_{1},u_{2}$ be as in $({\bf S3})_{\nu}$ and $0<\rho<\g/2$.
For any $\nu\geq0$ one has that, if
\begin{equation}\label{eq:4.25ham}
\begin{aligned}
 CN_{\nu-1}^{\tau}&||{ u}_{1}-{ u}_{2}||_{\gots_{0}+\h_{3}}^{{\rm sup}}\leq \rho \e \quad \Rightarrow \\
&P_{\nu}^{\g}({ u}_{1})\subseteq P_{\nu}^{\g-\rho}({ u}_{2}), \quad \calO_{\nu}^{\g}({ u}_{1})\subseteq \calO_{\nu}^{\g-\rho}({ u}_{2}).
\end{aligned}
\end{equation}

\end{lemma}

\proof
We start by proving that ${\bf (Si)_{0}}$ hold for $i=0,\ldots,4$.

${\bf (S1)_{0}}$. Clearly the properties (\ref{eq:4.20ham})-(\ref{eq:4.21ham}) hold by (\ref{eq:4.11ham}), (\ref{eq:4.12ham})
and the form of $\mu_{k}^{0}$ in (\ref{automondoschifo2}), recall that $r_{k}^{0}=0$ .
Moreover, $m_{2}, |m_{1}|$ and $m_{0}$ real  imply that $\mu_{k}^{0}$ are imaginary. In addition to this,  our hypotheses guarantee that $\RR_{0}$ and
$\calL_{0}$ are hamiltonian operators.

${\bf (S2)_{0}}$. We have to extend the eigenvalues $\mu_{k}^{0}$ from
the set $\Lambda_{0}^{\g}$ to the entire $\Lambda $. Namely we extend
the functions $m_{2}(\la), m_{1}(\la)$ and $m_{0}(\la)$ to a $\tilde{m}_{i}(\la)$ for $i=0,1,2$ which are Lipschitz in $\Lambda$, with the same sup norm
and Lipschitz semi-norm, 
by Kirszbraum theorem.

${\bf (S3)_{0}}$. It holds by (\ref{eq:B15aham}) and \eqref{eq:B15bham} for $\gots_{0}$, $\gots_{0}+\be$ using 
(\ref{eq:4.2ham}) and (\ref{eq:4.14ham}).

${\bf (S4)_{0}}$. By definition one has 
$\Lambda_{0}^{\g}({ u}_{1})=\Lambda_{o}=\Lambda_{0}^{\g-\rho}({ u}_{2})$, then the  
(\ref{eq:4.25ham}) follows trivially.

\subsubsection{Kam step}

In this Section we  show in detail one step of the KAM iteration. In other words we will
show how to define the transformation $\Phi_{\nu}$ and $\Psi_{\n}$ that trasform the operator
$\calL_{\n}$ in the operator $\calL_{\nu+1}$. For simplicity we shall avoid to write the index,
but we will only write $+$ instead of $\n+1$. 

 We consider a transformation of the form $\Phi=\exp{(\Psi)}$, with $\Psi:=(\Psi_{\s}^{\s'})_{\s,\s'=\pm1}$, acting on the operator 
$$\calL= \oo\cdot\del_{\f} \uno+\DD+\RR
$$
with $\DD$ and $\RR$  as in \eqref{eq:4.16ham}, 
We define the operator
$$
e^{ad(\Psi)} L:=\sum_{m=0}^{\infty}\frac{1}{m!}[\Psi, L]^{m}, \quad {\rm with} \quad [\Psi, L]^{m}=[\Psi, [\Psi,L]^{m-1}], \quad [\Psi,L]=\Psi L-L\Psi
$$
acting on the matrices $L$. One has that
\begin{equation}\label{pippoham}
e^{ad(\Psi)}L= e^{-\Psi}Le^{\Psi}.
\end{equation}
Clearly the \eqref{pippoham} hold since $\Psi$ is a linear operator.
Then, $\forall\; { h}\in {\bf H}^{s}$, by conjugation one has
\begin{equation}\label{eq:4.1.22ham}
\begin{aligned}
\Phi^{-1}\calL\Phi&=e^{ad(\Psi)}(\oo\cdot\del_{\f}\uno+\DD)+e^{ad(\Psi)}\RR\\
&=\oo\cdot\del_{\f}+\DD+[\Psi,\oo\cdot\del_{\f}\uno+\DD]+\Pi_{N}\RR\\
&+\sum_{m\geq2}\frac{1}{m!}[\Psi,\oo\cdot\del_{\f}\uno+\DD]^{m}+\Pi_{N}^{\perp}\RR+\sum_{m\geq1}\frac{1}{m!}[\Psi,\RR]^{m}
\end{aligned}
\end{equation}
where  $\Pi_{N}$ is defined in (\ref{smoothop}).
The smoothing operator $\Pi_{N}$ is necessary for technical reasons: it will be used in order to obtain
suitable estimates on the high norms of the transformation $\Phi$.

In the following Lemma we will show how to solve the \emph{homological} equation
\begin{equation}\label{homeqham}
\begin{aligned}
&[\Psi, \oo\cdot\del_{\f}\uno+\DD]+\Pi_{N}\RR=[\RR], \quad {\rm where}\\
&[\RR]_{\s,j}^{\s',j'}(\ell):=
\left\{
\begin{aligned}
& (\RR)_{\s,j}^{\s',k}(0), \quad \s=\s',\; k=j,-j,\; \ell=0\\
& 0 \quad \quad {\rm otherwise},
\end{aligned}\right.
\end{aligned}
\end{equation}
for $k,k'\in \CC\times\NNN\times\ZZZ^{d}$.

\begin{lemma}[{\bf Homological equation}]\label{lemhomham}
For any $\la\in\Lambda^{\g}_{\nu+1}$ there exists a unique solution $\Psi=\Psi(\f)$ of the homological equation
(\ref{homeqham}), such that
\begin{equation}\label{eq:4.1.33ham}
|\Psi|_{s,\g}\leq C N^{2\tau+1}\g^{-1}
|\RR|_{s,\g}
\end{equation}
Moreover, for $\g/2\leq \g_{1},\g_{2}\leq 2\g$, and if $u_{1}(\la), u_{2}(\la)$ are Lipschitz functions,
then $\forall \; s\in[\gots_{0},\gots_{0}+\be]$, 
$\la\in \Lambda^{\g_{1}}_{+}(u_{1})\cap\Lambda^{\g_{2}}_{+}(u_{2})$, one has
\begin{equation}\label{eq:4.1.44ham}
|\Delta_{12}\Psi|_{s}\leq C N^{2\tau+1}\g^{-1}\left(|\RR({u}_{2})|_{s}
||{ u}_{1}-{u}_{2}||_{\gots_{0}+\h_{2}}+
|\Delta_{12}\RR|_{s}\right),
\end{equation}
where we define $\Delta_{12}\Psi=\Psi({u}_{1})-\Psi({u}_{2})$.\\
Finally, one has $\Phi : {\bf  H}^{s}\to {\bf H}^{s}$ is symplectic.
\end{lemma}

\begin{proof}
We rewrite the equation \eqref{homeqham} on each component $k=(\s,j,p),k'=(\s',j',p')$ and we get the following
matricial equation
\begin{equation}\label{homeq2ham}
i\oo\cdot(p-p')\Psi_{\s,\und{j},p}^{\s',\und{j'},p'}+\Omega_{\s,\und{j}}\Psi_{\s,\und{j},p}^{\s',\und{j'},p'}-
\Psi_{\s,\und{j},p}^{\s',\und{j'},p'}\Omega_{\s',\und{j'}}=-\RR_{\s,\und{j}}^{\s',\und{j'}}(p-p')
\end{equation}
where $\Omega_{\s,\und{j}}$ is defined in \eqref{eq:4.16ham}
and where we have set
\begin{equation}\label{homeqhomeq}
\Psi_{\s,\und{j},p}^{\s',\und{j'},p'}:=\left(\begin{matrix} \Psi_{\s,j,p}^{\s',j',p'} & \Psi_{\s,j,p}^{\s',-j',p'}\\
\Psi_{\s,-j,p}^{\s',j',p'}  & \Psi_{\s,-j,p}^{\s',-j',p'}\end{matrix}
\right)
\end{equation}
the matrix block indexed by $(j,j')$. To solve equation \eqref{homeq2ham} we can use Lemma \ref{finitema3}
with $A:=i\oo\cdot p\uno+\Omega_{\s,\und{j}}$ and $B=i\oo\cdot p'\uno +\Omega_{\s',\und{j'}}$.
%
 Hence if we write $\mu_{\s,h}$ and $\mu_{\s',h'}$ with $h=j,-j$ and $h'=j',-j'$ the eigenvalues respectively of 
 $\Omega_{\s,\und{j}}$ and $\Omega_{\s',\und{j'}}$,
 %
 \begin{equation}\label{homeq3ham}
 \begin{aligned}
 \|\Psi_{\s,\und{j},p}^{\s',\und{j'},p'}\|_{\infty}&\stackrel{(\ref{eq:419ham})}{\leq}
 C \frac{\langle\ell\rangle^{\tau}\g^{-1}}{|\s j^{2}-\s' j'^{2}|} \max_{\substack{h=j,-j, h'=j',-j'}}|\RR_{\s,h}^{\s',h'}(\ell)|, \quad\\ 
 &\qquad \s=\s', j\neq j',\quad{\rm or} \;\s\neq\s', \; \forall j,j'\\
 \|\Psi_{\s,\und{j},p}^{\s',\und{j},p'}\|_{\infty}&\stackrel{(\ref{eq:419ham})}{\leq} C\langle\ell\rangle^{\tau}|j|\g^{-1} \max_{\substack{h=j,-j}}|\RR_{\s,h}^{\s',h}(\ell)|, \quad \s=\s',\; j=j',
 \end{aligned}
 \end{equation}
 where  we fixed $p-p'=\ell$. Clearly the solution $\Psi$ is T\"opliz in time.
 Unfortunately bounds \eqref{homeq3ham}
 are not sufficient in order to estimate the decay norm of the matrix $\Psi_{\s}^{\s'}$.  
 Roughly speaking one needs to prove, for any $\ell$, that
 $\Psi_{\s,j}^{\s',j'}(\ell)\approx o(1/\langle j-j'\rangle^{s})$, and $\Psi_{\s,j}^{\s',-j'}\approx o(1/\langle j+j'\rangle^{s})$.
 Actually we are able to prove the following.
 
 Assume that either $|j|\leq \frac{C}{\gote}$ or $|j'|\leq \frac{C}{\gote}$ for some large $C>0$ and $\gote$ defined in \eqref{nondeg}.
Assume also that
 \begin{equation}\label{quiquo3}
\max_{\substack{h=j,-j, h'=j',-j'}}|\RR_{\s,h}^{\s',h'}(\ell)|=|\RR_{\s,j}^{\s',j'}(\ell)|.
 \end{equation}
 
 By \eqref{homeq3ham} we have that
 \begin{equation}\label{quiquo}
 \begin{aligned}
( |\Psi_{\s,j}^{\s',j'}|^{2}+ |\Psi_{\s,-j}^{\s',-j'}|^{2})\langle j-j'\rangle^{2s}&+
 ( |\Psi_{\s,j}^{\s',-j'}|^{2}+ |\Psi_{\s,-j}^{\s',j'}|^{2})\langle j+j'\rangle^{2s}\\
& \leq C \frac{\langle\ell\rangle^{2\tau}\g^{-2}}{|\s j^{2}-\s' j'^{2}|^{2}} |\RR_{\s,j}^{\s',j'}(\ell)|^{2}\left(
\langle j-j'\rangle^{2s}+\langle j+j'\rangle^{2s}
\right)\\
&\leq  \tilde{C} \frac{\langle\ell\rangle^{2\tau}\g^{-2}}{|\s j^{2}-\s' j'^{2}|^{2}} |\RR_{\s,j}^{\s',j'}(\ell)|^{2}
\langle j-j'\rangle^{2s}
 \end{aligned}
\end{equation}
where we used the fact that, for a finite number of $j$ (or finite $j'$), one has
$$
\langle j+j'\rangle\leq K \langle j-j'\rangle,
$$
for some large $K=K(\gote)>0$. 
Note also that the smaller is $\gote$ the larger is the constant $K$.
If the \eqref{quiquo3} does not hold one can treat the other cases by reasoning
as done in \eqref{quiquo}.
Assume now that
\begin{equation}\label{quiquo5}
|j|,|j'|\geq \frac{C}{\gote}
\end{equation}
holds. Here the situation is more delicate. 
Consider the matrices $\Omega_{\s,\und{j}}, \Omega_{\s',\und{j'}}$ in equation \eqref{homeq2ham}
which have, by \eqref{automondoschifo2},  eigenvalues $\mu_{\s,j}$, $\mu_{\s,-j}$ and $\mu_{\s',j'}$, $\mu_{\s',-j'}$ respectively.
First of all one can note that by \eqref{quiquo5}
\begin{equation}\label{quiquo4}
|\mu_{\s,j}-\mu_{\s,-j}|,\;\geq |m_{1}|\langle j\rangle\geq c\e \gote\langle j\rangle, \quad  |\mu_{\s',j'}-\mu_{\s',-j'}|\geq |m_{1}|\langle j' \rangle
\end{equation}
 by the \eqref{nondeg}. Hence we can define the invertible matrices
 \begin{equation}\label{quiquo10}
 U_{\s,\und{j}}:=\left(
 \begin{matrix}
 \frac{\Omega_{\s,-j}^{-j}-\mu_{\s,j}}{\mu_{\s,j}-\mu_{\s,-j}} & \frac{-\Omega_{\s,j}^{-j}}{\mu_{\s,j}-\mu_{\s,-j}} \\
 \frac{-\Omega_{\s,-j}^{j}}{\mu_{\s,j}-\mu_{\s,-j}}  &
  \frac{\Omega_{\s,j}^{j}-\mu_{\s,-j}}{\mu_{\s,j}-\mu_{\s,-j}} 
 \end{matrix}
 \right), 
 \end{equation}
 and moreover one can check that
 \begin{equation}\label{quiquo7}
 U_{\s,\und{j}}^{-1}\Omega_{\s,\und{j}}U_{\s,\und{j}}=D_{\s,\und{j}}=\left(\begin{matrix} \mu_{\s,j} & 0 \\
 0 &\mu_{\s,-j}\end{matrix}
 \right),
\end{equation}
 In order to simplify the notation we set
 \begin{equation}
 f_{\s,j}^{(1)}:=\frac{\Omega_{\s,-j}^{-j}-\mu_{\s,j}}{\mu_{\s,j}-\mu_{\s,-j}}, \qquad f^{(2)}_{\s,j}:=  \frac{\Omega_{\s,j}^{j}-\mu_{\s,-j}}{\mu_{\s,j}-\mu_{\s,-j}},
 \qquad c_{\s,j}:=\frac{-\Omega_{\s,j}^{-j}}{\mu_{\s,j}-\mu_{\s,-j}}.
 \end{equation}
 First of all, by using \eqref{quiquo10}, \eqref{quiquo4} and \eqref{eq:4.20ham} one has
 \begin{equation}\label{quiquo11}
|f_{\s,j}^{(1)}|+|f_{\s,j}^{(2)}|
\leq 4\frac{C}{c\gote}, \qquad
|c_{\s,j}|\leq  \frac{1}{c\e\gote}|r_{j}^{-j}|.
 \end{equation}

Hence one has
\begin{equation}\label{quiquo23}
U_{\s}:={\rm diag}_{|j|\geq C/\gote, j\in \NNN}U_{\s,\und{j}}, \qquad |U_{\s}|_{s,\g}\leq \frac{C}{|m_{1}|}|\RR_{\s}^{\s'}|_{s,\g},
\end{equation}
and moreover $U_{\s}$ diagonalizes the matrix $\Omega_{\s}={\rm diag}_{|j|\geq C/\gote}\Omega_{\s,\und{j}}$. 
Setting $U_{\s}^{-1}\Psi_{\s}^{\s'}U_{\s}=Y_{\s}^{\s'}$, equation \eqref{homeq2ham}, for $\s,\s'=\pm1$, reads
\begin{equation}\label{quiquo20}
\ii \oo\cdot\del_{\f}Y_{\s}^{\s'}+D_{\s}Y_{\s}^{\s'}-\s' Y_{\s}^{\s'}D_{\s}=
U_{\s}^{-1}R_{\s}^{\s'}U_{\s}.
\end{equation}
For $|\ell|\leq N $ we set
\begin{equation}\label{quiquo21}
Y_{\s,j}^{\s',j'}(\ell)=\frac{(U_{\s}^{-1}\RR_{\s}^{\s'}U_{\s})_{\s,j}^{\s',j'}(\ell)}{\ii\oo\cdot\ell+\mu_{\s,j}-\mu_{\s',j'}}
\end{equation}
and hence 
we get the bound
\begin{equation}\label{quiquo22}
|Y_{\s}^{\s'}|_{s}\leq \g^{-1}N^{\tau}|U_{\s}^{-1}\RR_{\s}^{\s'}U_{\s}|_{s},
\end{equation}
where we used the estimates   \eqref{eq:419ham} on the small divisors.

 By the definition, the estimate \eqref{quiquo23} and the interpolation properties in Lemma \ref{bubbole}
  we can bound the decay norm of $\Psi$ as
 \begin{equation}\label{homeq4ham}
 |\Psi|_{s}\leq C(s)\g^{-1} N^{\tau}|\RR|_{s},
 \end{equation}
 using that $|\RR|_{s}/|m_{1}|\leq C$ for some constant $C>0$. 
 Moreover the following hold:
 \begin{lemma}\label{regulham}
 Define the operator $A$ as
\begin{equation}\label{eqrmk1ham}
A_{k}^{k'}=A_{\s,j}^{\s',j'}(\ell):=\left\{
\begin{aligned}
&\Psi_{\s,j}^{\s,j'}(\ell), \quad \s=\s'\in\CC,\;\; j=\pm j'\in \ZZZ
\;\; \ell\in\ZZZ^{d},\\
&0, \quad {\rm otherwise},
\end{aligned}\right.
\end{equation}
then the operator $\Psi-A$ is regularizing and hold
\begin{equation}\label{regul1ham}
|D(\Psi-A)|_{s}\leq \g^{-1}N^{\tau}|\RR|_{s},
\end{equation}
where $D:=diag\{j\}_{j\in\ZZZ}$.
 \end{lemma}
 This Lemma will be used in the study of the remainder of the conjugate operator. 
In particular we will use it to prove that the reminder is still in the class of operators described in $(\ref{eq:4.12ham})$.

Now we need a bound on the Lipschitz semi-norm of the transformation. Then, given $\oo_{1},\oo_{2}\in \Lambda_{\nu+1}^{\g}$, one has, for $k=(\s,j,p),k'=(\s',j',p')\in\CC\times\ZZZ\times\ZZZ^{d}$, and $\ell:=p-p'$,
\begin{equation}\label{homeq5ham}
\begin{aligned}
{\oo_1}\cdot\ell \Big[\Psi_{\s,\und{j}}^{\s',\und{j'}}(\ell,\oo_{1})&-\Psi_{\s,\und{j}}^{\s',\und{j'}}(\ell,\oo_{2})\Big]+
\Omega_{\s,\und{j}}(\oo_{1})\left[\Psi_{\s,\und{j}}^{\s',\und{j'}}(\ell,\oo_{1})-
\Psi_{\s,\und{j}}^{\s',\und{j'}}(\ell,\oo_{2})\right]+\\
&-\left[\Psi_{\s,\und{j}}^{\s',\und{j'}}(\ell,\oo_{1})-\Psi_{\s,\und{j}}^{\s',\und{j'}}(\ell,\oo_{2})\right]\Omega_{\s',\und{j'}}(\oo_{1})\\
&+(\oo_{1}-\oo_{2})\cdot\ell \Psi_{\s,\und{j}}^{\s',\und{j'}}(\ell,\oo_{2})+\\
&+\left[\Omega_{\s,\und{j}}(\oo_{1})-\Omega_{\s,\und{j}}(\oo_{2})\right]\Psi_{\s,\und{j}}^{\s',\und{j'}}(\ell,\oo_{2})\\
&+
\Psi_{\s,\und{j}}^{\s',\und{j'}}(\ell,\oo_{2})\left[\Omega_{\s',\und{j'}}(\oo_{1})-\Omega_{\s',\und{j'}}(\oo_{2})\right]=\\
&=\RR_{\s,\und{j}}^{\s',\und{j'}}(\ell,\oo_{1})-\RR_{\s,\und{j}}^{\s',\und{j'}}(\ell,\oo_{1}).
\end{aligned}
\end{equation}
First we can note that
\begin{equation}\label{homeq6ham}
\begin{aligned}
|\Omega_{\s,j}^{\phantom{g}j}(\oo_1)-\Omega_{\s',j'}^{\phantom{g}j'}(\oo_2)|&\leq|m_{2}(\oo_{1})-m_{2}(\oo_{2})||\s j^{2}-\s'j'^{2}|+\e\g^{-1}\\
&+
|m_{1}(\oo_{1})-m_{1}(\oo_{2})||\s j-\s'j' |+|m_{0}(\oo_{1})-m_{0}(\oo_{2})|\\
&\leq C |\oo_{1}-\oo_{2}|(\e\g^{-1}|\s j^{2}-\s' j'^{2}|+\e\g^{-1}+\e\g^{-1})
\end{aligned}
\end{equation}
where we used the (\ref{eq:4.16ham}), (\ref{eq:4.20ham}) and (\ref{eq:3.2.44}) to estimate the Lipschitz semi-norm 
of the constants $m_{i}$. Following the same reasoning, one can estimate the sup norm of the matrix
$\Omega_{\s,\und{j}}(\oo_{1})-\Omega_{\s,\und{j}}(\oo_{2})$.
Therefore by triangular inequality one has
\begin{equation}\label{homeq7ham}
\begin{aligned}
&\|\Psi_{\s,\und{j}}^{\s',\und{j'}}(\ell,\oo_{1})-\Psi_{\s,\und{j}}^{\s',\und{j'}}(\ell,\oo_{2})\|_{\infty}\lessdot |\RR_{\s,\und{j}}^{\s',\und{j'}}(\ell,\oo_{1})-\RR_{\s,\und{j}}^{\s',\und{j'}}(\ell,\oo_{1})|_{max} N^{\tau}\g^{-1}
+\\
&+|\oo_{1}-\oo_{2}|\left(|\ell|+\e \g^{-1}|\s j^{2}-\s' j^{j'}|\right)\\
&+|\oo_{1}-\oo_{2}|\left(\e\g^{-1}|\s j-\s' j'|\e\g^{-1}\right)\|\RR_{\s,h}^{\s',h'}(\ell,\oo_{2})\|_{\infty}\frac{N^{2\tau+1}\g^{-2}}{|\s j^{2}-\s' j'^{2}|},
\end{aligned}
\end{equation}
for $|\ell|\leq N$, $j\neq j' $ and $\e\g^{-1}\leq1$. 
As done for the estimate \eqref{homeq4ham}  for a finite number of $j$ of a finite number of $j'$ the bound \eqref{homeq7ham} 
is sufficient to get,
 for $\oo\in\Lambda_{\nu+1}^{\g}$ and using also the bound \eqref{eq:419ham} with $j=j'$, the estimate
\begin{equation}\label{homeq8ham}
|\Psi|_{s,\g}:=|\Psi|_{s}^{\rm sup}+\g\sup_{\oo_{1}\neq\oo_{2}}\frac{|\Psi(\oo_{1})-\Psi(\oo_{2})|}{|\oo_{1}-\oo_{2}|}
\leq C\g^{-1}N^{2\tau+1}|\RR|_{s,\g},
\end{equation}
that is the \eqref{eq:4.1.33ham}.

On the other hand, in the case of \eqref{quiquo5}, we can reason as follows. Consider the diagonalizing matrix $U_{\s,\und{j}}$ defined in \eqref{quiquo7}
and recall  that by \eqref{quiquo23} also the lipschitz semi-norm of $U_{\s}$ is bounded by the lipschitz semi-norm of $\RR_{\s}^{\s'}$. 
Hence by \eqref{quiquo20}, \eqref{quiquo21}, using again the interpolation properties of the decay norm in Lemma \eqref{bubbole}
one get the Lipschitz bound in \eqref{homeq8ham}.
 Note also that the Lemma \ref{regulham} holds  with $|\cdot|_{s,\g}$ and $N^{2\tau+1}$ instead 
of $|\cdot|_{s}$ and $N^{\tau}$.

The proof of the bound \eqref{eq:4.1.44ham} is based on the same strategy used to proof \eqref{homeq8ham}. We refer to 
the proof of Lemma $4.39$ of \cite{FP}.


%

Finally we show that $\Psi$ is an hamiltonian vector field, and hence the transformation $\Phi$ is symplectic.
By hypothesis $\RR$ is hamiltonian, hence by Lemma \ref{lemFouham} we have
\begin{equation}\label{homeq11ham}
\left(\ol{\RR_{\s}^{\s}}\right)^{T}=-\RR_{\s}^{\s}, \quad \ol{\RR_{\s}^{-\s}}=\RR_{-\s}^{\s}, \quad
\ol{\RR_{\s}^{\s'}}=\RR_{-\s}^{-\s'},
\quad \forall \; \s,\s'\in\CC.
\end{equation}
Moreover, by inductive hypothesis $({\bf S1})_{\nu}$ one can note that
\begin{equation}\label{homeq12ham}
\left(\ol{\Omega_{\s}}\right)^{T}=-\Omega_{\s}=\Omega_{-\s}.
\end{equation}
By \eqref{homeq11ham}, \eqref{homeq12ham} one can easily note that the solution of the equation
$$
\oo\cdot\del_{\f}\Psi_{\s}^{\s'}+\Omega_{\s}\Psi_{\s}^{\s'}-\Psi_{\s}^{\s'}\Omega_{\s'}=\RR_{\s}^{\s'},
$$
satisfies conditions in \eqref{homeq11ham}, hence, again by Lemma \ref{lemFouham}, $\Psi$ is hamiltonian. This concludes the proof of Lemma \ref{lemhomham}.
\end{proof}

Next Lemma concludes one step of our KAM iteration.
\begin{lemma}[{\bf The new operator $\calL_{+}$}]\label{newopham}
Consider the operator $\Phi=\exp(\Psi)$ defined in Lemma \ref{lemhomham}. Then the operator
$\calL_{+}:=\Phi^{-1}\calL\Phi$ has the form
\begin{equation}\label{newop1ham}
\calL_{+}:=\oo\cdot\del_{\f}\uno+\DD_{+}+\RR_{+},
\end{equation}
where the diagonal part is
\begin{equation}\label{newop2ham}
\begin{aligned}
\DD_{+}&={\rm diag}_{(\s,j)\in\CC\times\ZZZ}\{\Omega_{\s,\underline{j}}^{+}\},
 \quad
\Omega_{\s,\underline{j}}^{+}(\la)=
\left(\begin{matrix}\Omega_{\s,j}^{+,j} & \Omega_{\s,j}^{+,-j} \\
\Omega_{\s,-j}^{+,j} & \Omega_{\s,-j}^{+,-j} \end{matrix}\right), \\
&
\begin{aligned}
\Omega_{\s,j}^{+, j}&:=-i \s m_{2}j^{2}-i\s|m_{1}|j+i\s m_{0}+i\s r_{j}^{+,j},\\
\Omega_{\s,j}^{+,-j}&:=i\s r_{j}^{+,-j},\\
r_{j}^{+,h}&:=r_{j}^{h}+\RR_{\s,j}^{\s,h}(0), \quad h=\pm j.
\end{aligned}
\end{aligned}
\end{equation}
with $(\s,j)\in \CC\times\ZZZ,  \la\in\Lambda$. The eigenvalues $\mu_{\s,h}^{+}$, with $h=j,-j$, of $\Omega_{\s,\underline{j}}$ satisfy
\begin{equation}\label{newop3ham}
\begin{aligned}
&|r_{j}^{+,h}-r_{j}^{h}|^{lip}\leq |(\RR)_{\s}^{\s}|^{lip}_{\gots_{0}},\\
&|\mu_{\s,h}^{+}-\mu_{\s,h}|^{{sup}}\leq |(\RR)_{\s}^{\s}|_{\gots_{0},\g}, \quad h=j,-j.
\end{aligned}
\end{equation}
The remainder $\RR_{+}$ is such that
\begin{equation}\label{newop4ham}
\begin{aligned}
|\RR_{+}|_{s}&\leq_{s} N^{-\be}|\RR|_{s+\be,\g}+N^{2\tau+1}\g^{-1}|\RR|_{s,\g}|\RR|_{\gots_{0},\g},\\
|\RR_{+}|_{s+\be}&\leq_{s+\be}|\RR|_{s+\be,\g}+N^{2\tau+1}\g^{-1}|\RR|_{s+\be,\g}|\RR|_{\gots_{0},\g},
\end{aligned}
\end{equation}
and $(\RR_{+})_{\s}^{\s}=O(\e\del_{x}^{-1})$ while $(\RR_{+})_{\s}^{-\s}=O(\e)$ for $\s=\pm1$. More precisely,
\begin{equation}\label{eq:4.21trisham}
|(\RR_{+})_{\s}^{-\s}|_{s}+|D(\RR_{+})_{\s}^{\s}|_{s}\lessdot|\RR_{+}|_{s},\; \s\in\CC, \quad {\rm where} \quad
D:={\rm diag}_{j\in\ZZZ}\{j\}.
\end{equation}
Finally, for $\g/2\leq \g_{1},\g_{2}\leq 2\g$, and for $u_{1}(\la),u_{2}(\la)$ Lipschitz functions, then for any
$s\in[\gots_{0},\gots_{0}+\be]$ and $\la\in \Lambda_{+}^{\g_{1}}(u_{1})\cap\Lambda_{+}^{\g_{2}}(u_{2})$ 
one has
\begin{eqnarray}\label{newop5ham}
|\Delta_{12}\RR_{+}|_{s}&\!\!\!\!\!\!\leq\!\!\!\!\!\!& |\Pi_{N}^{\perp}\Delta_{12}\RR|_{s}+
+N^{2\tau+1}\g^{-1}\Big(|\RR(u_{1})|_{s}+|\RR(u_{2})|_{s} \Big)|\Delta_{12}\RR|_{\gots_{0}}\nonumber\\
&\!\!\!\!\!\!+\!\!\!\!\!\!&N^{2\tau+1}\g^{-1}\Big(\!|\RR(u_{1})|_{s}\!+\!|\RR(u_{2})|_{s}\! \Big)\Big(|\RR(u_{1})|_{\gots_{0}}\!+\!|\RR(u_{2})|_{\gots_{0}} \Big)||u_{1}\!-\!u_{2}||_{\gots_{0}+\h_{3}}\nonumber\\
&\!\!\!\!\!\!+\!\!\!\!\!\! &N^{2\tau+1}\g^{-1}\Big(|\RR(u_{1})|_{\gots_{0}}+|\RR(u_{2})|_{\gots_{0}} \Big)|\Delta_{12}\RR|_{s}
\end{eqnarray}

\end{lemma}

\begin{proof} The \eqref{newop2ham} follow by the \eqref{homeqham}. Note that the term $\RR_{\s,j}^{\s,k}(0)=\RR_{-\s,j}^{-\s,k}$
for $k=j,-j$ and hence the new correction $r_{j}^{+,h}$ does not depend on $\s$. Moreover, by \eqref{decay2} one has
\begin{equation}\label{newop6ham}
|\Omega_{\s,j}^{+,k}-\Omega_{\s,j}^{\phantom{g}k}|^{lip}\leq |(\RR)_{\s}^{\s}|_{\gots_{0}}^{lip}, \quad k=j,-j.
\end{equation}
Moreover, one has
\begin{equation}\label{newop6bisham}
\begin{aligned}
|\mu_{\s,j}^{+}-\mu_{\s,j}|&\leq 2\sup_{h=\pm j}|r_{h}^{+,h}-r_{h}^{h}|+|j||b_{j}^{+}-b_{j}|\\
&\leq 2\sup_{h=\pm j}|r_{h}^{+,h}-r_{h}^{h}|+\frac{|j|}{|j|}\sup_{h=\pm j}|r_{j}^{+,h}-r_{j}^{h}|
\stackrel{(\ref{newop6ham})}{\lessdot}|(\RR)_{\s}^{\s}|,
\end{aligned}
\end{equation}
then the \eqref{newop3ham} follows. Now, by \eqref{eq:4.1.22ham}
one has that
\begin{equation}\label{newop7ham}
\RR_{+}:=\Pi_{N}^{\perp}\RR+\sum_{n\geq2}\frac{1}{n!}[\Psi,\oo\cdot\del_{\f}\uno+\DD]^{n}+
\sum_{n\geq1}\frac{1}{n!}[\Psi,\RR]^{n}:=\Pi_{N}^{\perp}\RR+\BB.
\end{equation}
Here we used the simple fact that $[A,B]^{n}=[A,[A,B]]^{n-1}$ for any $n\geq1$.
Hence we can estimate
\begin{equation*}\label{newop8ham}
\begin{aligned}
|\RR_{+}|_{s,\g}&\leq_{s}|\Pi_{N}^{\perp}\RR|_{s,\g}+\sum_{k\geq2}\frac{1}{k!}|[\Psi,\Pi_{N}\RR]^{k-1}|_{s,\g}+
\sum_{n\geq1}\frac{1}{n!}|[\Psi,\RR]^{n}|_{s,\g}\\
&\leq_{s}
|\Pi_{N}^{\perp}\RR|_{s,\g}+\sum_{n\geq1}\frac{1}{n!}|[\Psi,\RR]^{n}|_{s,\g}
\leq_{s}\!|\Pi_{n}^{\perp}\RR|_{s}\!\\
&+\!\sum_{n\geq1}\frac{(nC(\gots_{0}))^{n-1}}{n!}|\Psi|_{\gots_{0},\g}^{n-1}|\RR|^{n-1}_{\gots_{0},\g}\!\left(|\Psi|_{s,\g}|\RR|_{\gots_{0},\g}\!+\!|\Psi|_{\gots_{0},\g}|\RR|_{s,\g}\right)\\
&\stackrel{(\ref{eq:2.22}),(\ref{eq:4.1.33ham})}{\leq}N^{-\be}|\RR|_{s+\be,\g}+N^{2\tau+1}\g^{-1}|\RR|_{s,\g}|\RR|_{\gots_{0},\g},
\end{aligned}
\end{equation*}
where we assumed that
\begin{equation}\label{newop9ham}
\sum_{n\geq1}\frac{n^{n-1}}{n!}C(\gots_{0})^{n-1}|\Psi|_{\gots_{0},\g}^{n-1}|\RR|^{n-1}_{\gots_{0},\g}<1.
\end{equation}

Now we have to estimate $\Delta_{12}\RR_{+}$ defined for $\la\in\La^{\g_{1}}(u_{1})\cup\Lambda^{\g_{2}}(u_2)$.
We write $\RR_{i}:=\RR(u_{i})$ for $i=1,2$. We first need a technical Lemma used to study the variation 
with respect to the function $u$, of the commutator between two operators.

\begin{lemma}\label{newop10ham}
Given operators $A(u), B(u)$ one has that the following identities hold for any $n\geq1$:
\begin{equation}\label{newop11ham}
[A_{1},B_{1}]^{n}=[A_{1},\Delta_{12}B]^{n}+[A_{1},B_{2}]^{n};
\end{equation} 
\begin{equation}\label{newop12ham}
[A_{1},B_{2}]^{n}=\Big[A_{1},[A_{2},B_{2}]\Big]^{n-1}+\Big[A_{1},[\Delta_{12}A,B_{2}]\Big]^{n-1};
\end{equation} 
\begin{equation}\label{newop13ham}
\begin{aligned}
&\Big[A_{1},[A_{2},B_{2}]\Big]^{n-1}-[A_{2},B_{2}]^{n}=(n-2)\Big[A_{1},\big[\Delta_{12}A,[A_{2},B_{2}]
\big]\Big]^{n-2}\\
&+\Big[\Delta_{12}A,[A_{2},B_{2}]^{n-1}
\Big].
\end{aligned}
\end{equation} 

\end{lemma}

\prova We prove the identities by induction. Let us start from the \eqref{newop11ham}. For $n=1$ it clearly holds.
We prove it for $n+1$ assuming that \eqref{newop11ham} holds for $n$. One has
\begin{equation}\label{newop11bisham}
\begin{aligned}
&\left[A_{1},\Delta_{12}B\right]^{n+1}+
[A_{1},B_{2}]^{n+1}=\Big[A_{1},[A_{1},\Delta_{12}B]^{n}\Big]+\Big[A_{1},[A_{1},B_{2}]^{n}\Big]\\
&\stackrel{(\ref{newop11ham})}{=}\Big[A_{1},[A_{1},B_{1}]^{n}\Big]=:[A_{1},B_{1}]^{n+1}.
\end{aligned}
\end{equation}
The remaining formul{\ae}   can be proved in the same way.
\EP

By using Lemma \ref{newop10ham}, one can rewrite the term $\BB$ in \eqref{newop7ham}. Then
setting $A_{s}:=|\RR_{1}|_{s}+|\RR_{2}|_{s}$ for any $s\geq0$, 
and using  \eqref{eq:2.12} and \eqref{newop9ham},  one obtains 
\begin{equation*}\label{newop15ham}
\begin{aligned}
|\Delta_{12}\BB|_{s}&
\stackrel{(\ref{eq:4.1.33ham}),(\ref{eq:4.1.44ham})}{\leq_{s}}
N^{2\tau+1}\g^{-1}A_{s}|\Delta_{12}\RR|_{\gots_{0}}+N^{2\tau+1}\g^{-1}A_{\gots_{0}}|\Delta_{12}\RR|_{s}
\\
&+2N^{4\tau+2}\g^{-1}A_{s}A_{\gots_{0}}^{2}||u_{1}-u_{2}||_{\gots_{0}+\h_{2}}\\
&+
2N^{4\tau+2}\g^{-2}A_{s}A_{\gots_{0}}|\Delta_{12}\RR|_{\gots_{0}}
+ N^{4\tau+2}\g^{-2}A_{s}A^{2}_{\gots_{0}}||u_{1}-u_{2}||_{\gots_{0}+\h_{2}}\\
&+
N^{4\tau+2}\g^{-2}A_{\gots_{0}}^{2}|\Delta_{12}\RR|_{s},
\end{aligned}
\end{equation*}
where we used the (\ref{eq:4.1.33ham}) and (\ref{eq:4.1.44ham}). If we assume that
\begin{equation}\label{newop16ham}
N^{2\tau+1}\g^{-1}A_{\gots_{0}}\leq1,
\end{equation}
then, using also \eqref{eq:2.22} we obtain the \eqref{newop5ham}.
Finally by using Lemma \ref{regulham} one can note that $[\Psi,\RR]_{\s}^{\s}=O(\e\del_{x}^{-1})$ while $[\Psi,\RR]_{\s}^{-\s}=O(\e)$ for $\s=\pm1$, this implies that the new remainder $\RR_{+}$ has the same properties.
\end{proof}

Clearly we proved Lemma \ref{newopham} by assuming the \eqref{newop9ham} and \eqref{newop16ham}. These hypotheses 
have to be verified inductively at each step. In the next Section we prove that the procedure described
above, can be iterated infinitely many times.

\subsection{Conclusions and Proof of Theorem \ref{KAMalgorithmham}}
To complete the proof of Lemma \ref{teo:KAMham} we proceed by induction. The proof of the iteration is essentially standard and based on the estimates of the previous Section.

We omit the proof of properties $({\bf S1})_{\nu+1},({\bf S2})_{\nu+1}$ and $({\bf S3})_{\nu+1}$ since one can repeat 
almost word by word the proof of Lemma $4.38$
in Section $4$
of  \cite{FP}.
The $(S4)_{\nu+1}$ is fundamentally  different. The difference depends on the multiplicity of the eigenvalues. Moreover the
result is weaker. This is why, in this case, the set of good parameters is smaller. We will see this fact in Section 6.

 \noindent
${\bf (S4)}_{\nu+1}$ Let $\oo\in\Lambda_{\nu+1}^{\g}$, then
by (\ref{eq:419ham}) and the inductive hypothesis ${\bf (S4)_{\nu}}$ one has
that $\Lambda_{\nu+1}^{\g}({ u}_{1})\subseteq\Lambda_{\nu}^{\g}({ u}_{1})\subseteq\Lambda_{\nu}^{\g-\rho}({u}_{2})\subseteq\Lambda_{\nu}^{\g/2}({ u}_{2})$.
Hence the eigenvalues $\mu_{h}^{\nu}(\oo,{u}_{2}(\oo))$ are well defined by the ${\bf (S1)_{\nu}}$.
Now, since $\la\in\Lambda_{\nu}^{\g}({u}_{1})\cap\Lambda_{\nu}^{\g/2}({u}_{2})$,
we have for $h=(\s,j) \in\CC\times\ZZZ$
and setting $h'=(\s',j')\in\CC\times\ZZZ$
\begin{equation}\label{eq:4.2.22ham}
\begin{aligned}
|(\mu_{h}^{\nu}&-\mu_{h'}^{\nu})(\oo,{ u}_{2}(\oo))-(\mu_{h}^{\nu}-\mu_{h'}^{\nu})(\oo,{ u}_{1}(\oo))|\leq
|\s j^{2}-\s'j'^{2}| |m_{2}(u_1)-m_{2}(u_2)|\\
&+|m_{0}(u_1)-m_{0}(u_2)||\s-\s'|
+\max_{\substack{j}}|r_{j}^{\nu,j}(\oo,{ u}_{2}(\oo))-r_{j}^{\nu,j}(\oo,{ u}_{1}(\oo))|\\
&+|j||b^{\nu}_{j}(u_{1})-b^{\nu}_{j}(u_{2})|+|j'||b^{\nu}_{j'}(u_{1})-b^{\nu}_{j'}(u_{2})|\\
&\stackrel{(\ref{eq:3.2.44}),(\ref{eq:4.24bisham}),(\ref{aaa2ham})}{\leq} 
\e C\left(|\s j^{2}-\s'j'^{2}|+||j|+|j'||\right) ||{u}_{2}-{ u}_{1}||_{\gots_{0}+\h_{2}},
\end{aligned}
\end{equation}
The (\ref{eq:4.2.22ham}) implies that for any $|\ell|\leq N_{\nu}$ and $j\neq \pm j'$,
\begin{equation}\label{eq:4.2.23ham}
\begin{aligned}
|i\oo\cdot\ell+\mu_{\s,j}^{\nu}({ u}_{2})-\mu_{\s',j'}^{\nu}({u}_{2})|
&\stackrel{(\ref{eq:419ham}),(\ref{eq:4.2.22ham})}{\geq}
\g|\s j^{2}-\s' j'^{2}|\langle\ell\rangle^{-\tau}\\
&-C|\s j^{2}-\s'j'^{2}| ||{ u}_{2}-{ u}_{1}||_{\gots_{0}+\h_{2}}\\
&\stackrel{{\bf (S4)_{\nu}}}{\geq}(\g-\rho)|\s j^{2}-\s' j'^{2}|\langle\ell\rangle^{-\tau},
\end{aligned}
\end{equation}
where we used that, for any $\la\in\Lambda_{0}$, one has $C\e N^{\tau}_{\nu}||{ u}_{1}-{ u}_{2}||_{\gots_{0}+\h_{2}}\leq \rho$ (note that this condition is weaker with respect to the hypothesis in $({\bf S4})_{\nu}$.
Now, the (\ref{eq:4.2.23ham}), 
 imply 
that if $\la\in \calP_{\nu+1}^{\g}({u}_{1})$ then 
$\la\in \calP_{\nu+1}^{\g-\rho}({ u}_{2})$.
Now assume that $\la\in \calO_{\nu+1}^{\g}(u_{1})$. We have two cases: if $|j|\geq 4|\oo| |\ell|/\e \gote$, then
we have no small divisors. Indeed one has
\begin{equation*}\label{nanonano}
\begin{aligned}
b_{j}^{\nu}(u)^{2}&=\left(-2|m_{1}|+\frac{r_{j}^{\nu,j}-r_{-j}^{\nu,-j}}{j}\right)^{2}+4\frac{|r_{j}^{\nu,-j}|^{2}}{|j|^{2}}\geq\left(2|m_{1}|-\frac{\e C}{|j|}\right)^{2}\\
&\stackrel{(\ref{eq:3.2.44})}{\geq}|m_{1}|^{2}\left(2-\frac{\e C}{|j|\e \gote}\right)^{2}\geq
|m_{1}|^{2}\left(2-\frac{\e \gote}{4|\oo||\ell|}\right)^{2}\geq \frac{|m_{1}|^{2}}{4}\geq \frac{(\e \gote)^{2}}{4},
\end{aligned}
\end{equation*}
for any $u$.
Hence it is obvious that 
\begin{equation}\label{eq:4.2.23bisham}
\begin{aligned}
|i\oo\cdot\ell+\mu_{\s,j}^{\nu}({ u}_{2})-\mu_{\s,-j}^{\nu}({u}_{2})|&\geq \frac{4|\oo||\ell|}{\e \gote}|b_{j}^{\nu}(u_{2})|-|\oo\cdot\ell|\\
&{\geq} |\oo||\ell|   \geq \frac{\g-\rho}{\langle \ell\rangle^{\tau}\langle j\rangle}.
\end{aligned}
\end{equation}
Let us consider the case $|j|\leq 4|\oo| |\ell|/\e \gote$: one has
\begin{equation*}\label{eq:4.2.23trisham}
\begin{aligned}
|i\oo\cdot\ell+\mu_{\s,j}^{\nu}({ u}_{2})-\mu_{\s,-j}^{\nu}({u}_{2})|
&\stackrel{(\ref{eq:419ham}),(\ref{eq:4.2.22ham})}{\geq}
\g\langle\ell\rangle^{-\tau}\langle j\rangle^{-1}-\e C|j| ||{ u}_{2}-{ u}_{1}||_{\gots_{0}+\h_{2}}\\
&{\geq}\frac{1}{\langle\ell\rangle^{\tau}\langle j\rangle}\left(\g-\e|j|^{2}{CN_{\nu}^{-\al+\tau+2}}\right)\geq\frac{\g-\rho}{\langle\ell\rangle^{\tau}\langle j\rangle}
\end{aligned}
\end{equation*}
 that is the ${\bf (S4)}_{\nu+1}$.

\noindent
\emph{Proof of Theorem \ref{KAMalgorithmham}}


We want apply Lemma \ref{teo:KAMham} to the linear operator
$\calL_{0}=\calL_{7}$ defined in (\ref{eq:3.5.9ham})
where 
$$
\RR_{0}:=\left(\begin{matrix}0 & q_{0}(\f,x) \\ -\bar{q}_{0}(\f,x)& 0 \end{matrix}\right)+\RR_{7},
$$
with $\RR_{7}$ defined in \eqref{eq:3.7.8ham}. One has that $\RR_{0}$ satisfies the $(iii)$ of Lemma \ref{lem:3.88}.
Then
\begin{equation}\label{eq:4.1.2ham}
\begin{aligned}
|\RR_{0}|_{\gots_{0}+\be}&\stackrel{(\ref{eq:3.2.7ham})}{\leq}\e C(\gots_{0}+\be)(1+||{\bf u}||_{\be+\gots_{0}+\h_{1},\g})
\stackrel{(\ref{eq:4.2ham})}{\leq}2\e C(\gots_{0}+\be), \qquad \Rightarrow \\
&N_{0}^{C_0}|\RR_{0}|_{\gots_{0}+\be}^{0}\g^{-1}\leq 1,
\end{aligned}
\end{equation}
if $\e\g^{-1}\leq \epsilon_{0}$ is small enough, that is the (\ref{eq:4.15ham}). Then we have to prove that 
in the set $\cap_{\nu\geq0}\Lambda_{\nu}^{\g}$
there exists
a final transformation 
%
\begin{equation}\label{eq:4.1.3ham}
\Phi_{\infty}=\lim_{\nu\to\infty}\tilde{\Phi}_{\nu}=\lim_{\nu\to\infty}\Phi_{0}\circ\Phi_{1}\circ\ldots\circ\Phi_{\nu}.
\end{equation}
and the normal form

\begin{equation}\label{eq:4.1.1ham}
\Omega_{\s,\und{j}}^{\infty}:=\Omega_{\s,\und{j}}^{\infty}(\la)=\lim_{\nu\to+\infty}\tilde{\Omega}^{\nu}_{\s,\und{j}}(\la)
=\tilde{\Omega}^{0}_{\s,\und{j}}(\la)+
\lim_{\nu\to+\infty}\left(\begin{matrix} i\s \tilde{r}_{j}^{\n,j} & i\s \tilde{r}_{j}^{\nu,-j} \\ i\s \tilde{r}_{-j}^{\nu,j} &
i\s \tilde{r}_{-j}^{\nu,-j}\end{matrix}\right).
\end{equation}
The proof that limits in \eqref{eq:4.1.3ham} and \eqref{eq:4.1.1ham} exist uses the bounds of Lemma \ref{teo:KAMham}. We refer the reader 
to \cite{FP} 
for more details.

%
%
%
%


The following Lemma gives us a connection between the Cantor sets defined in Lemma
\ref{teo:KAMham} and Theorem \ref{KAMalgorithmham}.
Again the proof is omitted since it is essentially the same of Theorem $4.27$ in Section 4 of \cite{FP}.

\begin{lemma}\label{lem:4.4ham}
One has that
\begin{equation}\label{eq:4.1.12ham}
\Lambda^{2\g}_{\infty}\subset\cap_{\nu\geq0}\Lambda_{\nu}^{\g}.
\end{equation}
\end{lemma}
%

Since one prove that in $\Lambda^{2\g}_{\infty}$ the limit in \eqref{eq:4.1.3ham} exists in norm $|\cdot|_{s,\g}$
one has
%
%
\begin{equation}\label{101ham}
\begin{aligned}
\calL_{\n}&\stackrel{(\ref{eq:4.16ham})}{=}\oo\cdot\del_{\f}\uno+\DD_{\n}+\RR_{\n}\stackrel{|\cdot|_{s,\g}}{\to}
\oo\cdot\del_{\f}\uno+\DD_{\infty}=:\calL_{\infty}, \\
& \DD_{\infty}:=diag_{(\s,j)\in C\times\ZZZ}\Omega_{\s,\und{j}}^{\infty}.
\end{aligned}
\end{equation}
and moreover
%
\begin{equation}\label{102ham}
\calL_{\infty}=\Phi_{\infty}^{-1}\circ\calL_{0}\circ\Phi_{\infty},
\end{equation}
that is the (\ref{eq:4.6ham}), 
while the (\ref{eq:4.5ham}) follows by the smallness in \eqref{eq:4.20ham} and the convergence
Finally, Lemma \ref{bubbole},
Lemma \ref{1.4} and (\ref{eq:4.8ham}) implies the bounds (\ref{eq:4.9ham}). 
This concludes the proof.
\EP

\zerarcounters
 \section{Inversion of the linearized operator}\label{sec:5ham}
 
 In this Section we prove the invertibility of $\calL(u)$, and consequently of $d_{u}F(u)$ (see \ref{lemmaccio}), 
by showing the appropriate tame estimates on the inverse. 
The following Lemma resume the results obtained in the previous Sections.

We have the following result.
\begin{lemma}\label{lemma5.8ham}
Let $\calL=W_{1}\calL_{\infty}W_{2}^{-1}$ where
\begin{equation}\label{eq:4.4.1ham}
 W_{i}=\VV_{i}\Phi_{\infty}, \quad 
 \VV_{1}:=\TT_{1}\TT_{2}\TT_{3}\rho\TT_{4}\TT_{5}\TT_{6}\TT_{7}, \quad
\VV_{2}=\TT_{1}\TT_{2}\TT_{3} \TT_{4}\TT_{5}\TT_{6}\TT_{7}.
\end{equation}
where $\VV_{i}$ and $\Phi_{\infty}$ are defined in Lemmata \ref{lem:3.88} and \ref{KAMalgorithmham}.
Let $\gots_{0}\leq s\leq q-\be-\h_{1}-2$, with $\h_{1}$ define in (\ref{eq:3.2.0ham}) and $\be$  in Theorem (\ref{KAMalgorithmham}). Then, for $\e\g^{-1}$ small enough, and
\begin{equation}\label{eq:4.4.2ham}
||{ u}||_{\gots_{0}+\be+\h_{1}+2,\g}\leq1,
\end{equation}
one has for any $\la\in\Lambda^{2\g}_{\infty}$,
\begin{equation}\label{eq:4.4.3ham}
\begin{aligned}
||W_{i}{h}||_{s,\g}+||W_{i}^{-1}{h}||_{s,\g}&
\leq C(s)\left(||{ h}||_{s+2,\g}+
||{u}||_{s+\be+\h_{1}+4,\g}||{ h}||_{\gots_{0},\g}\right),
\end{aligned}
\end{equation}
for $i=0,1$. Moreover, $W_{i}$ and $W_{i}^{-1}$ symplectic.
\end{lemma}

\begin{proof}
Each $W_i$  is composition of two operators, the $\VV_i$ satisfy the (\ref{eq:3.2.3ham}) while $\Phi_\infty$ satisfies (\ref{eq:4.8ham}). We use  
Lemma \ref{bubbole}
in order to pass to the operatorial norm. Then Lemma \ref{lem5} 
implies the bounds (\ref{eq:4.4.3ham}). Moreover the transformations $W_{i}$ and $W_{i}^{-1}$ symplectic because they are composition of symplectic transformations $\VV_{i}$,$\VV_{i}^{-1}$ and
 $\Phi_{\infty}$, $\Phi_{\infty}^{-1}$ .
 \end{proof}

\begin{proof}[{\bf Proof of Proposition \ref{teo2ham}}.]
Thanks to Lemma \ref{lemma5.8ham} the proof of Proposition \ref{teo2ham} is almost concluded.
We fix the constants $\h=\h_{1}+\be+2$ (the constant $\h$ has to be chosen) and $q>\gots_{0}+\h$.
Let $\Omega_{\s,j}^{\phantom{g},j}$ and $\Omega_{\s,j}^{-j}$ be the functions defined in \eqref{eq:4.1.1ham}, and consequently $\mu_{\s,j}$ the eigenvalues of the matrices $\Omega_{\s,\und{j}}$. Therefore
by Lemmata \ref{KAMalgorithmham} and \ref{lemma5.8ham} item $(i)$ in Proposition \ref{teo2ham} hold. 
\end{proof}

Now we prove the following Lemma that is the equivalent result of Lemma \ref{inverseofl} in the Hamiltonian case. 

\begin{lemma}\label{inverselinftyham}
For  ${ g}\in{\bf H}^{s}$, consider the equation
\begin{equation}\label{eq:4.4.7ham}
\calL_{\infty}({u}){h}={ g}.
\end{equation}
If $\oo\in \Lambda^{2\g}_{\infty}({\bf u})\cap \calP^{2\g}_{\infty}(u)$  (defined in \eqref{martina10ham} and \eqref{primedimham}),  then there exists a unique solution 
$\calL_{\infty}^{-1}{g}:={h}\in{\bf H}^{s}$. Moreover, for all Lipschitz family ${ g}:=
{g}(\oo)\in {\bf H}^{s}$ one has
\begin{equation}\label{eq:4.4.8ham}
||\calL_{\infty}^{-1}{ g}||_{s,\g}\leq C \g^{-1}||{ g}||_{s+2\tau+1,\g}.
\end{equation}
\end{lemma}

\begin{proof} 
One can follows the same strategy used for Lemma 
$5.44$ in \cite{FP} 
and conclude using Lemma \ref{finitema}.
\end{proof}

\begin{proof}[{\bf Proof of Lemma \ref{inverseofl}}]

A direct consequence of Lemma \ref{lemma5.8ham} is that, once one has conjugated the operator
$\calL$ in \eqref{lemmaccio4} to a block-diagonal operator $\calL_{\infty}$  in \eqref{eq:4.6ham} is essentially
trivial to invert it:

In order to conclude the proof of Lemma \ref{inverseofl} it is sufficient
to collect the results of Lemmata \ref{lemma5.8ham} and \ref{inverselinftyham}. In particular one uses \eqref{eq:4.4.3ham} and \eqref{eq:4.4.8ham}
to obtain the estimate
\begin{equation}\label{eq:4.4.25ham}
\begin{aligned}
||{h}||_{s,\g}&=\|W_{2}\calL_{\infty}^{-1}W_{1}^{-1}{ g}\|_{s,\g}\\
&\leq C(s)\g^{-1} \left( ||{ g}||_{s+2\tau+5,\g}+
||{ u}||_{s+4\tau+\be+10+\h_{1},\g}||{ g}||_{\gots_{0},\g}
\right),
\end{aligned}
\end{equation}
%
\end{proof}
Note that by Lemma \ref{lemmaccio} the estimates \eqref{eq:4.4.25ham} holds also for the linearized operator 
$d_{u}\calF(u)$.

\zerarcounters
\section{Measure estimates}\label{sec6ham}
In Section \ref{sec:3ham}, \ref{sec:4ham} and \ref{sec:5ham} we prove that in the set
$\Lambda^{2\g}_{\infty}({ u}_{n})\cap \calP^{2\g}_{\infty}({ u}_n)$ we have good bounds 
on the inverse of $\calL(u_{n})$. 
We also give a precise characterization of this set in terms of the eigenvalues of $\calL$.
Now in the Nash-Moser proposition \ref{teo4} we defined in an implicit way the sets $\calG_n$ in order to ensure bounds on the inverse of  $\calL({ u}_{n})$. 
In this section we prove Proposition \ref{STIMEmisura} which  is the analogous
analysis performed in Section $6$ of \cite{FP}.
\begin{proposition}[{\bf Measure estimates}]\label{measurebruttebrutte}
Set $\g_n:=(1+2^{-n})\g$ and consider the set $\calG_{\infty}$ of Proposition \ref{teo4}
with $\mu=\zeta$ defined in Lemma \ref{inverseofl} and fix $\g:=\e^{a}$ for some $a\in(0,1)$. We have
\begin{subequations}\label{eq136totham}
\begin{align}
&\cap_{n\geq0}\calP^{2\g_n}_{\infty}({u}_n)\cap \Lambda^{2\g_n}_{\infty}({ u}_n)\subseteq \calG_{\infty}, \label{eq136bham}\\
&|\Lambda\backslash\calG_{\infty}|\to 0, \;\; {\rm as} \;\; \e\to0. \label{eq136ham}
\end{align}
\end{subequations}
\end{proposition}

{\bf Proof of Proposition \ref{measurebruttebrutte}}.
 Let $({ u}_{n})_{\geq0}$ be the sequence of approximate solutions introduced in Proposition  
\ref{teo4} which  is well defined in $\calG_{n}$ and satisfies the hypothesis
of Proposition \ref{teo2ham}.   $\calG_{n}$ in turn is defined in 
Definition \ref{invertibility}. 
We now define inductively a sequence of nested sets $G_{n}\cap H_{n}$ for $n\geq0$. 
Set $G_{0}\cap H_{0}=\Lambda$ and 
\begin{equation}\label{eq142bis}
\begin{aligned}
G_{n+1}&:=
\left\{\begin{aligned}
 \oo\in G_{n} \; :\; &|i\oo\cdot\ell+\mu_{\s,j}({ u}_{n})-\mu_{\s',j'}({ u}_{n})|
\geq\frac{2\g_{n}|\s j^{2}-\s' j'^{2}|}{\langle\ell\rangle^{\tau}}, \\
&\;\forall\ell\in\ZZZ^{n},\;\;  \s,\s'\in\CC, \;\; j,j'\in\ZZZ
\end{aligned}\right\},\\
H_{n+1}&:=
\left\{\begin{aligned}
 \oo\in H_{n} \; :\; &|i\oo\cdot\ell+\mu_{\s,j}({ u}_{n})-\mu_{\s,-j}({ u}_{n})|
\geq\frac{2\g_{n}}{\langle\ell\rangle^{\tau}\langle j\rangle}, \\
&\;\forall\ell\in\ZZZ^{n}\backslash\{0\},\;\;  \s\in\CC, \;\; j\in\ZZZ
\end{aligned}\right\},\\
P_{n+1}&:=
\left\{\begin{aligned}
 \oo\in P_{n} \; :\; &|i\oo\cdot\ell+\mu_{\s,j}({ u}_{n})|
\geq\frac{2\g_{n}\langle j\rangle^{2}}{\langle\ell\rangle^{\tau}}, \\
&\;\forall\ell\in\ZZZ^{n},\;\;  \s\in\CC, \;\; j\in\ZZZ
\end{aligned}\right\},
\end{aligned}
\end{equation}
Recall that $\mu_{\s,j}(u_{n})$ and $\mu_{\s,-j}(u_{n})$ are the eigenvalues of the matrices $\Omega_{\s,\und{j}}$
defined in Proposition \ref{teo2ham} in \eqref{1.2.2bis}.
The following Lemma implies \eqref{eq136bham}.
\begin{lemma}\label{megalemma} Under the Hypotheses of Proposition \ref{measurebruttebrutte},
for any $n\geq0$, one has
\begin{equation}\label{eq137}
P_{n+1}\cap G_{n+1}\cap H_{n+1}\subseteq \calG_{n+1}.
\end{equation}
\end{lemma}

\begin{proof}
For any $n\geq0$
and if $\la\in G_{n+1}$,  one has 
by Lemmata
\ref{inverselinftyham} and \ref{inverseofl}, (recalling that $\g\leq \g_{n}\leq 2\g$ and $2\tau+5<\zeta$)
\begin{equation}\label{eq139}
\begin{aligned}
||\calL^{-1}({u}_{n}){h }||_{s,\g}&\leq C(s)\g^{-1}\left(
||{ h}||_{s+\zeta,\g}+||{u}_{n}||_{s+\zeta,\g}||{h}||_{\gots_{0},\g}\right),\\
||\calL^{-1}({u}_{n})||_{\gots_{0},\g}&\leq C(\gots_{0})\g^{-1}
N_{n}^{\zeta}||{ h}||_{\gots_{0},\g},
\end{aligned}
\end{equation}
for $\gots_{0}\leq s\leq q-\mu$, for any ${h}(\la)$ Lipschitz family. 
The (\ref{eq139}) are nothing but the (\ref{eq104}) in Definition \ref{invertibility} with $\mu=\zeta$ .
It represents the loss of regularity that you have when you perform the
regularization procedure in Section \ref{sec:3ham} and during the diagonalization algorithm in
Section \ref{sec:4ham}. This justifies our choice of $\mu$ in Proposition \ref{measurebruttebrutte}.
\end{proof}

\noindent
Now we prove formula \eqref{eq136ham} that is the most delicate point. 
It turns out, by an explicit computation, that we can write for $j\neq0$,
\begin{equation}\label{luna2}
\mu_{\s,j}-\mu_{\s,-j}{:=}i\s\sqrt{
(-2|m_{1}|j+r_{j}^{j}-r_{-j}^{-j})^{2}+4|r_{j}^{-j}|}:=j b_{j}=jb_{j}(u_{n}),
\end{equation}
where $r_{j}^{k}$, for $j,k\in \NNN$ are the coefficients of the matrix $R_{\s,\und{j}}$ in \eqref{1.2.2bisham}, 
and we define
\begin{equation}\label{luna3}
\psi(\oo,u_{n}):={\oo}\cdot\ell+j b_{j}(u_{n}).
\end{equation}
Now we write
for any $\ell\in \ZZZ^{d}\backslash\{0\}$ and $j\in \ZZZ$, 
\begin{equation}\label{luna}
H_{n}:=\bigcap_{\substack{\s\in\CC, \\(\ell,j)\in\ZZZ^{d+1}}}\!\!\! A^{\s}_{\ell,j}(u_{n})
:=\bigcap_{\substack{\s\in\CC,\\ (\ell,j)\in\ZZZ^{d+1}}}\left\{\oo\in H_{n-1} : |i\oo\cdot\ell+j b_{j}(u_{n})|\geq\frac{\g_{n}}{\langle j\rangle\langle\ell\rangle^{\tau}}
\right\}.
\end{equation}

Clearly one need to estimate the measure of $\bigcap_{n\geq0}H_{n}$. The strategy to get such estimate is quite standard and it is the following:
\begin{itemize}
\item[{\bf a}.] First one give an estimate of the resonant set for fixed $(\s,j,\ell)\in \CC\times\ZZZ\times\ZZZ^{d}$ (namely $(A^{\s}_{\ell,j})^{c}$).
This point require a lower bound
 on the Lipschitz sub-norm of the function $\psi$ in \eqref{luna3}. In this way we can give an estimate of the measure of the bad set using the standard arguments to estimate the measure of sub-levels of Lipschitz functions. This is in general non trivial but
in the case of the sets $G_{n}$ and $P_{n}$ there is a well established strategy to follow
that uses that $\mu_{\s,j}\sim O(j^{2})$. In the case of the sets $H_{n}$ the problem
is more difficult since $\mu_{\s,j}\sim O(\e j)$, hence, even if $j$ is large, it could happen that $\mu_{\s,j}\sim \oo\cdot\ell$.
However we prove such lower bound (see \eqref{luna10}) using result of Lemma \ref{nani2} and 
non-degeneracy condition on $m_{1}$ (see \eqref{mammamia}). Moreveor we use deeply the fact that we have $d$ parameters $\oo_{i}$ for $i=1,\ldots,d$ to move. On the contrary in Section $6$ of \cite{FP}
 the authors 
 performed the estimates by choosing a diophantine direction 
$\bar{\oo}$ and using as frequency the vector $\oo=\la \bar{\oo}$ with  $\la\in[1/2,3/2]$, hence using  just  one parameter. 
In this case this is not possible.
\item[{\bf b}.] Item ${\bf a.}$ provides and estimate like $|(A^{\s}_{j,\ell})^{c}|\sim \g/(j|\ell|^{\tau})$.
The second point is to have some summability of the series in $j$ since one need to control 
$\bigcup_{j,\ell}(A^{\s}_{\ell,j})^{c}$. One can sum over $\ell$ by choosing  $\tau$ large enough. In principle on can think
to weaker the Melnikov conditions and ask for a lower bound of the type 
\begin{equation}\label{albero}
|\psi|\geq\g/|j|^{2}|\ell|^{\tau}.
\end{equation}
This can cause two problems. If one ask \eqref{albero} it may be very difficult to prove the lower bound on the Lipschitz norm.
Secondly in the reduction algorithm one must have a remainder $\RR$ that support 
the loss of $2$ derivatives in the space. Our strategy is different: we use results in  Lemmata \ref{luna4} and \ref{luna6} to prove
that the number of $j$ for which $(A^{\s}_{\ell,j})^{c}\neq \emptyset$ is controlled by $|\ell|$.
\item[{\bf c}.] Finally one has to prove some ``relation'' between the sets $H_{n}$ and $H_{k}$ for $k\neq n$. Indeed
the first two points imply only that the set $H_{n}$ has large measure as $\e\to 0$. But in principle as $n$ varies this sets can be unrelated, so that the intersection can be empty. Roughly speaking in Lemma \ref{luna11} we prove that 
lots of resonances at the step $n$ have been already removed at the step $n-1$. In other words we prove that, if $|\ell|$ is sufficiently small, if $\psi(u_{n-1})$ satisfies the Melnikov conditions, then also $\psi(u_{n})$ automatically has the good bounds.
Again this point is different from the case studied in Section $6$ of \cite{FP}.
Indeed with double eigenvalues one is able to prove the previous claim 
only for $n$ large enough and not for any $n$. This is the reason in this case the set of good parameters is small, but in any case of full measure.
\end{itemize}
\noindent
In the following Lemma we resume the key result one need to prove Proposition \ref{measurebruttebrutte}.
\begin{lemma}\label{cavallo2}
For any $n\geq0$ one has
\begin{equation}\label{frodo5}
|P_{n}\backslash P_{n+1}|,|G_{n}\backslash G_{n+1}|,|H_{n}\backslash H_{n+1}|\leq C\sqrt{\g}.
\end{equation}
Moreover, if $n\geq\bar{n}(\e)$ (where $\bar{n}(\e)$ is defined in Lemma \ref{luna11}), then one has
\begin{equation}\label{frodo6}
|P_{n}\backslash P_{n+1}|,|G_{n}\backslash G_{n+1}|,|H_{n}\backslash H_{n+1}|\leq C\sqrt{\g} N_{n}^{-1}.
\end{equation}
In particular $\bar{n}(\e)$ has the form
\begin{equation}\label{frodo10}
\bar{n}(\e):=a {\rm log}{ {\rm log}\left[b\frac{1}{c \g \e}\right]},
\end{equation}
with $a,b,c>0$ independent on $\e$.
%
%
\end{lemma}

By Lemma \ref{cavallo2} follows the \eqref{eq136}. Indeed on one hand we have
\begin{equation}\label{frodo11}
\begin{aligned}
|\Lambda\backslash \cap_{n\geq0} H_{n}|&\leq
\sum_{n=0}^{\bar{n}(\e)}|H_{n}\backslash H_{n+1}|+\sum_{n> \bar{n}(\e)}|H_{n}\backslash H_{n+1}|\leq C \g\bar{n}(\e).
\end{aligned} 
\end{equation}
The same bounds holds for $|\Lambda\backslash \cap_{n\geq0}G_{n}|,|\Lambda\backslash \cap_{n\geq0}P_{n}|.$
Now, fixing $\g:=\g(\e)=\e^{a}$ with $a\in (0,1)$, one has that 
$$
|\Lambda\backslash \calG_{\infty}|\leq C\sqrt{\g(\e)}(1+\bar{n}(\e))\to0, \quad {\rm as} \quad \e\to0.
$$
This concludes the proof of Proposition \ref{measurebruttebrutte}.  It remains to check Lemma \ref{cavallo2} following the strategy in three point explained above.
We will give the complete proof only for the sets $H_{n}$ that is more difficult. 
The inductive estimates on $G_{n}$ and $P_{n}$ is very similar, anyway one can follows essentially word by word  the proof of Proposition $1.10$ in Section $6$ of \cite{FP}. Similar measure estimates can be also found in \cite{BBM}. 


\begin{lemma}\label{luna4}
If $|b_{j}||j|\geq 2|{\oo}\cdot\ell|$ or $|b_{j}||j|\leq |{\oo}\cdot\ell|/2$ then 
$(A^{\s}_{\ell,j}(u_{n}))^{c}=\emptyset$.
\end{lemma}
\begin{proof} Lemma follows by the fact that ${\oo}$ is diophantine with constant $\tau_{0}$ and $\tau>\tau_{0}$
and from the  smallness of $|m_{1}|$.
\end{proof}

\noindent
Thanks Lemma \ref{luna4} in the following we will consider only the $j\in \SSSS_{\ell,n}\subseteq\ZZZ$ where
\begin{equation}\label{luna5}
\SSSS_{\ell,n}:=\left\{j\in\ZZZ \; \frac{|{\oo}\cdot\ell|}{2}\leq |j|b_{j}(u_n)\leq2|{\oo}\cdot\ell|
\right\}
\end{equation}
for some constant $C>0$. In order to estimate the measure of $(A^{\s}_{\ell,j}(u_{n}))^{c}$ we need the following technical Lemma.
\begin{lemma}\label{luna6}
If $j\in \SSSS_{\ell,n}\cap (A_{\ell,n})^{c}$, where
$$A_{\ell}:=\{j\in\ZZZ : |j|\leq4|\ell| C/\gote\},$$ then
one has that $|b_{j}(u_{n})|\geq |m_{1}(u_{n})|/2$. 
\end{lemma}
\prova
It follow by
\begin{equation}\label{nano}
\begin{aligned}
b_{j}^{2}&=\left(-2|m_{1}|+\frac{r_{j}^{j}-r_{-j}^{-j}}{j}\right)^{2}+4\frac{|r_{j}^{-j}|^{2}}{|j|^{2}}\geq\left(2|m_{1}|-\frac{\e C}{|j|}\right)^{2}\\
&\stackrel{(\ref{eq:3.2.44})}{\geq}|m_{1}|^{2}\left(2-\frac{\e C}{|j|\e \gote}\right)^{2}\geq
|m_{1}|^{2}\left(2-\frac{1}{4|\ell|}\right)^{2}\geq \frac{|m_{1}|^{2}}{4}.
\end{aligned}
\end{equation}
\EP

\noindent
 An consequence of Lemmata \ref{luna4} and \ref{luna6} is that we need to study the sets $A^{\s}_{\ell,j}$ only for 
\begin{equation}\label{padre}
|j|\leq \frac{C|\ell|}{\e \gote}.
\end{equation}
It is essentially what explained in item ${\bf b}.$ Note the here we used the non-degeracy of the constant $m_{1}$.

\begin{lemma}\label{nani2}
For any $n\geq0$ and $j\in \SSSS_{\ell,n}$ one has
\begin{equation}\label{nani3}
|b_{j}(u_n)|^{lip}\leq K \frac{1}{|j|}\left[|m_{1}|^{lip}|j|+\e C\right],
\end{equation}
for some $K>0$.
\end{lemma}
\prova
One can note that,
\begin{equation}\label{nani4}
\begin{aligned}
|b_{j}(\oo_1)-&b_{j}(\oo_{2})|=\left|\frac{ b^2_{j}(\oo_1)-b^{2}_{j}(\oo_{2})}{b_{j}(\oo_1)+b_{j}(\oo_{2})}\right|\leq\\
&\leq|\oo_{1}-\oo_{1}|\left[|m_{1}|^{lip}+\frac{1}{|j|}(|r_{j}^{j}|^{lip}+|r_{-j}^{-j}|^{lip}+|r_{j}^{-j}|^{lip})\right],
\end{aligned}
\end{equation}
using  that
\begin{equation}\label{nani5}
\frac{|(-2|m_{1}(\oo_{1})|+(r_{j}^{j}-r_{-j}^{-j})(\oo_{1})/j)|+|(-2|m_{1}(\oo_{1})|+(r_{j}^{j}-r_{-j}^{-j})(\oo_{1})/j)|}{b_{j}(\oo_1)+b_{j}(\oo_{2})}\leq2,
\end{equation} 
and that the same bound holds also for $|(r_{j}^{-j})(\oo_1)|/|j|(b_{j}(\oo_1)+b_{j}(\oo_{2}))$.
\EP

\noindent
An immediate consequence of \eqref{nani3} is that 
\begin{equation}\label{nani6}
|j||b_{j}|^{lip}\stackrel{(\ref{eq:3.2.44})}{\leq}4|\ell|\frac{C}{\gote}2K\e C, \qquad j\in \SSSS_{\ell,n}\cap A_{\ell}
\end{equation}
\begin{equation}\label{nani7}
|j||b_{j}|^{lip}\stackrel{(\ref{mammamia})}{\leq}K|j|\frac{1}{|j|}\left[\e|m_{1}(0)|C\frac{|\ell|}{\e \gote}+\e C\right]
\leq \tilde{K}\e |\ell|, \qquad j\in\SSSS_{\ell,n} \cap(A_{\ell})^{c}
\end{equation}

\noindent
By Lemmata \ref{luna6} and \ref{nani2} we deduce the following fundamental estimates on the function $\psi$ defined in  \eqref{luna3}. First we note that,
since there exists $i\in\{1,\ldots,d\}$ such that $|\ell_{i}|\geq |\ell|/2d$, one has
$$
|\del_{\oo_i}\oo\cdot\ell|\geq\frac{|\ell|}{2d}.
$$
Hence one has
\begin{equation}\label{luna10}
\begin{aligned}
|\psi|^{lip}&\geq \left(\frac{|\ell|}{2d}-|j||b_{j}|^{lip}
\right)\stackrel{(\ref{nani6})}{\geq}\frac{|\ell|}{4d},
\end{aligned}
\end{equation}
for $\e$ small enough for any $j\in \SSSS_{\ell,n}$. The \eqref{luna10} is fundamental in order to estimate 
the measure of a single resonant set and this is what we claimed in item ${\bf a}$. The following Lemma 
is the part ${\bf c}.$ of the strategy,

\begin{lemma}\label{luna11}
For $|\ell|\leq N_{n}$ one has that for any $\e>0$ there exists $\bar{n}:=\bar{n}(\e)$ such that  
if $n\geq\bar{n}(\e)$ then
\begin{equation}\label{luna122}
(A^{\s}_{\ell,j}(u_{n}))^{c}\subseteq (A^{\s}_{\ell,j}(u_{n-1}))^{c}.
\end{equation}
\end{lemma}


\begin{proof}
We first have to estimate
\begin{equation}\label{speriamobene}
\begin{aligned}
|j||b_{j}(u_{n})-b_{j}(u_{n-1})|&\leq
4\max_{h=\pm j}\{|r_{j}^{-h}(u_{n})-r_{j}^{ -h}(u_{n-1})|\}\\
&+2|m_{1}(u_{n})-m_{1}(u_{n-1})| |j|.
\end{aligned}
\end{equation}
By Lemma \ref{teo:KAMham}, using  the $({\bf S4})_{n+1}$ with
$\g=\g_{n-1}$ and $\g-\rho=\g_{n}$,  and with ${ u}_{1}={u}_{n-1}$, ${ u}_{2}={u}_{n}$,
we have
\begin{equation}\label{eq162}
\Lambda_{n+1}^{\g_{n-1}}({ u}_{n-1})\subseteq\Lambda_{n+1}^{\g_{n}}({ u}_{n}),
\end{equation}
since, for $\e\g^{-1}$ small enough, and $n\geq \bar{n}(\e)$ defined as
\begin{equation}
\bar{n}(\e):=\frac{1}{\log(3/2)}\log\left[\frac{1}{(\ka-\tau-3)\log N_{0}}\log\left(\frac{1}{C\g \e}\right)\right]
\end{equation}
\begin{equation}\label{eq163}
 C N_{n}^{\tau}\sup_{\la\in G_{n}}||{ u}_{n}-{ u}_{n-1}||_{\gots_{0}+\mu}
\leq \e (\g_{n-1}-\g_{n})=:\e \rho=\e \g2^{-n}.
\end{equation}
where $\ka$ is defined in (\ref{teo41}) with
$\nu=2$, $\mu=\zeta$ defined in (\ref{eq:4.4.18}) with $\h=\h_{1}+\be$, $\mu>\tau$ 
(see Lemmata \ref{measurebruttebrutte}, \ref{megalemma} and (\ref{eq:4.14ham}), (\ref{eq:3.2.0ham})).
We also  note that,
\begin{equation}\label{eq164}
G_{n}\cap H_{n}\stackrel{(\ref{eq142bis}), (\ref{martina10ham})}{\subseteq}\Lambda_{\infty}^{2\g_{n-1}}({u}_{n-1})
\stackrel{(\ref{eq:4.1.12ham})}{\subseteq}\Lambda_{n+1}^{\g_{n-1}}({ u}_{n-1})
\stackrel{(\ref{eq162})}{\subseteq}
\Lambda_{n+1}^{\g_{n}}({ u}_{n}).
\end{equation}
This means that 
$\la\in H_{n}\cap G_{n}\subset \Lambda_{n+1}^{\g_{n-1}}({u}_{n-1})\cap\Lambda_{n+1}^{\g_{n}}({ u}_{n})$,
and hence, we can apply the ${\bf (S3)}_{\nu}$, with $\nu=n+1$, in Lemma \ref{teo:KAMham} to get
for any $h,k=\pm j$,
\begin{equation*}\label{eq165}
\begin{aligned}
\!\!\!\!\!\!|r_{h}^{k}({ u}_{n})-&r_{h}^{k}({ u}_{n-1})|\leq
|r_{h}^{n+1,k}({ u}_{n})-r_{h}^{n+1,k}({ u}_{n-1})|\\
&+|r_{h}^{k}({ u}_{n})-r_{h}^{n+1,k}({ u}_{n})|
+
|r_{h}^{k}({ u}_{n-1})-r_{h}^{n+1,k}({ u}_{n-1})|\\
\stackrel{(\ref{eq:3.2.6aham}),(\ref{eq:4.24bisham}),(\ref{teo41})}\leq& C\e^{2}\g^{-1} N_{n}^{-\ka}
+\e\left(1+||{u}_{n-1}||_{\gots_{0}+\h_{1}+\be}+||{ u}_{n}||_{\gots_{0}+\h_{1}+\be}\right)N_{n}^{-\al}.
\end{aligned}
\end{equation*}
Now, first of all $\ka>\al$ by (\ref{teo41}), (\ref{eq:4.14ham}), moreover
$\h_{1}+\be<\h_{5}$ then by ${\bf (S1)}_{n}$, $({\bf S1})_{n-1}$, one has 
$||{u}_{n-1}||_{\gots_{0}+\h_{5}}+||{u}_{n}||_{\gots_{0}+\h_{5}}\leq 2$, we obtain
\begin{equation}\label{eq166}
|r_{h}^{k}({u}_{n})-r_{h}^{k}({ u}_{n-1})|\stackrel{(\ref{eq165})}{\leq}
\e N_{n}^{-\al}.
\end{equation}
Then, by (\ref{speriamobene}), (\ref{eq:3.2.44}) and (\ref{eq166}) one has that 
\begin{equation}\label{eq167}
|(\mu_{\s,j}-\mu_{\s,-j})(u_{n})-(\mu_{\s,j}-\mu_{\s,-j})(u_{n-1})|\leq C\e |j| N_{n}^{-\al},
\end{equation}
hence  for $|\ell|\leq N_{n}$, and $\la\in G_{n}\cap H_{n}$, we have
\begin{equation}\label{eq168}
\begin{aligned}
|i\oo\cdot\ell+\mu_{\s,j}({ u}_{n})-\mu_{\s,j}({ u}_{n})|
&\stackrel{(\ref{eq167})}{\geq}\frac{2\g_{n-1}}{\langle\ell\rangle^{\tau}\langle j\rangle}
-C\e |j|N_{n}^{-\al}
\geq\frac{2\g_{n}}{\langle\ell\rangle^{\tau}\langle j\rangle},
\end{aligned}
\end{equation}
since $j\in\SSSS_{\ell,n}$, hence $|j|\leq 4|\oo||\ell|/\e\gote$, and $n$ is such that $N_{n}^{\tau-\al+2}\lessdot\g2^{-n}\e$.
The \eqref{eq168} implies the \eqref{luna122}.
\end{proof}
\noindent
An immediate consequence of Lemma \ref{luna11} is the following.
\begin{proof}{\emph{Proof of Lemma \ref{cavallo2}.}}
First of all, write
\begin{equation}\label{luna14}
\begin{aligned}
&H_{n}\backslash H_{n+1}:=\bigcup_{\substack{\s\in\CC, j\in\ZZZ \\ \ell\in\ZZZ^{d}}} 
(A_{\ell, j}^{\s}({ u}_{n}))^{c}.\\
\end{aligned}
\end{equation}
By using Lemma \ref{luna11} and equation \eqref{luna5},
we obtain 
\begin{equation}\label{luna15}
\begin{aligned}
&H_{n}\backslash H_{n+1}\subseteq H_{n}^{(1)}\cup H_{n}^{(2)}\cup H_{n}^{(3)}\cup H_{n}^{(4)}\\
&H_{n}^{(1)}:=\Big(\bigcup_{\substack{\s\in\CC,\\
 \; 
j\in \SSSS_{\ell}\cap A_{\ell} \\ |\ell|\leq N_{n} }}(A_{\ell, j}^{\s}({ u}_{n}))^{c}\Big),
\quad  H_{n}^{(2)}:=\Big(\bigcup_{\substack{\s\in\CC,\\
 \; 
j\in \SSSS_{\ell}\cap A_{\ell} \\ |\ell|> N_{n} }} 
(A_{\ell, j}^{\s}({ u}_{n}))^{c}\Big), \\
&H_{n}^{(3)}:=\Big(\bigcup_{\substack{\s\in\CC,\\
 \; 
j\in \SSSS_{\ell}\cap (A_{\ell})^{c} \\ |\ell|\leq N_{n} }} 
(A_{\ell, j}^{\s}({ u}_{n}))^{c}\Big), \quad H_{n}^{(4)}:=\Big(\bigcup_{\substack{\s\in\CC,\\
 \; 
j\in \SSSS_{\ell}\cap (A_{\ell})^{c} \\ |\ell|> N_{n} }} 
(A_{\ell, j}^{\s}({ u}_{n}))^{c}\Big).
\end{aligned}
\end{equation}
One has that the cardinality if the set $\SSSS_{\ell,n}\cap A_{\ell}$ is less than $4|\ell|C/\gote$. This implies that
\begin{equation}\label{stime1}
|H^{(2)}|\leq\sum_{|\ell|> N_{n}}\frac{4|\ell|C\g_{n}}{\gote\langle j\rangle\langle\ell\rangle^{\tau}}\frac{4d}{|\ell|}\lessdot
C\g N_{n}^{-1}.
\end{equation}
Let us estimate the measure of the sets $H^{(i)}$ for $i=3,4$.
The cardinality of $\SSSS_{\ell,n}\cap(A_{\ell})^{c}$ is less than $K|\ell|/\e \gote$, hence we have to study the case $j\in\SSSS_{\ell,n}\cap(A_{\ell})^{c}$ more carefully. We introduce the sets  
\begin{equation}\label{nani77}
B_{\ell,j}^{\s}:=\left\{\oo\in H_{n-1} : |i\oo\cdot\ell+j b_{j}(u_{n})|\geq\frac{\g'_{n}\al_n}{ \langle\ell\rangle^{\tau_{1}}}
\right\},
\end{equation}
for $\ell\in\ZZZ^{d}\backslash\{0\}$, $j\in\SSSS_{\ell,n}\cap (A_{\ell})^{c}$,
where $\al_{n}:={\rm inf}_{j}|b_{j}(u_{n})|$, $\g'_{n}=(1+2^{-n})\g'$, $\g'\leq\g_0$ and $\tau_1>0$.
We have the following result.
\begin{lemma}\label{nani10}
Given $\g'$ and $\tau_1$, there exist $\g$ and $\tau$ such that
if $\la\in B_{\ell,j}^{c}$ then $\la\in A_{\ell,j}^{\s}$ for $j\in\SSSS_{\ell,n}\cap (A_{\ell})^c$.
\end{lemma}
\begin{proof}
First of all
\begin{equation*}
j\in\SSSS_{\ell}, \;\; \Rightarrow b_{j}\geq\frac{|\oo\cdot\ell|}{2|j|}, \;\; \Rightarrow \al_{n}\geq \frac{\g_{0}}{2\langle\ell\rangle^{\tau_0}\langle j\rangle},
\end{equation*}
hence
\begin{equation*}
|\oo\cdot\ell+jb_{j}|\geq \frac{\g'_{n}\al_{n}}{\langle\ell\rangle^{\tau_1}}\geq \frac{\g'_n\g_{0}}{\langle j\rangle\langle\ell\rangle^{\tau_{1}+\tau_0}2}\geq \frac{\g_{n}}{\langle j\rangle\langle \ell\rangle^{\tau}},
\end{equation*}
if $\g'\g_{0}\geq 2\g$ and $\tau\geq \tau_{1}+\tau_0$.
\end{proof}

\noindent
By Lemma \ref{nani10} follows that 
\begin{equation}\label{frodo3}
|H_{n}^{(4)}|\leq \sum_{|\ell|>N_{n}}\sum_{j\in\SSSS_{\ell,n}\cap (A_{\ell})^{c}}|B_{\ell j}^{\s}|\leq
 \sum_{|\ell|>N_{n}}\frac{4|\ell|K\g'_{n}\al_{n}}{\e\gote \langle\ell\rangle^{\tau_{1}}}\frac{4d}{|\ell|}
 \lessdot C\g' N_{n}^{-1}
\end{equation}

Unfortunately, for the sets $H_{n}^{(1)}$ and $H_{n}^{(3)}$ we cannot provide an estimate like \eqref{frodo3}; by the summability of the series in $\ell$ we can only conclude
\begin{equation}\label{frodo4}
|H_{n}^{(1)}|, |H^{(3)}_{n}|\leq C\g'.
\end{equation}
This implies the \eqref{frodo5} for any $n\geq0$. Moreover by Lemma \ref{luna11} we have that
if $n\geq \bar{n}(\e)$ then $H_{n}^{(1)}=H_{n}^{(3)}=\emptyset$, hence the \eqref{frodo6} follows by \eqref{stime1} and \eqref{frodo3}.
Lemma \ref{cavallo2} implies
 (\ref{eq136ham}) by choosing, for instance, $\g:=(\g')^{2}\leq\g_0\leq1$.
 \end{proof}

\appendix
\section{Technical Lemmata}\label{techtech}
%
%
%
%
Now we recall classical tame estimates for composition of functions.

\begin{lemma}\label{lemA2}{\bf Composition of functions} Let $f : \TTT^{d}\times B_{1}\to\CCC$, where
$B_{1}:=\left\{y\in\RRR^{m} : |y|<1\right\}$. it induces the composition operator on $H^{s}$
\begin{equation}\label{A11}
\tilde{f}(u)(x):=f(x,u(x),Du(x),\ldots,D^{p}u(x))
\end{equation}
where $D^{k}$ denotes the partial derivatives $\del_{x}^{\al}u(x)$ of order $|\al|=k$.

Assume $f\in C^{r}(\TTT^{d}\times B_{1})$. Then

\noindent
$(i)$ For all $u\in H^{r+p}$ such that $|u|_{p,\infty}<1$, the composition operator $(\ref{A11})$
is well defined and
\begin{equation}\label{A12}
||\tilde{f}(u)||_{r}\leq C||f||_{C^{r}}(||u||_{r+p}+1),
\end{equation}
where the constant $C$ depends on $r,p,d$. If $f\in C^{r+2}$, then, for all $|u|^{\infty}_{s },
|h|^{\infty}_{p }<1/2$, one has
\begin{equation}\label{A13}
\begin{aligned}
||\tilde{f}(u+h)-\tilde{f}(u)||_{r}&\leq C||f||_{C^{r+1}}(||h||_{r+p}+|h|^{\infty}_{p }||u||_{r+p}),\\
||\tilde{f}(u+h)-\tilde{f}(u)-\tilde{f}'(u)[h]||_{r}&\leq
C||f||_{C^{r+2}}|h|^{\infty}_{p}(||h||_{r+p}+|h|^{\infty}_{p }||u||_{r+p}).
\end{aligned}
\end{equation}

\noindent
$(ii)$ the previous statement also hold replacing $||\cdot||_{r}$ with the norm $|\cdot|_{\infty}$.
\end{lemma}

\begin{proof} For the proof see \cite{Ba2} and \cite{Moser-Pisa-66}.
\end{proof}


\begin{lemma}\label{change}{\bf (Change of variable)} Let $p : \RRR^{d}\to \RRR^{d}$
be a $2\pi-$periodic function in $W^{s,\infty}$, $s\geq1$, with $|p|^{\infty}_{1 }\leq1/2$.
Let $f(x)=x+p(x)$. Then one has
$(i)$ $f$ is invertible, its inverse is $f^{-1}(y)=g(y)=y+q(y)$ where $q$ is $2\pi-$periodic,
$q\in W^{s,\infty}(\TTT^{d};\RRR^{d})$ and $|q|^{\infty}_{s }\leq C|p|^{\infty}_{s }$. More precisely,
\begin{equation}\label{A18}
|q|_{L^{\infty}}=|p|_{L^{\infty}}, \; |Dq|_{L^{\infty}}\leq 2|Dp|_{L^{\infty}},
\; |Dq|^{\infty}_{s-1 }\leq C|Dp|^{\infty}_{s-1 },
\end{equation}
where the constant $C$ depends on $d,s$. 

Moreover, assume that $p=p_{\la}$ depends in a Lipschitz way by a parameter 
$\la\in\Lambda\subset\RRR^{d}$, an suppose, as above, that $|D_{x}p_{\la}|_{L^{\infty}}\leq1/2$
for all $\la$. Then $q=q_{\la}$ is also Lipschitz in $\la$, and
\begin{equation}\label{A19}
|q|^{\infty}_{s,\g}\leq C\left(|p|^{\infty}_{s,\g}+
\left[\sup_{\la\in\Lambda}|p_{\la}|^{\infty}_{s+1 }\right]
|p|_{L^{\infty},\g}
\right)\leq C|p|^{\infty}_{s+1,\g},
\end{equation}
the constant $C$ depends on $d,s$ (it is independent on $\g$).

$(ii)$ If $u\in H^{s}(\TTT^{d};\CCC)$, then $u\circ f(x)=u(x+p(x))\in H^{s}$, and, with the same $C$ as
in $(i)$ one has
\begin{subequations}\label{A20}
\begin{align}
||u\circ f-u||_{s}&\leq C(|p|_{L^{\infty}}||u||_{s+1}+|p|^{\infty}_{s }||u||_{2}),\label{A20b}\\
||u\circ f||_{s,\g}&\leq 
C(|u|_{s+1,\g}+|p|^{\infty}_{s,\g}||u||_{2,\g})\label{A20c}.
\end{align} 
\end{subequations}
The 
(\ref{A20b}) and (\ref{A20c}) hold also for $u\circ g$ and if one replace norms $||\cdot||_{s}$, $||\cdot||_{s,\g}$ with $|\cdot|_{s}^{\infty}$, $|\cdot|_{s,\g}^{\infty}$.
\end{lemma}

\begin{lemma} \label{lem5} {\bf (Composition).} Assume that
for any $||u||_{s_{0}+\mu_{i},\g}\leq1$ the operator $\calQ_{i}(u)$ satisfies
\begin{equation}\label{A35}
||\calQ_{i}h||_{s,\g}\leq C(s)(||h||_{s+\tau_{i},\g}+||u||_{s+\mu_{i},\g}||h||_{s_{0}+\tau_{i}\g}),
\quad i=1,2.
\end{equation}
Let $\tau:=\max\{\tau_{1},\tau_{2}\}$, and $\mu:=\max\{\mu_{1},\mu_{2}\}$. Then,
for any
\begin{equation}\label{A36}
||u||_{s_{0}+\tau+\mu,\g}\leq1,
\end{equation}
one has that the composition operator $\calQ:=\calQ_{1}\circ \calQ_{2}$ satisfies
\begin{equation}\label{A37}
||\calQ h||_{s,\g}\leq C(s)(||h||_{s+\tau_{1}+\tau_{2},\g}+||u||_{s+\tau+\mu,\g}||h||_{s_{0}+\tau_{1}+\tau_{2},\g}).
\end{equation}
\end{lemma}

\begin{proof} It is sufficient to apply the estimates (\ref{A35}) to $\calQ_{1}$ first, then to $\calQ_{2}$ and 
using the condition (\ref{A36}).
\end{proof}

\smallskip
\noindent
{{\bf Proof of Lemma \ref{lemmaccio}.}} We first show that $T$ is symplectic. Consider $W=(w^{(1)},w^{(2)}),V=(v^{(1)},v^{(2)})\in H^{s}(\TTT^{d+1};\RRR)\times H^{s}(\TTT^{d+1};\RRR) $ and set $w=w^{(1)}+iw^{(2)}$, $v=v^{(1)}+iv^{(2)}$,
then one has
\begin{equation}\label{lemmaccio5}
\begin{aligned}
\tilde{\Omega}(TW,TV)&:=\int_{\TTT}\left(\begin{matrix}\frac{i}{\sqrt{2}} w \\ \frac{1}{\sqrt{2}} \bar{w}\end{matrix}\right)\cdot J\left(\begin{matrix}\frac{i}{\sqrt{2}} v \\ \frac{1}{\sqrt{2}} \bar{v}\end{matrix}\right)d x=
\int_{\TTT}W JVd x=:\tilde{\Omega}(W,V).
\end{aligned}
\end{equation}

To show the (\ref{lemmaccio3}) is sufficient to apply the definition of $T_{1}$.
First of all consider the linearized operator in some $z=(z^{(1)},z^{(2)})$
\begin{equation}\label{linop}
D_{z}\calF(\oo t, x, z)=D_{\oo}+\e D_{z}g(\oo t,x, z)=D_{\oo}+\e \del_{z_{0}}g+\e\del_{z_{1}}g \del_{x}+\e\del_{z_{2}}g\del_{xx}
\end{equation}
where $D_{\oo}$  and $g$ are defined in (\ref{totale}) and (\ref{13}) and
\begin{equation}\label{linop2}
\del_{z_{i}}g:=(a^{(i)}_{jk})_{j,k=1,2}:=(\del_{z^{(j)}_{i}}g_{k})_{j,k=1,2}.
\end{equation}
All the coefficients $a^{(i)}_{jk}$ are evaluated in $(z^{(1)},z^{(2)},z^{(1)}_{x},z^{(2)}_{x},z^{(1)}_{xx},z^{(2)}_{xx})$.
%
By using the definitions (\ref{linop}), (\ref{linop2}) and recalling that $g=(g_{1},g_{2})=(-f_{1},f_{2})$ and
$\ff=f_{1}+if_{2}$, one can check with an explicit computation that
$$
\calL(z)=T_{1}^{-1}T d_{z}\calF(\oo t,x,z)T^{-1}T_{1}
$$
has the desired form.
\EP

\end{document}